\documentclass[12pt]{amsart}

\usepackage{amsmath,amssymb,lscape}
\usepackage[displaymath, mathlines, pagewise]{lineno}
\usepackage[pagebackref,colorlinks,urlcolor=blue,citecolor=blue,linkcolor=blue]{hyperref}
\usepackage{enumerate}
\usepackage{bm}
\usepackage{color}

  \usepackage{graphicx}
  \usepackage{latexsym} 
  \usepackage[all]{xy}  
  \usepackage{amsfonts} 
  \usepackage[2emode]{psfrag}

\usepackage{pdflscape}

\newbox\pullbackbox
\setbox\pullbackbox=\hbox{\xy 0;<1mm,0mm>: \POS(2,0)\ar@{-} (0,2) \ar@{-} (4,2)
\endxy}
\def\pullback{\copy\pullbackbox}

\newbox\pushoutbox
\setbox\pushoutbox=\hbox{\xy 0;<1mm,0mm>: \POS(2,2)\ar@{-} (0,0) \ar@{-} (4,0)
\endxy}
\def\pushout{\copy\pushoutbox}

\newbox\pushoutboxh
\setbox\pushoutboxh=\hbox{\xy 0;<1mm,0mm>: \POS(2,2)\ar@{-} (2,0) \ar@{-} (0,2)
\endxy}

\newbox\pushoutboxv
\setbox\pushoutboxv=\hbox{\xy 0;<1mm,0mm>: \POS(0,2)\ar@{-} (2,2) \ar@{-} (0,0)
\endxy}
\def\pushoutv{\copy\pushoutboxv}

\newbox\pushoutboxw
\setbox\pushoutboxw=\hbox{\xy 0;<1mm,0mm>: \POS(0,0)\ar@{-} (2,0) \ar@{-} (0,2)
\endxy}
\def\pushoutw{\copy\pushoutboxw}

\newtheorem{proposition}{Proposition}[section]
\newtheorem{lemma}[proposition]{Lemma}
\newtheorem{corollary}[proposition]{Corollary}
\newtheorem{theorem}[proposition]{Theorem}

\theoremstyle{definition}
\newtheorem{definition}[proposition]{Definition}
\newtheorem{example}[proposition]{Example}
\newtheorem{examples}[proposition]{Examples}

\theoremstyle{remark}
\newtheorem{remark}[proposition]{Remark}
\newtheorem{remarks}[proposition]{Remarks}

\newcommand{\thlabel}[1]{\label{th:#1}}
\newcommand{\thref}[1]{Theorem~\ref{th:#1}}
\newcommand{\selabel}[1]{\label{se:#1}}
\newcommand{\seref}[1]{Section~\ref{se:#1}}
\newcommand{\lelabel}[1]{\label{le:#1}}
\newcommand{\leref}[1]{Lemma~\ref{le:#1}}
\newcommand{\prlabel}[1]{\label{pr:#1}}
\newcommand{\prref}[1]{Proposition~\ref{pr:#1}}
\newcommand{\colabel}[1]{\label{co:#1}}
\newcommand{\coref}[1]{Corollary~\ref{co:#1}}
\newcommand{\relabel}[1]{\label{re:#1}}
\newcommand{\reref}[1]{Remark~\ref{re:#1}}
\newcommand{\exlabel}[1]{\label{ex:#1}}
\newcommand{\exref}[1]{Example~\ref{ex:#1}}

\newcommand{\eqlabel}[1]{\label{eq:#1}}
\newcommand{\equref}[1]{(\ref{eq:#1})}

\newcommand{\Vect}{{\sf Vect}}

\newcommand{\coev}{{\rm coev}}
\newcommand{\ev}{{\rm ev}}

\newcommand{\tto}{\twoheadrightarrow}

\newcommand{\Hom}{{\sf Hom}}

\newcommand{\coker}{{\sf coker}\,}
\renewcommand{\ker}{{\sf ker}\,}
\newcommand{\im}{{\sf Im}\,}

\newcommand{\Mod}{{\sf Mod}}
\newcommand{\PMod}{{\sf PMod}}
\newcommand{\qPMod}{{\sf qPMod}}
\newcommand{\lPMod}{{\sf lPMod}}
\newcommand{\gPMod}{{\sf gPMod}}
\newcommand{\PCD}{{\sf PCD}}

\newcommand{\Span}{{\sf Span}}
\renewcommand{\span}{{\sf span}}
\newcommand{\Par}{{\sf Par}}
\newcommand{\pa}{{\sf par}}

\newcommand{\can}{{\ul{\sf can}}}

\def\ot{\otimes}
\def\bul{\bullet}

\newcommand{\bk}[1]{\left< #1\right>}

\def\AA{{\mathbb A}}

\def\CC{{\mathbb C}}

\def\NN{{\mathbb N}}

\def\ZZ{{\mathbb Z}}

\newcommand{\Bb}{\mathcal{B}}
\newcommand{\Cc}{\mathcal{C}}
\newcommand{\Dd}{\mathcal{D}}
\newcommand{\Ee}{\mathcal{E}}

\newcommand{\Mm}{\mathcal{M}}

\newcommand{\Oo}{\mathcal{O}}

\newcommand{\Zz}{\mathcal{Z}}

\def\text#1{{\rm {\rm #1}}}

\def\ol{\overline}
\def\ul{\underline}
\def\dul#1{\underline{\underline{#1}}}

\def\Set{\dul{\rm Set}}

\def\colim{{\rm colim\,}}

\def\lim{{\rm lim\,}}

\def\Alg{{\sf Alg}}

\begin{document}
\title{Geometrically Partial actions}
\author{Jiawei Hu} 
\address{Departement of Mathematical Sciences, East China Normal University, NO.500,Dongchuan Road, Shanghai, China}
\address{D\'epartement de Math\'ematiques, Facult\'e des sciences, Universit\'e Libre de Bruxelles, Boulevard du
Triomphe, B-1050 Bruxelles, Belgium}
\email{hhjjwilliam@gmail.com}
\author{Joost Vercruysse}
\address{D\'epartement de Math\'ematiques, Facult\'e des sciences, Universit\'e Libre de Bruxelles, Boulevard du
Triomphe, B-1050 Bruxelles, Belgium}
\email{jvercruy@ulb.ac.be}

\begin{abstract}
We introduce ``geometric'' partial comodules over coalgebras in monoidal categories, as an alternative notion to the notion of partial action and coaction of a Hopf algebra introduced by Caenepeel and Janssen.
The name is motivated by the fact that our new notion suits better if one wants to describe phenomena of partial actions in algebraic geometry. 
Under mild conditions, the category of geometric partial comodules is shown to be complete and cocomplete and the category of partial comodules over a Hopf algebra is lax monoidal. 
We develop a Hopf-Galois theory for geometric partial coactions to illustrate that our new notion might be a useful additional tool in Hopf algebra theory.
\end{abstract}

\maketitle

\tableofcontents

\section*{Introduction}

The coordinate algebras of algebraic groups provide classical examples of Hopf algebras and the interaction between Hopf algebra theory and algebraic geometry that arises from this construction has showed to be very fruitful for both worlds. 
With the rise of quantum groups in 80s of the the 20th century, deformations of Hopf algebras associated to algebraic groups have inspired the field of non-commutative (algebraic) geometry, where non-commutative algebras play the role of non-commutative spaces and (non-commutative, non-cocommutative) Hopf algebras coacting on these algebras play the role of symmetry groups of these spaces. 

The aim of the present paper is to introduce a new type of symmetries in (non-commutative) algebraic geometry, that correspond to partial group actions.

It is well-known that (usual) actions of a (discrete) group $G$ on a $k$-algebra $A$ are in correspondence with semi-direct product structures, or smash product structures, on $A\ot kG$. In order to describe certain algebras (such as Toeplitz algebras) as a generalized smash product, the notion of a partial group action was introduced in the setting of $C^*$-algebras by Exel \cite{Exel}. Roughly, a partial action of a group $G$ on an object $X$ associates to each element of $X$ an isomorphism between two appropriate subobjects of $X$. In case these subobjects always coincide with the whole object $X$, the action is a usual (or as we will call them from now on: global) group action. Immediate examples of these partial actions can be obtained by restricting a (global) action to an arbitrary subobject of $X$. Since its introduction, this notion of partial group action and the related notion of partial representation, has been investigated from a purely algebraic point of view and many interesting results have been obtained, see e.g.\ \cite{DE}, \cite{DEP}. 

A first attempt to bring partial actions from the setting of groups to the setting of Hopf algebras was made by Caenepeel and Janssen in \cite{CJ}. This approach has shown to be very successful in the sense that many classical Hopf-algebraic results appear to have a partial counterpart.

However, in this initial approach, several aspects of the theory remained unclear. For example, the definition of Caenepeel and Janssen only allowed to describe partial (co)actions of Hopf algebras on other (co)algebras. It was not possible to define partial actions on vector spaces nor to define partial actions of algebras other than Hopf algebras.
A next step was made in \cite{ABV}, where it was shown that, in analogy with classical actions of Hopf algebras, partial actions can be viewed as internal algebras in an appropriate monoidal category. However, in contrast to the classical case, the monoidal category in play is no longer the usual monoidal category of representations (or modules) of the Hopf algebra $H$, but rather the category of partial representations which coincides with the category of representations over a newly constructed Hopf algebroid $H_{par}$. Lately, it was shown in \cite{ABV2} how partial representations can be globalized and the partial 
representations of Sweedler's 4-dimensional Hopf algebra were completely classified.
A recent development in the theory of partial actions, is the approach of \cite{Timmerman}, where the initial theory of partial actions over $C^*$-algebras is merged with the Hopf-algebra setting, in the study of partial actions of $C^*$-quantum groups. 

Furthermore, 
it turns out that if one studies partial actions of Hopf algebras that arise from algebraic groups, the partial actions are not what one would expect. 
Indeed, it was observed in \cite{BV} that a partial coaction of a Hopf algebra $\Oo(G)$, which is the coordinate algebra of an algebraic group $G$, on an algebra $\Oo(X)$, which is the coordinate algebra of an algebraic space $X$, is always global unless $X$ is a disjoint union of non-empty subspaces. The spirit of partial actions would however also ask for more involved examples, where the elements of the algebraic group $G$ act as an isomorphism between arbitrary algebraic subspaces of $X$. Indeed, as we mentioned before, examples of partial actions can be constructed by restricting global actions. If the algebraic group $G$ acts (globally) on an algebraic variety $X$, we expect that the same group acts partially on arbitrary (not necessarily irreducible) subvarieties of $X$. For a more concrete example, one could consider two circles in the real plane intersecting in two points. From the global point of view, such a configuration has only few symmetries (or more precisely there are very few isometries of the plane that restrict to this union of circles). Nevertheless, each of the individual circles has a lot of symmetries. Partial actions allow to describe at once the (few) global symmetries of the pair of circles, and the (many) symmetries of the individual circles, as well as combinations of these.

To describe this kind of phenomena from a Hopf-algebraic point of view, we propose an alternative definition of partial (co)actions of Hopf algebras, that we call {\em geometric partial (co)actions} and that also allows us to bring partial action into the realm of non-commutative geometry as the algebraic structure to describe partial symmetries. 

To arrive at this goal, we will give in \seref{groups} a detailed study of partial actions of groups on sets, and provide a new approach to these. This approach is motivated by category theory, where partial morphisms have an interpretation as spans where one of the legs is a monomorphism. Given any category $\Cc$, one can build this way a bicategory of partial morphisms, which is a subbicategory of the category of spans over $\Cc$. A partial action of a group $G$ on an object $X$ is then noting else than a lax functor from $G$ into the endo-hom category of partial morphisms from $X$ to $X$. 

Based on this viewpoint, we generalize the notion of partial action of a group to partial (co)actions of (co)algebras in arbitrary categories with pullbacks (respectively pushouts). More precisely, given a coalgebra $H$ in the monoidal category $\Cc$, a {\em partial comodule datum} for $H$ is a quadruple $(X,X\bul H,\pi,\rho)$, where $\pi:X\ot H\to X\bul H$ is an epimorphism and $\rho:X\to X\bul H$ is a morhpism in $\Cc$. By considering 3 levels of strictness for the coassociativity condition on a given partial comodule datum, we consider then 3 versions of partial comodules: quasi, lax and geometric partial comodules. The name for the latter version is motivated by the fact that the above mentioned examples of partial actions of algebraic groups arise exactly as those `geometric partial comodules'. The initial partial actions of groups coincide with geometric partial actions of groups, viewed as coalgebras in the opposite of the category of sets. In case of arbitrary (Hopf) algebras, this new notion covers the one of Caenepeel and Janssen, but allows to go far beyond the notions of partial actions and partial representations as discussed above. Finally, our definition allows to consider partial (co)modules over arbitrary (co)algebras, where before it was only possible to consider such structures over Hopf algebras (with bijective antipode).

Although partial comodules are only a laxified version of classical comodules, they share surprisingly many properties with classical comodules. In particular, we show that a version of the fundamental theorem for comodules is still valid for geometric partial comodules and the category of partial comodules is complete and cocomplete. All this is shown in \seref{coalgebra}.

One of the key features of Hopf algebras, is that their categories of (co)modules have a natural monoidal structure, inherited by the monoidal structure of the base category wherein the considered Hopf algebra is defined. At this point, the theory of (geometric) partial modules becomes different from the global theory. Indeed, although the category of quasi partial comodules over a bialgebra is shown to posses an associative monoidal structure (with an oplax unit), the more interesting category of geometric partial comodules has only an {\em oplax} monoidal structure \cite{Leinster} (see \seref{bialgebra}). By definition an oplax monoidal structure requires the existence of $n$-fold tensor products, along with suitable coherence conditions. Where the tensor product of global comodules over a Hopf $k$-algebra is given by the tensor product of the underlying vector spaces, the vector space tensor product of two geometric partial comodules is in general no longer a geometric partial comodule. Therefore, their tensor product is defined as the biggest geometric quotient of the underlying vector space product.

Using this oplax monoidal structure, one can give meaning to an algebra in the category of geometric partial comodules. We discuss these `geometric partial comodule algebras' and initiate a Hopf-Galois theory for them in \seref{Galois}.

Some remarks on notation: given an object $X$ in a category $\Cc$, we denote the identity morphism on $X$ in $\Cc$ by $id_X$ or shortly by $X$. If $\Cc$ is a monoidal category with tensor product $\ot$ and monoidal unit $k$, then we denote the right unit constraint as $r_X:X\to X\ot k$ (classically, the reversed direction is used, but it turns out that this direction is more convenient for our work).

\section{A categorical reformulation of partial group actions}\selabel{groups}

\subsection{The classical definition of a partial group action}

Let $G$ be a group and $X$ a set. A {\em partial action datum} of $G$ on $X$ is a couple $(X_g,\alpha_g)_{g\in G}$, where
\begin{itemize}
\item $\{X_g\}_{g\in G}$, a family of subsets of $X$ indexed by the group $G$;
\item $\{\alpha_g:X_{g^{-1}}\to X\}_{g\in G}$ a family of maps indexed by the group $G$;
\end{itemize}

Recall from \cite{Exel} that a {\em partial action} $\alpha$ of $G$ on $X$ is a partial action datum $(X_g,\alpha_g)_{g\in G}$ that satisfies the following axioms
\begin{enumerate}[{({PA}1)}]
\item $X_e=X$ and $\alpha_e=id_X$, where $e$ denotes the unit of $G$;
\item $\alpha_g(X_{g^{-1}}\cap X_h)\subset X_g\cap X_{gh}$;
\item $\alpha_h\circ \alpha_g=\alpha_{hg}$ on $X_{g^{-1}}\cap X_{(hg)^{-1}}$.
\end{enumerate}

Remark that thanks to the second axiom (PA2), the third axiom (PA3) makes sense, since 
\begin{linenomath}
\[\alpha_h\circ \alpha_g(X_{g^{-1}}\cap X_{(hg)^{-1}})\subset \alpha_h(X_g\cap X_{h^{-1}})\subset X_{hg}\cap X_{h}\]
\end{linenomath}
and 
\begin{linenomath}
$$\alpha_{hg}(X_{g^{-1}}\cap X_{(hg)^{-1}})\subset X_h\cap X_{hg}.$$
\end{linenomath}
Furthermore, combining (PA2) and (PA3), we find that 
\begin{linenomath}$$X_g\cap X_{gh}=\alpha_g\circ \alpha_{g^{-1}}(X_g\cap X_{gh})\subset \alpha_g(X_{g^{-1}}\cap X_h)$$
\end{linenomath}
and therefore, we can deduce the stronger axiom 
\begin{itemize}
\item[(PA2')] $\alpha_g(X_{g^{-1}}\cap X_h)= X_g\cap X_{gh}$.
\end{itemize}
If we take in particular $h=e$, then we find that $\alpha_g(X_{g^{-1}})=X_g$. 
Moreover, since $\alpha_g\circ \alpha_{g^{-1}}(x)=x$ for all $x\in X_g$, we find that each map $\alpha_g$ induces a bijection $\alpha_g:X_{g^{-1}}\to X_g$. This last fact is often supposed as part of the definition of a partial action.

Many examples of partial actions have been observed in recent literature. It makes no sense to repeat them here, however, we will gave a few exemplary ones, which will make the transition to some of the new results in this paper more easy.

\begin{examples}\exlabel{paract}
\begin{enumerate}
\item Consider a (global) action of the group $G$ on a set $Y$, and let $X\subset Y$ be any (non-empty) subset of $Y$. Then $G$ acts partially on $X$, by defining $X_g:=\{x\in X~|~g^{-1}\cdot x\in X\}$ and defining $\alpha_g:X_{g^{-1}}\to X_g,\ \alpha_g(x)=g\cdot x$.
\item As a particular case of the previous one, we consider the following geometric example. Let $G=(\AA^2,+)$ be the group of $2$-dimensional affine translations. This group acts strictly transitive on the affine plane $\AA^2$. Consequently, this group acts partially on any subset of $\AA^2$. In one of the next sections, we will discuss in more detail the case of the partial action of this group on two intersecting lines.
\item Consider the additive group $\ZZ$. For any $z\in\ZZ$ with $z\ge 0$ we define its domain $X_{-z}=\NN$ and its action $\alpha_z:\ZZ\to \ZZ, x\mapsto x+z$. On the other hand for each $z<0$ we define its domain $X_{-z}=\{x\in \ZZ ~|~x\ge -z\}$ and its action $\alpha_z:\ZZ\to \ZZ, x\mapsto x+z$. Then one easily verifies this defines a partial action which is obtained by restricting the action of $\ZZ$ on itself to $\NN$.
\end{enumerate}
\end{examples}

\subsection{Lax and quasi partial actions}

As we explained, the axiom (PA2) in the definition of partial actions is designed to make sense of axiom (PA3) which expresses the associativity. However, this axiom can be weakened further.
\begin{definition}
Let $G$ be a group, $X$ a set and $\alpha=(X_g,\alpha_g)$ be a partial action datum. We say that $\alpha$ is a {\em lax} partial action of the following axioms hold
\begin{enumerate}[{({LPA}1)}]
\item $X_e=X$ and $\alpha_e=id_X$, where $e$ denotes the unit of $G$;
\item $X_{g^{-1}}\cap \alpha_g^{-1}(X_{h^{-1}})\subset X_{(hg)^{-1}}$.
\item $\alpha_h\circ \alpha_g=\alpha_{hg}$ on $X_{g^{-1}}\cap \alpha_g^{-1}(X_{h^{-1}})$.
\end{enumerate}
\end{definition}

Axiom (LPA2) tells that if $x\in X_{g^{-1}}$ and $\alpha(g)(x)\in X_{h^{-1}}$, then $x\in X_{(hg)^{-1}}$ and therefore axiom (LPA3) makes sense. As one can easily verify, any partial action is a lax partial action and the converse holds if and only if $\alpha_g(X_{g^{-1}})\subset X_g$ (see \prref{laxvsgeom}).
The following example shows that lax partial actions are a proper generalization of partial actions.

\begin{example}
This example is a variation of \exref{paract} (3). Consider the additive group $\ZZ$. For any $z\in\ZZ$ with $z\ge 0$ we define its domain $X_{-z}=\ZZ$ and its action $\alpha_z:\ZZ\to \ZZ, x\mapsto x+z$. On the other hand for each $z<0$ we define its domain $X_{-z}=\{x\in \ZZ ~|~x\ge -z\}$ and its action $\alpha_z:X_{-z}\to \ZZ, x\mapsto x+z$. Then one can verify that this is indeed a lax partial action and moreover it is not a partial action, since $\rho_{z}:X_{-z}\to X_z=\ZZ$ is not a bijection for any $z<0$. 
\end{example}

For sake of completeness, we also state another weakening of the definition of partial action, which is, by our opinion, naturally the most general version of a partial action.

\begin{definition}
Let $G$ be a group, $X$ a set and $\alpha=(X_g,\alpha_g)$ be a partial action datum. We say that $\alpha$ is a {\em quasi} partial action of the following axioms hold
\begin{enumerate}[{({QPA}1)}]
\item $X_e=X$ and $\alpha_e=id_X$, where $e$ denotes the unit of $G$;
\item $\alpha_h\circ \alpha_g=\alpha_{hg}$ on $X_{g^{-1}}\cap \alpha_g^{-1}(X_{h^{-1}})\cap X_{(gh)^{-1}}$.
\end{enumerate}
\end{definition}

Remark that in this definition, we ask associativity to hold exactly there where both $\alpha_h\circ \alpha_g$ and $\alpha_{hg}$ make sense. The following construction shows that quasi partial actions properly generalize lax and usual partial actions.

\begin{example}
Let $G$ be a group acting (globaly) on a set $X$. For any $g\in G$ consider an arbitrary subset $X_{g}\subset X$ and let $\alpha_g:X_{g^{-1}}\to X$ be the restriction of the action of $g$ on $X$. Then this defines a quasi partial action of $G$ on $X$.
\end{example}

\subsection{Partial actions and spans}

We will now reformulate the definition of a partial action, making no explicit reference to the elements of the set or the group, but stating everything internally in the category $\Set$ of sets. This way, the definition can be easily lifted to any (monoidal) category (with pullbacks). As we will show, quasi and partial actions arise naturally in this context.\\

Recall that in any category $\Cc$, a {\em span} from $X$ to $Y$ is a triple $(A,f,g)$, where $A$ is an object of $\Cc$ and $f:A\to X$ and $g:A\to Y$ are two morphisms of $\Cc$. If $\Cc$ has pullbacks and $(A,f,g), (B,h,k)$ are spans from $X$ to $Y$ and from $Y$ to $Z$ respectively, then one constructs a new span, called the {\em composition span}, from $X$ to $Z$ by the following pullback construction:
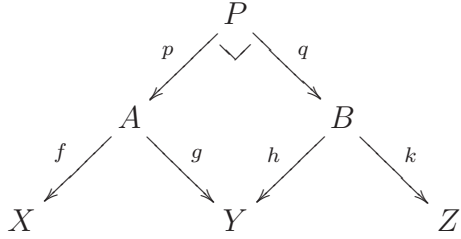
\begin{linenomath}$$\xymatrix{
&& P \ar[dl]_p \ar[dr]^q \ar@{}[d]|<<<{\pullback}\\
 & A \ar[dl]_f \ar[dr]^g && B \ar[dl]_h \ar[dr]^k\\
X && Y && Z
}$$
\end{linenomath}
which we will denote as $(B,h,k)\bul (A,f,g)$. 
Given two spans $(A,f,g), (B,h,k):X\to Y$, a {\em morphism of spans} $\alpha:(A,f,g)\to (B,h,k)$ is a map $\alpha:A\to B$ such that the following diagram commutes
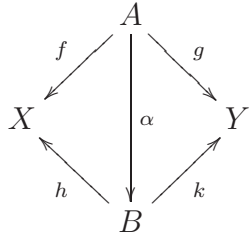
\begin{linenomath}$$
\xymatrix{
& A \ar[dl]_f \ar[dd]^\alpha \ar[dr]^g\\
X && Y\\
& B \ar[ul]^h \ar[ur]_k
}$$
\end{linenomath}
In this way, we obtain a bicategory $\Span(\Cc)$, whose $0$-cells are the objects of $\Cc$, $1$-cells are spans and $2$-cells are morphisms of spans. We can also consider the (usual) category $\span(\Cc)$, whose objects are the objects of $\Cc$ and whose morphisms are isomorphism classes of spans.

In what follows, we will use the following variation on the usual category of spans.
\begin{definition}
By a {\em partial morphism} from $X$ to $Y$ in a category $\Cc$, we mean a morphism $(A,f,g)$ in the category $\Span(\Cc)$, with the additional property that $f:A\to X$ is a monomorphism. 
By $\Par(\Cc)$ we denote the subbicategory of $\Span(\Cc)$, with the $0$-cells as $\Span(\Cc)$ (and $\Cc$), whose $1$-cells are given by partial morphisms in $\Cc$. By $\pa(\Cc)$ we denote the corresponding subcategory of $\span(\Cc)$.
\end{definition}

Remark that the above definition of $\Par(\Cc)$ makes sense since the pullback of a monomorphism is a monomorphism. Moreover, if $\alpha,\beta:(A,f,g)\to (B,h,k)$ are $2$-cells in $\Par(\Cc)$ then $\alpha=\beta$ since $h\circ \alpha=f=h\circ \beta$ and $h$ is a monomorphism. Hence $\Par(\Cc)$ is locally a poset. In the particular case of $\Par(\Set)$, a there is a morphism of spans $\alpha:(A,f,g)\to (B,h,k)$ if and only if $A$ is a subset of $B$ and $g$ is the restriction of $k$ to $A$.

We will denote a partial morphism from $X$ to $Y$ by a dotted arrow $\xymatrix{X\ar@{.>}[r] & Y}$. When we  consider a partial map as a triple $(A,f,g)$, we will often omit to write explicitly the first map $f$, as it is an inclusion and supposed to be known if we know the object $A$, ie. we will write $(A,f,g)=(A,g)=g$. \\

\begin{lemma}
Let $G$ be a group and $X$ a set. Then there is a bijective correspondence between
\begin{enumerate}[(i)]
\item partial action data of $G$ on $X$;
\item partial morphisms $G\times X\to X$;
\item maps $G\to \Par(X,X)$.
\end{enumerate}
\end{lemma}

\begin{proof}
$\ul{(i)\Leftrightarrow (ii)}$.
Let $(X_g,\alpha_g)_{g\in G}$ be a partial action datum of $G$ on a set $X$. Then we can construct the set
\begin{equation}\eqlabel{GbulX}
G\bul X=\{(g,x)~|~x\in X_{g^{-1}}\}\subset G\times X,
\end{equation}
which is the set of all ``compatible pairs'' in $G\times X$. 
Clearly, the partial action then induces a well-defined map $\alpha:G\bul X\to X, \alpha(g,x)=\alpha_g(x)$. Hence 
we obtain a partial morphism
$\xymatrix{G\times X\ar@{.>}[r] & X}$, 
\[
\xymatrix{
& G\bul X \ar@{_(->}[dl]_{\iota_X} \ar[dr]^\alpha \\
G\times X \ar@{.>}[rr] && X
}
\]
Conversely, consider any partial morphism $\alpha=(G\bul X,\iota,\alpha):G\times X\to X$, where $G\bul X$ is a subset of $G\times X$, $\iota:G\bul X\to G\times X$ is the canonical inclusion and $\alpha:G\bul X\to X$ is a map. Then for any $g\in G$, we can define $X_{g^{-1}}=\{x\in X~|~(g,x)\in G\bul X\}$ and we recover formula \equref{GbulX}.\\
$\ul{(i)\Leftrightarrow(iii)}$.
Let $(X_g,\alpha_g)_{g\in G}$ be a partial action datum, then for any $g\in G$ we have
\[\xymatrix{
& X_{g^{-1}} \ar@{_(->}[dl]_{\iota_g} \ar[dr]^{\alpha_g}\\
X \ar@{.>}[rr] && X
}\]
where $\iota_g:X_g\to X$ is the canonical inclusion, is a partial endomorphism of $X$ which defines a map $G\to \Par(X,X)$. Conversely, any map $G\to \Par(X,X)$ gives in the same way a family $(X_g,\alpha_g)_{g\in G}$, i.e. a partial action datum.
\end{proof}

The natural question that now arises is what are the conditions on a partial morphism $\alpha:G\times X\to X$ for the associated partial action datum to become an actual partial action. A first naive guess would be to impose the usual associativity and unitality conditions of an action expressed in the category $\pa(\Cc)$, or equivalently to impose that the map $G\to \pa(X,X)$ is a morphism of monoids (where the the later is the endomorphism monoid of $X$ in the (1-)category $\pa(\Set)$). However, as we will point out now, this leads to a global action. 

\begin{lemma}\lelabel{partialglobal}
Let $G$ be a group with multiplication $m:G\times G\to G,\ m(g,h)=gh$ and the unit $e:\{*\}\to G,e(*)=e$. Consider a partial action datum $(X_g,\alpha_g)_{g\in G}$. 
\begin{enumerate}
\item the following assertions are equivalent
\begin{enumerate}[(i)]
\item The partial action datum satisfies axiom (PA1);
\item The associated partial morphism $\alpha:G\times X\to X$ satisfies $\alpha\circ (e\times X) \simeq X$ in $\Par(\Set)$.
\item The associated map $\alpha':G\to \Par(X,X)$ preserves the unit.
\end{enumerate}
\item 
The following assertions are equivalent
\begin{enumerate}[(i)]
\item The partial action datum defines a global action of $G$ on $X$;
\item The associated partial morphism $\alpha:G\times X\to X$ satisfies the following identities  in $\pa(\Set)$ (i.e.\ isomorphism of spans)
\begin{eqnarray*}
\alpha\bul (e\times X) &\simeq& X\\
(G\times \alpha)\bul \alpha &\simeq& (m\times G)\bul\alpha
\end{eqnarray*}
\item The associated map $\alpha':G\to \Par(X,X)$ is a morphism of monoids.
\end{enumerate}
\end{enumerate}
\end{lemma}

\begin{proof}
\ul{(1)}. 
Let us compute the composition of spans $\alpha\bul (e\times X)$. This leads to the following pullback
\[
\xymatrix@!C{
&& \{*\}\bul X \ar@{_(->}[dl]_{\iota_X} \ar[dr]^{e\bul X} \ar@{}[d]|<<<{\pullback}\\
& X\cong \{*\}\times X \ar[dl]_\cong \ar[dr]^{e\times X} && G\bul X \ar@{_(->}[dl]_{\iota_X} \ar[dr]^\alpha\\
X \ar@{.>}[rr]&& G\times X \ar@{.>}[rr] && X 
}
\]
where $\{*\}\bul X=\{x\in X~|~x\in X_e\}\cong X_e$. Hence $\alpha\circ (e\times X)$ is the identity morphism on $X$ in the category $\Par(\Set)$, if and only if $X_e=X$ and $\alpha_e=id_X$, which is exactly axiom (PA1). Furthermore, it is clear that this is equivalent to saying that $\alpha'(e)=(X_e,\iota_e,\alpha_e)$ is the span $(X,id_X,id_X)$.\\ 
\ul{(2)}. By part (1), we only have to prove the equivalence of the associativity constraints.
Let us compute the composition of spans $(G\times \alpha)\bul \alpha$ in $\Par(\Set)$, which is given by the following pullback.
\[
\hspace{-1cm}
\xymatrix@!C{
&& G\bul(G\bul X) \ar[dr]^{G\bul \alpha} \ar@{_(->}[dl]  \ar@{}[d]|<<<{\pullback}\\
& G\times (G\bul X) \ar[dr]^{G\times \alpha} \ar@{_(->}[dl]_{G\times \iota} && G\bul X \ar[dr]^\alpha \ar@{_(->}[dl]^\iota  \\
G\times G\times X \ar@{.>}[rr] && G\times X \ar@{.>}[rr] && X
}
\]
Explicitly we find
\[ G\bul(G\bul X)=\{(h,g,x)\in G\times G\times X ~|~ x\in X_{g^{-1}}, gx\in X_{h^{-1}}\}.\]
Similarly, we can compute the composition $(m\times G)\bul\alpha$ in $\Par(\Set)$, which is again given by a pullback
\[
\xymatrix@!C{
&&(G\times G)\bul X\ar@{_(->}[dl]_{\iota'} \ar[dr]^{m\bul X} \ar@{}[d]|<<<{\pullback}\\
&G\times G\times X \ar@{=}[dl] \ar[dr]^{m\times X}&& G \bul X \ar@{_(->}[dl]_{\iota} \ar[dr]^\alpha\\
G\times G\times X \ar[rr] && G\times X \ar@{.>}[rr] && X
}
\]
where now
\begin{linenomath}
$$(G\times G)\bul X=\{(h,g,x)\in G\times G\times X~|~ x\in X_{(hg)^{-1}}\}.$$
\end{linenomath}

We then find that $(g,g^{-1},x)\in G\bul(G\bul X)$ if and only if $x\in X_g$ (and $g^{-1}x\in X_{g^{-1}}$). On the other hand, $(g,g^{-1},x)\in (G\times G)\bul X$ if and only if $x\in X_e=X$. Hence, we obtain that the action is global if and only if $G\bul(G\bul X)$ and $(G\times G)\bul X$ are isomorphic spans. 

In the same way, if the action is global, then clearly $\alpha'$ is a morphism of monoids. Conversely, if $\alpha'$ is a morphism of monoids then we obtain in particular that $\alpha'(g^{-1})\bul \alpha'(g)=\alpha'(e)= (X,id_X,id_X)$. Since the underlying set of the span of $\alpha'(g^{-1})\bul \alpha'(g)$ is given by $\{x\in X_{g^{-1}~|~gx\in X_g}$, we find that $\alpha'(g^{-1})\circ \alpha'(g)=\alpha'(e)$ implies that $X_e=X_g$ for all $g\in G$ and hence we have a global action.
\end{proof}

As we have just observed, partial actions are not just actions in the category of partial morphisms. The monoid morphism $G\to \pa(X,X)$ can also be viewed as a functor between $2$ one-object categories.
However, since the $\Par(\Set)$ is a bicategory, the $\Par(X,X)$ becomes a (monoidal) category, or a one-object bicategory. Consequently, there is a natural laxified version of a partial action considering only a {\em lax} functor between $G$ and $\Par(X,X)$. In the next proposition, we show that this coincides exactly with the lax partial actions we introduced above.

Recall that a lax functor $F:\Bb\to \Bb'$ between 2 bicategories consists of 
\begin{itemize}
\item a map from the $0$-cells of $\Bb$ to the $0$-cells of $\Bb'$, 
\item for any pair of $0$-cells $X,Y$ of $\Bb$, a functor $F_{X,Y}:\Bb(X,Y)\to \Bb'(X',Y')$
\item for any $0$-cell $X$ in $\Bb$ a $2$-cell $u_X:id_{FX}\to F(id_X)$'
\item for any two $1$-cells $a\in \Bb(X,Y)$ and $b\in\Bb(Y,Z)$ a $2$-cell $\alpha_{a,b}:F(a)\bul F(b)\to F(a \bul b)$ (where $\bul$ denotes the horizontal composition), which in natural in $a$ and $b$;
\end{itemize}
satisfying the usual coherence axioms. If the category $\Bb'$ is locally a poset, then these coherence conditions follow automatically from the above information.

\begin{proposition}\prlabel{laxpartial}
Let $G$ be a group with multiplication $m:G\times G\to G,\ m(g,h)=gh$ and the unit $e:\{*\}\to G,e(*)=e$. Consider a partial action datum $(X_g,\alpha_g)_{g\in G}$. 
The following assertions are equivalent
\begin{enumerate}[(i)]
\item The partial action datum defines a lax partial action of $G$ on $X$;
\item For the associated partial morphism $\alpha:G\times X\to X$, there exist morphisms of spans $u:X\to \alpha\bul (e\times X)$ and $\theta:(G\times \alpha)\bul \alpha \to (m\times G)\bul\alpha$;
\item The associated map $\alpha':G\to \Par(X,X)$ is a lax functor where $G$ is considered as a locally discrete $2$-category with one $0$-cell.
\end{enumerate}
\end{proposition}

\begin{proof}
$\ul{(i)\Leftrightarrow(ii)}.$ As we have shown in \leref{partialglobal}, $\alpha\bul (e\times X)$ is given by the span $(X_e,\iota_e,\alpha_e):X\to X$. There existence of a morphism of spans $u:X\to \alpha\bul (e\times X)$, means that $X\subset X_e\subset X$, hence $X=X_e$, and $\alpha_e=id_X$.

Furthermore, we also know from \leref{partialglobal} the explicit form of $(G\times \alpha)\bul \alpha$ and $(m\times G)\bul\alpha$. The existence of the morphism $\theta$ then means that $G\bul (G\bul X)\subset (G\times G)\bul X$, which is exactly axiom (LPA2) and the restriction of the partial action $\alpha_{gh}$ to $X_{g^{-1}}\cap \alpha^{-1}_g(X_{h^{-1}})$ coincides with $\alpha_h\circ \alpha_g$, which is exactly axiom (LPA3).\\
$\ul{(i)\Rightarrow(iii)}.$ Recall that the map $\alpha':G\to \Par(X,X)$ is given by $\alpha'(g)=(X_{g^{-1}},\iota_g,\alpha_g)$. Both $G$ and $\Par(X,X)$ are considered as one-object bicategories and moreover $G$ has only trivial $2$-cells, $\Par(X,X)$ is a poset. 
Hence, $\alpha':G\to \Par(X,X)$ induces a lax functor if and only if there exists morphism of spans $u':(X,id_X,id_X)\to (X_e,\iota_e,\alpha_e)$ and $\theta':\alpha'(h)\bul \alpha'(g)\to \alpha'(hg)$. 
As in the first part of the proof, the existence $u'$ is equivalent axiom (LPA1). Furthermore, remark that
$\alpha'(h)\bul \alpha'(g)$ is given by the span
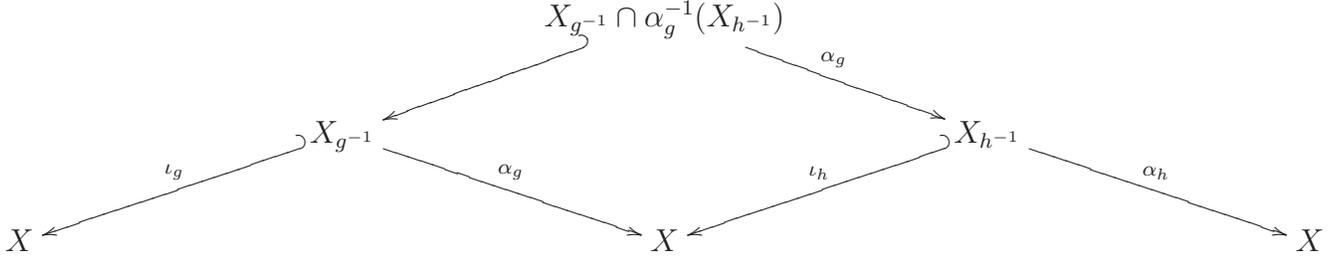
\[
\hspace{-1cm}
\xymatrix@!C{
&& X_{g^{-1}}\cap \alpha_g^{-1}(X_{h^{-1}}) \ar@{_(->}[dl] \ar[dr]^{\alpha_g}\\
& X_{g^{-1}} \ar@{_(->}[dl]_{\iota_g} \ar[dr]^{\alpha_g} &&  X_{h^{-1}} \ar@{_(->}[dl]_{\iota_h} \ar[dr]^{\alpha_h}\\
X && X && X
}
\]
Hence the existence of $\theta'$ means that axioms (PLA2) and (PLA3) hold.
\end{proof}

As we have pointed out before, partial actions are a special instance of lax partial actions. In the next result we provide equivalent conditions for a lax partial action to be partial.

Let us first make the following observation. Given a partial action datum, we can consider the pullback
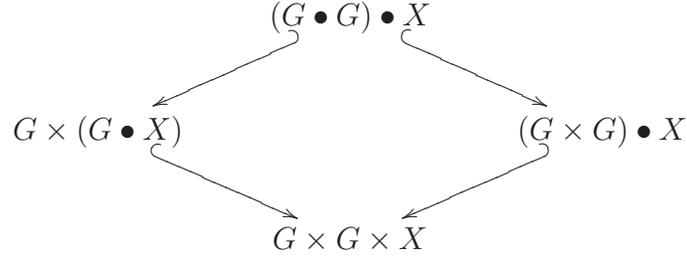
\[
\xymatrix@!C{
& (G\bul G)\bul X \ar@{_(->}[dl] \ar@{^(->}[dr]\\
G\times (G\bul X) \ar@{^(->}[dr] && (G\times G)\bul X \ar@{_(->}[dl]\\
& G\times G\times X
}
\]
which is nothing else than the intersection $G\times (G\bul X) \cap (G\times G)\bul X$. If the partial action datum defines a lax partial action, then existence of the morphism of spans $\theta:(G\times \alpha)\bul \alpha \to (m\times G)\bul\alpha$ implies that the following diagram commutes
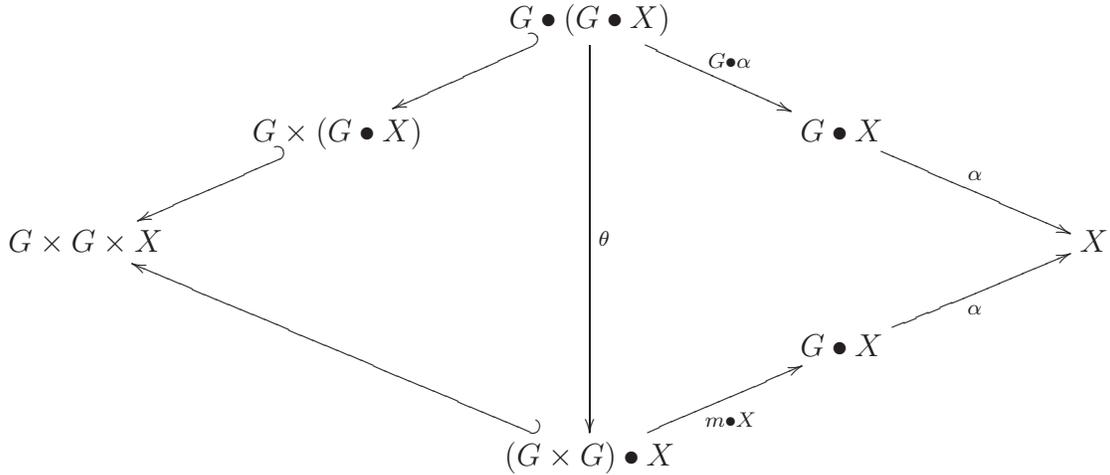
\[
\xymatrix@!C{
&& G\bul (G\bul X) \ar@{_(->}[dl] \ar[dr]^{G\bul \alpha} \ar[dddd]^{\theta}\\
& G\times(G\bul X)  \ar@{_(->}[dl] && G\bul X \ar[dr]^\alpha\\
G\times G\times X &&&& X\\
&&& G\bul X \ar[ur]_{\alpha}\\
&& (G\times G)\bul X \ar@{_(->}[uull] \ar[ur]_{m\bul X}
}
\]
and therefore the image of $\theta$ lies in $(G\bul G)\bul X$, i.e. we can corestrict $\theta$ to a morphism
$\bar\theta: G\bul (G\bul X)\to (G\bul G)\bul X$.

\begin{proposition}\prlabel{laxvsgeom}
Let $\alpha$ be a lax partial action of the group $G$ on the set $X$. Then the following statements are equivalent
\begin{enumerate}[(i)]
\item $\alpha$ is a partial action;
\item $\bar\theta:G\bul (G\bul X)\to (G\bul G)\bul X$ is an isomorphism;
\item for each $g\in G$, we have that $\alpha_g:X_{g^{-1}}\to X_g$.
\end{enumerate}
\end{proposition}

\begin{proof}
$\ul{(ii)\Leftrightarrow(i)\Rightarrow(iii)}$.
By definition, partial actions and lax partial actions only differ in their second axiom. 
From the above discussion, we know that 
\begin{eqnarray*}
(G\bul G)\bul X&=&G\times (G\bul X) \cap (G\times G)\bul X\\
&=& \{(h,g,x)\in G\times G\times X~|~x\in X_{g^{-1}} \cap X_{(gh)^{-1}} \}
\end{eqnarray*}
Therefore, $\bar\theta$ is an isomorphism we obtain that if
\[(h^{-1}g^{-1},g,x)\in (G\bul G)\bul X,\ \text{i.e.} \quad x\in X_{g^{-1}}\cap X_{h}\]
then also
\[((gh)^{-1},g,x)\in G\bul (G\bul X),\ \text{i.e.}\quad
x\in X_{g^{-1}}\ \text{and}\ gx\in X_{hg}\]
Hence we find that $\alpha_g(X_{g^{-1}}\cap X_h)\subset X_{gh}$. 
In particular, taking $h=e$, then we find that  
$\alpha_g(X_{g^{-1}})\subset X_g$. Combining both, we recover exactly axiom (PA2).

$\ul{(iii)\Rightarrow(i)}$. For any $g\in G$ and $x\in X_{g}$ we find that $g^{-1}x\in X_{g^{-1}}$, we can apply $g$ on $g^{-1}x$ and find that $x=g\cdot g^{-1}x$. So $x\in X_g$ if and only if $x=gy$ for some $y\in X_{g^{-1}}$. 

Now take any $x\in X_{g^{-1}}\cap X_h$. Then by the above, we can write $x=hy$ for some $y\in X_{h^{-1}}$. Since we have that $y\in X_{h^{-1}}$ and $x=hy\in X_{g^{-1}}$, it follows by axiom (LPA2) that $y\in X_{(gh)^{-1}}$ and $gh\cdot y=g\cdot (hy)=gx$. In particular, we find that $gx=gh\cdot y\in X_{gh}$. Hence we obtain exactly axiom (PA2).
\end{proof}

Finally, we also restate the definition of quasi partial action in terms of spans, the proof of which is clear. 

\begin{proposition}
A partial action datum $(X_g,\alpha_g)$ defines a quasi partial action of $G$ on $X$ if and only if 
the equivalent statements of \leref{partialglobal} (1) hold and the associativity constraint $\alpha_h\circ\alpha_g=\alpha_{hg}$ holds on all elements of the following pullback
\[
\xymatrix@!C{
& \Theta \ar[dl]_{\theta_1} \ar[dr]^{\theta_2}  \\
G\bul (G\bul X) \ar@{^(->}[dr] && (G\bul G)\bul X \ar@{_(->}[dl]\\
& G\times G\times X
}
\]
Consequently, a quasi partial action is lax if and only if the span $(\Theta,\theta_1,\theta_2):G\bul (G\bul X)\to (G\bul G)\bul X$ is induced by a morphism, and the quasi partial action is a partial action if and only if $\Theta$ is an isomorphism.
\end{proposition}

\begin{proof}
Let us just remark that the associativity on $\Theta$ means that the following diagram commutes
\[\xymatrix@!C{
& \Theta \ar[dl]_{\theta_1} \ar[dr]^{\theta_2}  \\
G\bul (G\bul X)  \ar[d]_{G\bul \alpha} && (G\bul G)\bul X  \ar[d]^{m\bul X}\\
G\bul X \ar[d]_{\alpha} && G\bul X \ar[d]_{\alpha} \\
X \ar@{=}[rr] && X
}\]
\end{proof}

\section{Partial comodules over a coalgebra}\selabel{coalgebra}

Let $\Cc$ be a braided monoidal category with pullbacks that are preserved by all endofunctors on $\Cc$ of the form $-\ot X$ and $X\ot-$. Then the observations from the previous section allow us to define partial actions of a Hopf algebra in $\Cc$ on any object in $\Cc$ such that taking $\Cc=\Set$ we recover the classical definition of partial actions of groups on sets. Since we will rather be interested in examples inspired by algebraic geometry, hence in coactions rather than actions, we will take a dual point of view and consider from now on a braided monoidal category $\Cc$ with pushouts that are preserved by all endofunctors of the form $-\ot X$ and $X\ot -$, and Hopf algebras mentioned below are Hopf algebras in $\Cc$. Remark that in such a category, the tensor product of two epimorphisms is an epimorphism. 
Since pushouts are colimits, any braided closed monoidal category will serve as an example, in particular any category of modules over a commutative ring $k$. In what follows the latter will be our standard example, were we in fact mostly will restrict to the case where $k$ is a field.
Inspired by this example we will denote the unit of the monoidal category $\Cc$ by $k$.

\subsection{Geometrically partial comodules}

In \cite{ABV}, the notion of a ``partial module'' over a Hopf algebra $H$ was introduced, by means of partial representations of Hopf algebras and similarly, ``partial comodules'' can be introduced by means to partial corepresentations, see \cite{coreps}. In this section, we introduce alternative notions of partial (co)module over any (co)algebra. To prevent a clash of terminologies in case $C=H$, we will call our notions (in rising order of generality) a {\em quasi}, {\em lax} and {\em geometric} partial (co)modules. 
We show that in the case of Hopf algebras, the partial modules of \cite{ABV}, and in particular, the partial actions of \cite{CJ}, appear as special cases of our quasi partial comodules. Examples arising from (usual) partial actions of (algebraic) groups on (algebraic) sets give rise to geometric partial comodules.

\begin{definition}
Let $(H,\Delta,\epsilon)$ be a coalgebra in a monoidal category $\Cc$. 
A {\em partial comodule datum} is a quadruple $\ul X=(X,X\bul H,\pi_X,\rho_X)$, where $X$ and $X\bul H$ are objects in $\Cc$, $\pi_X:X\ot H\tto X\bul H$ is an epimorphism and $\rho_X:X\to X\bul H$ is a morphism in $\Cc$.
\end{definition}

Remark that a partial comodule datum can be viewed as the following cospan in $\Cc$ 
\[
\xymatrix@!C{
X \ar@{.>}[rr] \ar[dr]^{\rho_X} && X\ot H  \ar@{->>}[dl]_{\pi_X} \\
& X\bul H
}\]
Suppose now that the category $\Cc$ has pushouts. Then to any partial comodule datum induces canonically four pushouts, that we denote by $X\bul k$, $(X\bul H)\bul H$, $X\bul(H\ot H)$ and $(X\bul H)\bul H$, and that are defined respectively by the following diagrams:
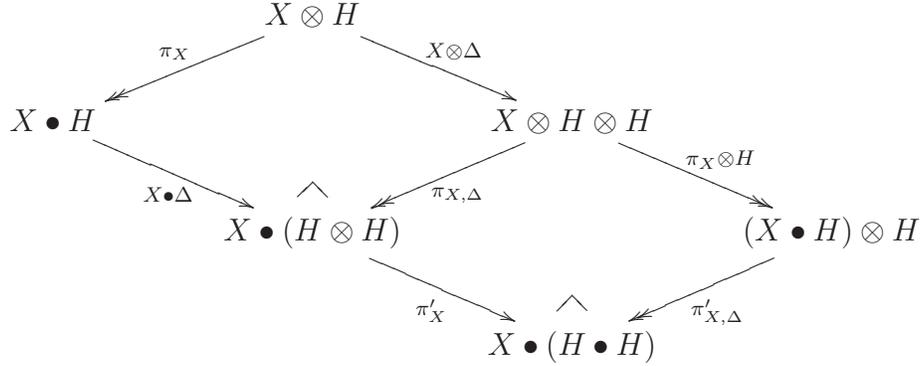
\[
\xymatrix@!C{
& X\ot H  \ar@{->>}[dl]_{\pi_X}\ar[dr]^{X\ot\epsilon}  \\
 X\bul H \ar[dr]_{X\bul \epsilon} && X\ot k \ar@{->>}[dl]^{\pi_{X,\epsilon}} \\
& X\bul k  \ar@{}[u]|<<<{\pushout}
}\quad
\xymatrix@!C{
& X\ot H \ar@{->>}[dl]_-{\pi_X} \ar[dr]^-{\rho_X\ot H}\\
X\bul H 
\ar[dr]_{\rho_X\bul H}  && (X\bul H)\ot H \ar@{->>}[dl]^{\pi_{X\bul H}}\\
&(X\bul H)\bul H\ar@{}[u]|<<<{\pushout}
}
\]
\[
\xymatrix@!C{
& X\ot H \ar@{->>}[dl]_-{\pi_X} \ar[dr]^-{X\ot \Delta}\\
X\bul H \ar[dr]_{X\bul \Delta} && X\ot H\ot H\ar@{->>}[dl]^{\pi_{X,\Delta}} \ar@{->>}[dr]^{\pi_X\ot H}\\
& X\bul(H\ot H) \ar[dr]_{\pi'_{X}} \ar@{}[u]|<<<{\pushout}  && (X\bul H)\ot H \ar@{->>}[dl]^{\pi'_{X,\Delta}} \\
&&X\bul(H\bul H) \ar@{}[u]|<<<{\pushout}
}
\]
Finally, we consider a last pushout that we denote as $\Theta$ and that is given by the following diagram
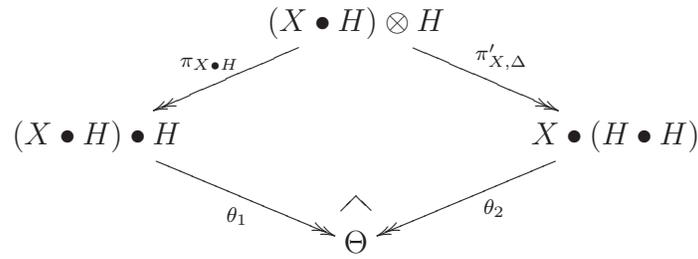
\[
\xymatrix@!C{
& (X\bul H)\ot H \ar@{->>}[dl]_-{\pi_{X\bul H}} \ar@{->>}[dr]^-{\pi'_{X,\Delta}}\\
(X\bul H)\bul H \ar@{->>}[dr]_-{\theta_1} && 
X\bul (H\bul H) \ar@{->>}[dl]^{\theta_2}\\
&\Theta \ar@{}[u]|<<<{\pushout}
}
\]
We will call $\Theta$ the {\em coassociativity pushout}.

We are now ready to state the exact definitions of a partial comodule.

\begin{definition}
Let $(H,\Delta,\epsilon)$ be a coalgebra in a monoidal category with pushouts $\Cc$. A {\em quasi partial comodule} is a 
partial comodule datum $(X,X\bul H,\pi_X,\rho_X)$
that satisfies the following conditions
\begin{enumerate}[{[QPC1]}]
\item $(X\bul\epsilon)\circ \rho_X=\pi_{X,\epsilon}\circ r_X:X\to X\bul k$ are identical isomorphisms. I.e.\ the following diagram commutes
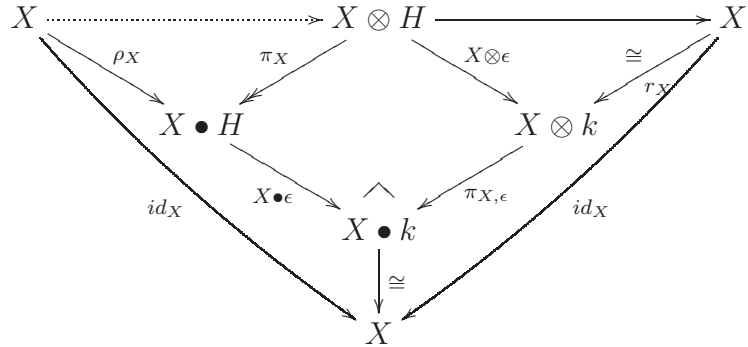
\[
\xymatrix@!C{
X \ar@{.>}[rr] \ar@/_/[dddrr]_{id_X} \ar[dr]^{\rho_X} && X\ot H  \ar@{->>}[dl]_{\pi_X}\ar[dr]^{X\ot\epsilon} \ar[rr] && X \ar[dl]^{r_X}_\cong \ar@/^/[dddll]^{id_X} \\
& X\bul H \ar[dr]_{X\bul\epsilon} && X\ot k \ar[dl]^{\pi_{X,\epsilon}} \\
&& X\bul k \ar[d]^{\cong} \ar@{}[u]|<<<{\pushout}\\
&& X
}
\]
\item $\theta_1\circ(\rho_X\bul H)\circ\rho_X=\theta_2\circ \pi'_X\circ (X\bul\Delta)\circ \rho_X$, i.e. the following diagram commutes
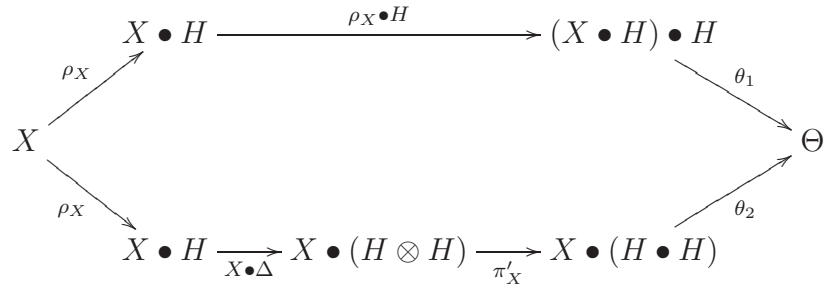
\[
\xymatrix{
&X\bul H \ar[rr]^-{\rho_X\bul H} && (X\bul H)\bul H \ar[dr]^-{\theta_1}\\
X \ar[ur]^-{\rho_X} \ar[dr]_-{\rho_X} &&&& \Theta\\
&X\bul H \ar[r]_-{X\bul \Delta} & X\bul(H\ot H) \ar[r]_-{\pi'_X} & X\bul(H\bul H) \ar[ur]_-{\theta_2}
}
\]
\end{enumerate}
A quasi partial comodule will be called a {\em lax partial comodule} when the cospan $\Theta_X:X\bul(H\bul H)\dashrightarrow (X\bul H)\bul H$ is induced by a morphism $\theta$ (in $\Cc$). Furthermore a lax partial comodule is called a {\em geometric partial comodule} if $\theta$ is an isomorphism.
\end{definition}

\begin{remarks}\relabel{ParSweedler}
\begin{enumerate}
\item As we have remarked before, partial morphisms and hence the structure of partial comodules, have a typical bicategorical behaviour. For example, both the unitality and coassociativity axiom of a (geometric) partial comodule consist of 2 parts. Firstly, it concerns an (iso)morphism of certain pushouts (to be precise, $X\bul k\cong X$ for unitality and $(X\bul H)\bul H \cong X\bul (H\bul H)$ for coassociativity of geometric partial comodules), and then an identity between certain maps, involving and uniquely characterizing the previous isomorphisms. An useful consequence of the unitality axiom is therefore that $(X\bul \epsilon)\circ \pi_X \simeq X\ot \epsilon$, where the symbol ``$\simeq$'' means that both maps are identical up to the (canonical) isomorphism $X\bul k\cong X$.
\item Remark that by uniqueness of colimits, the pushout $\Theta_X=(\Theta,\theta_1,\theta_2)$ is unique up to isomorphism and hence is not part of the structure of a quasi partial comodule. Similarly, if $\Theta_X$ is induced by a morphism $\theta_X$, then this morphism is uniquely determined by its property 
$\theta_X\circ \pi_{X\bul H} = \pi'_{X,\Delta}$ since $\pi_{X\bul H}$ is an epimorphism. Also, whenever there exists a morphism $\theta_X$ with this property, then $\Theta\cong (X\bul H)\bul H$.
\item We will often denote a  (quasi) partial comodule by $(X,\pi_X,\rho_X)$ or just by $X$. 
\item Of course, one can make state dual definitions of a {\em quasi}, {\em lax} and {\em geometric partial module} over an algebra. We leave the details to the reader, it suffices to apply the above definition to the opposite category $\Cc^{op}$. 
\item When working in the base category $\Cc=\Mm_k$ We will sometimes use the following Sweedler notation for quasi partial comodules. For any $x\ot h\in X\ot H$, we write $\pi_X(x\ot h)=x\bul h\in X\bul H$. Remark that $X\bul H$ is no longer a tensor product (see below for an interpretation of $X\bul H$ as a monoidal product when $H$ is a bialgebra). Hence $x\bul h$ represents a certain class of tensors in $X\ot H$ and by the surjectivity of $\pi_X$, any element of $X\bul H$ can be represented in such a way, although non-uniquely. We then write $\rho_X(x)=x_{[0]}\bul x_{[1]}$, which means that there exists an element $x_{[0]}\ot x_{[1]}\in X\ot H$ such that $\rho_X(x)=\pi_X(x_{[0]}\ot x_{[1]})$. Again, the element $x_{[0]}\ot x_{[1]}\in X\ot H$ is not unique, so some care is needed in this notation. However, the class $x_{[0]}\bul x_{[1]}\in X\bul H$ is well-defined since $\rho_X$ is a proper map. Axiom [QPC1] tells us then that, as for usual coactions, $x_{[0]}\epsilon(x_{[1]})=x$ for all $x\in X$, and in particular this expression makes sense. We will treat axiom [QPC2] in a similar way by the expression
\[x_{[0][0]}\bul x_{[0][1]}\bul x_{[1]} = x_{[0]}\bul x_{[1](1)}\bul x_{[1](2)}\]
However, this expression now holds in the pushout $\Theta$, and by definition, 
the left hand side in the above expression is the notation for $\theta_1\circ(\rho_X\bul H)\circ\rho_X(x)$ and the right hand side is $\theta_2\circ \pi'_X\circ (X\bul\Delta)\circ \rho_X(x)$, for the same $x\in X$. 
\end{enumerate}
\end{remarks}

A first class of examples is obtained from the results of the previous section by taking $\Cc=\Set^{op}$. Indeed, quasi, lax and (usual) partial actions of a group coincide in this way with quasi, lax and geometric partial (co)modules. Remark that in the above formation, these notions also allow to consider partial actions of arbitrary monoids rather than groups.

Before we give some more examples, let us first state the following (well-known) lemma that will be useful for our purposes.

\begin{lemma}\lelabel{pushoutvect}
Consider vector spaces $U$, $V$, $W$ and linear maps $f:U\to V$, $g:U\to W$, where $g$ is surjective. Then the pushout of the pair $(f,g)$ is given by $P=V/f(\ker g)$, where $\ol g:V\to P$ is the canonical surjection and $\ol f:W\to P$ is given by $\ol f(w)=f(u)+f(\ker g)$, where $u$ is any element of $U$ such that $g(u)=w\in W$. 
\end{lemma}

Let us now provide some examples.

\begin{example}[Quotient of a global comodule] \exlabel{quotient}
Consider a global $H$-comodule $X$ with coaction $\rho:X\to X\ot H$ and any epimorphism $\pi:X\to Y$ in $\Cc$. Then we can define a partial comodule datum over $Y$ by taking the pushout of the pair $(\pi, (\pi\ot H)\circ \rho)$
\begin{equation}\eqlabel{quotient}
\xymatrix@!C{
&& X\ar[ddll]_{\pi} \ar[dr]^{\rho}\\
&&& X\ot H \ar[dr]^-{\pi\ot id}\\
Y \ar[drr]_{\rho_Y} &&&& Y\ot H \ar[dll]^{\pi_Y}\\
&& Y\bul H 
\ar@{}[u]|<<<{\pushout}
}
\end{equation}
Consider the following diagram.\\
\resizebox{\textwidth}{!}{
$$\xymatrix@!C{
X\ar@{->>}[dd]_\pi
\ar[rr]^\rho&&X\ot H\ar@{->>}[dd]^{\pi\ot H}
\ar[rr]^{\rho \ot H}&&X\ot H\ot H\ar@{->>}[dd]^{\pi\ot H\ot H}&&X\ot H\ar@{->>}[dd]^{\pi\ot H}\ar[ll]_{X\ot\Delta}&&X\ar[ll]_{\rho}
\ar@{->>}[dd]_{\pi}\\
\\
Y\ar@{.>}[rr]\ar[dr]&&Y\ot H\ar@{.>}[rr]\ar@{->>}[dl]\ar[dr]^{\rho_Y\ot H}&&Y\ot H\ot H\ar@{->>}[dl]^{\pi_Y\ot H}\ar@{=}[dr]&&Y\ot H\ar[ll]\ar[dl]\ar@{->>}[dr]&&Y\ar@{.>}[ll]\ar[dl]\\
&Y\bul H\ar@{}[u]|<<<{\pushout}\ar[dr]&&(Y\bul H)\ot H\ar@{}[u]|<<<{\pushout}\ar@{->>}[dl]\ar@{=}[dr]&&Y\ot H\ot H\ar@{}[u]|<<<{\pushout}\ar@{->>}[dl]\ar@{->>}[dr]&&Y\bul H\ar@{}[u]|<<<{\pushout}\ar[dl]&\\
&&(Y\bul H)\bul H\ar@{}[u]|<<<{\pushout}\ar@{=}[dr]&&(Y\bul H)\ot H\ar@{->>}[dl]\ar@{->>}[dr]\ar@{}[u]|<<<{\pushout}&&Y\bul (H\ot H)\ar@{->>}[dl]\ar@{}[u]|<<<{\pushout}&&\\
&&&(Y\bul H)\bul H \ar@{}[u]|<<<{\pushout} &&Y\bul (H\bul H)\ar@{}[u]|<<<{\pushout}&&&} 
$$
}
By composing pushouts in the diagram, we see that $(Y\bul H)\bul H$ is the pushout of the pair $(\pi,(\rho_Y\ot H)\circ (\pi\ot H)\circ \rho)$. Moreover, diagram chasing and the coassociativity of $(X,\rho)$ tells us that
\begin{eqnarray*}
(\rho_Y\ot H)\circ (\pi\ot H)\circ \rho &=& (\pi_Y\ot H)\circ (\pi\ot H\ot H)\circ (\rho\ot H)\circ \rho\\
&=& (\pi_Y\ot H)\circ (\pi\ot H\ot H)\circ (X\ot \Delta)\circ \rho\\
&=& (\pi_Y\ot H)\circ (Y\ot \Delta)\circ (\pi\ot H)\circ \rho
\end{eqnarray*}
And hence $(Y\bul H)\bul H$ is has to be isomorphic to $Y\bul(H\bul H)$, which is exactly the pushout of 
$(\pi,(\pi_Y\ot H)\circ (Y\ot \Delta)\circ (\pi\ot H)\circ \rho)$. 
We can conclude that $(Y,\rho_Y,\pi_Y)$ is a geometric partial comodule. 

Performing this construction in $\Cc=\Set^{op}$, we recover \exref{paract} (1)
\end{example}

\begin{example}[Quotient of a partial comodule]\exlabel{quotientpartial}
The previous example can be generalized in the following way. Let $(X,X\bul H,\pi_X,\rho_X)$ be a partial $H$-comodule datum, and $p:X\to Y$ an epimorphism. Then consider the pushout $P$ of the pair $(\pi_X,p\ot H)$:
\[
\xymatrix{
& X\ot H \ar[dl]_{\pi_X} \ar[dr]^{p\ot H}\\
X\bul H \ar[dr]_{p_1} && Y\ot H \ar[dl]^{p_2}\\
& P
}
\]
Moreover, we can then define a partial comodule datum $(Y,Y\bul H,\pi_Y,\rho_Y)$ by considering the following pushout
\[
\xymatrix@!C{
&& X\ar[ddll]_{p} \ar[dr]^{\rho_X}\\
&&& X\ot H \ar[dr]^-{p_1}\\
Y \ar[drr]_{\rho_Y} &&&& P \ar[dll]^{\pi'_Y}\\
&& Y\bul H 
\ar@{}[u]|<<<{\pushout}
}
\]
and taking $\pi_Y=\pi'_Y\circ p_2$. Similar to the previous example, one can show that $Y$ is a quasi or geometric partial comodule if $X$ is so.
\end{example}

\begin{example}[Partial action in the affine plane]\exlabel{affine}
Since the affine group $(\AA^2,+)$ acts strictly transitive on the affine plane, the algebra $A=k[x,y]$ is a Galois object over the bialgebra $H=k[x,y]$. In particular, $A$ is an $H$-comodule with coaction $\rho:k[x,y]\to k[x,y]\ot k[x,y]\cong k[x,y,x',y'],\ \rho(f)(x,y,x',y')=f(x+x',y+y')$ where $f\in k[x,y]$. Considering the quotient $B=k[x,y]/(xy)$ we find by the previous example that $B$ is a partial $H$-comodule with 
$B\bul H=k[x,y,x',y']/\rho((xy))$. 
Remark that $\rho((xy))$ is not an ideal in $k[x,y,x',y']$, hence $B\bul H$ is not an algebra quotient of $B\ot H$.
Furthermore, $\rho_B:B\to B\bul H$ given by $\rho_B(\ol f)(x,y,x',y')=\ol f(x+x',y+y')$ for all $\ol f\in B$.

As we have remarked in the introduction, it follows from the results of \cite{BV} that this example cannot be described by means of partial actions in the sense of Caenepeel-Janssen (see \exref{CJ}).
\end{example}

\begin{example}[A partial action on the quantum plane]\exlabel{quantum}
By a similar construction as in the previous example, we obtain a partial action on the quantum plane.
Consider the tensor algebra $T(V)$ where $V$ is a $2$-dimensional vector space. Then this tensor algebra is known to be a Hopf algebra and it coacts on itself by the comultiplication. We can view $T(V)$ as the free algebra $k\bk{x,y}$ with two generators $x,y$ and the coaction is then given by the comultiplication $\Delta:k\bk{x,y}\to k\bk{x,y}\ot k\bk{x,y},\ \Delta(x)=x\ot 1+1\ot x, \Delta(y)=y\ot 1+1\ot y$. Now consider the quantum plane $k_q[x,y]=k\bk{x,y}/(xy-qyx)$. By \exref{quotient}, the quantum plane is a partial comodule over the tensor algebra. 
\end{example}

\begin{example}\exlabel{CJ}
Consider a partial coaction in the sense of Caenepeel-Janssen \cite{CJ}. This means that $H$ is a Hopf algebra, $A$ is an algebra and
\[\rho:A\to A\ot H, \rho(a)=a_{[0]}\ot a_{[1]}\]
is a linear map satisfying the following axioms:
\begin{enumerate}[{(CJ1)}]
\item $(ab)_{[0]}\ot (ab)_{[1]}=a_{[0]}b_{[0]}\ot a_{[1]}b_{[1]}$
\item $a_{[0][0]}\ot a_{[0][1]}\ot a_{[1]}= a_{[0]}1_{[0]}\ot a_{[1](1)}1_{[1]} \ot a_{[1](2)}$
\item $a_{[0]}\epsilon(a_{[1]})=a$
\end{enumerate}

Then we define $e=1_{[0]}\ot 1_{[1]}\in A\ot H$, which is an idempotent, because of the first axiom. Then we get that 
\[A\ot H=(A\ot H)e\oplus (A\ot H)e'\]
where $e'=1-e$.
If we put $A\bul H=(A\ot H)e$, then we have that the map
\[\pi:A\ot H\to A\bul H, a\ot h\mapsto a1_{[0]}\ot h1_{[1]}\]
is surjective with right inverse the inclusion map and kernel $(A\ot H)e'=\{a\ot h-a1_{[0]}\ot h1_{[1]}~|~a\ot h\in A\ot H\}=N$ and $A\bul H=(A\ot H)/N$. 

This allows us to define the partial action datum over $A$:
\[\xymatrix{A\ar[dr]_-{\ol\rho=\pi\circ\rho} \ar[rr]^\rho && A\ot H \ar[dl]^-\pi\\ 
& A\bul H
}\]
To check the coassociativity, we consider the diagram
\[
\resizebox{1.1\textwidth}{!}{
\xymatrix@!C{
A \ar[rr]^\rho \ar[dr]_{\ol\rho} && A\ot H \ar[rr]^{\rho\ot H} \ar@{->>}[dl]^{\pi}\ar[dr]_{\ol\rho\ot H} && 
A\ot H\ot H \ar@{->>}[dl]^{\pi\ot H} \ar@{=}[dr] && A\ot H\ar[ll]^{A\ot \Delta}\ar[dl]^{A\ot \Delta} \ar@{->>}[dr]_{\pi} && A \ar@{.>}@{->}[dl]^{\ol \rho} \ar[ll]_\rho\\
& A\bul H\ar[dr]_{\ol\rho\bul H} && (A\bul H)\ot H\ar@{->>}[dl]^{\pi_{A\bul H}} \ar@{=}[dr] && A\ot H\ot H\ar@{->>}[dl]^{\pi\ot H}\ar@{->>}[dr]_{\pi_{A,\Delta}} && A\bul H\ar[dl]^{A\bul \Delta} & &\\
&&(A\bul H)\bul H\ar@{=}[dr] \ar@{}[u]|<<<{\pushout}&&(A\bul H)\ot H\ar@{->>}[dl]^{\pi_{A\bul H}}\ar@{->>}[dr]_{\pi'_{A,\Delta}} \ar@{}[u]|<<<{\pushout}&&A\bul (H\ot H)\ar@{->>}[dl]^{\pi'_{X}} \ar@{}[u]|<<<{\pushout}&&\\
&&&(A\bul H)\bul H  
\ar[dr]_{\theta_1}
\ar@{}[u]|<<<{\pushout} && A\bul (H\bul H) \ar[dl]^{\theta_2} \ar@{}[u]|<<<{\pushout}\\
&&&&\Theta  \ar@{}[u]|<<<{\pushout} &&&&
}
}
\]
Using \leref{pushoutvect}, we find that $(A\bul H)\bul H\cong ((A\bul H)\ot H)/K$, $A\bul(H\bul H)=((A\bul H)\ot H)/L$ and $\Theta=((A\bul H)\ot H)/(K+L)$
where
\[K=\{a_{[0]}\ot a_{[1]}\ot h-a_{[0]}1_{[0][0]}\ot a_{[1]}1_{[0][1]}\ot h1_{[1]}|a\ot h\in A\ot H\}\]
and
\[L=\{a1_{[0]}\ot h_{(1)}1_{[1]}\ot h_{(2)}-a1_{[0]}1_{[0']}\ot h_{(1)}1_{[1](1)}1_{[1']}\ot h_{(2)}1_{[1](2)}|a\ot h\in A\ot H\}\]
Although in general $K$ and $L$ are not necessarily isomorphic subspaces of $(A\bul H)\ot H$, we see because of axiom (CJ2) that $(\pi\ot H)\circ (\rho\ot H)\circ \rho(a)=(\pi\ot H)\circ (A\ot \Delta)\circ \rho(a)$ in $(A\bul H)\ot H$. Hence the coassociativity holds in particular in the quotient $\Theta$. We can conlude that a Caenepeel-Janssen partial action induces a quasi (and not geometric) partial comodule.
\end{example}

\begin{example}\exlabel{XbulHcomodule}
Let $(X,X\bul H,\pi_X,\rho_X)$ be a quasi partial $H$-comodule. We know that $X\ot H$ is a (global) right $H$-comodule with coaction $X\ot \Delta$. By applying the result of \exref{quotient}, we find that the epimorphism $\pi_X:X\ot H\to X\bul H$ induces  
 $X\bul H$ with the structure of a partial $H$-comodule under the partial coaction 
\[
\xymatrix@!C{
X\bul H \ar[dr]_-{\pi'_X\circ X\bul\Delta} && (X\bul H)\ot H \ar@{->>}[dl]^-{\pi'_{X,\Delta}}\\
& X\bul(H\bul H)
}
\]
which is geometric by \exref{quotient}. 
Therefore, we obtain that the following pushouts are isomorphic, where we denote $\ol{X\bul\Delta}=\pi'_X\circ X\bul\Delta$.
\begin{eqnarray}
\xymatrix@!C{
&(X\bul H)\ot H \ar@{->>}[dl]_-{\pi'_{X,\Delta}} \ar[dr]^-{\ol{X\bul\Delta}\ot H}\\
X\bul (H\bul H) \ar[dr]_{\ol{X\bul(\Delta \bul H)}} && (X\bul (H\bul H))\ot H \ar@{->>}[dl]^{\pi'_{X,\Delta,H}}\\
&X\bul ((H\bul H))\bul H)
}
\end{eqnarray}
\begin{eqnarray}
\hspace{-1cm}
\xymatrix@!C{
& (X\bul H)\ot H \ar@{->>}[dl]_-{\pi'_{X,\Delta}} \ar[dr]^-{(X\bul H)\ot \Delta}\\
X\bul (H\bul H) \ar[dr]_-{X\bul (H\bul \Delta)} && (X\bul H)\ot H\ot H \ar@{->>}[dl]^{\pi_{X,H,\Delta}} \ar@{->>}[dr]^{\pi'_{X,\Delta}\ot H} \\
& X\bul (H\bul (H\ot H)) \ar@{->>}[dr]_-{\pi''_{X,1}} && (X\bul (H\bul H))\ot H \ar@{->>}[dl]^{\pi'_{X,H,\Delta}} \\
&& X\bul (H\bul (H\bul H))
}
\end{eqnarray}
Denote $\ol{(X\bul (H\bul \Delta))}=\pi''_{X,1}\circ (X\bul (H\bul \Delta))$. Then we find that the following morphisms are identical up-to-isomorphism of their codomains. 
\[\ol{(X\bul (H\bul \Delta))}\circ \ol{(X\bul \Delta)}\simeq \ol{(X\bul (\Delta\bul H))}\circ \ol{(X\bul \Delta)}\]

When $X$ itself is a geometric partial comodule, one can use the isomorphism $\theta:X\bul (H\bul H)\cong (X\bul H)\bul H$ to rewrite the above pushouts as
\begin{eqnarray*}
 (X\bul (H\bul H))\bul H &\cong& X\bul ((H\bul H))\bul H) \\
\cong(X\bul H)\bul (H\bul H)  &\cong& X\bul (H\bul (H\bul H))
\end{eqnarray*}
we will explain this in more detail in \seref{gencoass}.
\end{example}

\subsection{Partial comodule morphisms}

\begin{definition}
If $(X,\pi_X,\rho_X)$ and $(Y,\pi_Y,\rho_Y)$ are two partial $H$-comodule data, then a {\em morphism} of partial $H$-comodule data is a couple 
$(f,f\bul H)$ of morphisms in $\Cc$, where $f:X\to Y$ and $f\bul H:X\bul H\to Y\bul H$ such that the following two squares commute
\[
\xymatrix{
X\ar[rr]^-f \ar[d]_{\rho_X} && Y\ar[d]^{\rho_Y}\\
X\bul H \ar[rr]^-{f\bul H} && Y\bul H\\
X\ot H \ar[rr]^-{f\ot H} \ar@{->>}[u]^{\pi_X} && Y\ot H \ar@{->>}[u]^{\pi_Y}
}
\]
A morphism of a quasi, lax or geometric partial comodule is a morphism of the underlying partial comodule data. 
We denote the categories of quasi, lax and geometric partial $H$-comodules respectively by
$\qPMod^H$, $\lPMod^H$ and $\gPMod^H$. When we denote $\PMod^H$, we mean any of the three partial comodule categories, without specifying which one.

If $H$ is an algebra in $\Cc$, then $H$ is a coalgebra in $\Cc^{op}$ and one defines the categories of partial modules as the opposite of the corresponding categories of partial comodules over the coalgebra $H$ in $\Cc^{op}$
\end{definition}

\begin{remark}\relabel{fbulunique}
If $(f,f\bul H)$ is a morphism of partial comodule data, then $f\bul H$ is determined by $f$. Indeed, suppose that both $(f, f\bul H), (g,g\bul H):X\to Y$ are morphisms of comodule data with $f=g$, then using the fact that $\pi_X$ is an epimorphism, it follows that $f\bul H=g\bul H$. This justifies that from now on we will denote a partial morphism $(f,f\bul H)$ just as $f$.

If moreover $\pi_X$ is a regular epimorphism (that is, it is a coequalizer) in $\Cc$, then one can express the property of the existence of $f\bul H$ more explicitly. We spell this out in the abelian case (were all epimorphisms are regular) in the next lemma.
\end{remark}

\begin{lemma}\lelabel{comodulemorphismabelian}
Suppose that the category $\Cc$ is abelian. Let $(X,\pi_X,\rho_X)$ and $(Y,\pi_Y,\rho_Y)$ be partial $H$-comodule data in $\Cc$. Then a morphism $f:X\to Y$ satisfies $(f\ot H) (\ker \pi_X)\subset \ker\pi_Y$ if and only if there exists a unique morphism $f\bul H:X\bul H\to Y\bul H$ such that $\pi_Y\circ (f\ot H)=(f\bul H)\circ \pi_X$.
\end{lemma}

\begin{proof}
The existence and uniqueness of $f\bul H$ follows directly by universal property of $(X\bul H,\pi_X)=\coker(\ker(\pi_X))$ in the abelian category $\Cc$. Conversely, if $f\bul H$ with the stated property exists, then $\pi_Y\circ (f\ot H)(\ker\pi_X)=(f\bul H)\circ \pi_X(\ker\pi_X)=0$ and by the universal property of $\ker\pi_Y$ we find then that $(f\ot H) (\ker \pi_X)\subset \ker\pi_Y$.
\end{proof}

\begin{remark}
In case $\Cc=\Vect$, one then finds that a map $f:X\to Y$ between two geometric partial modules is a morphism of partial comodules if and only if the following conditions hold:
\begin{enumerate}
\item $f(x)\bul h=0$ if $x\bul h=0$;
\item $f(x_{[0]})\bul x_{[1]}=f(x)_{[0]}\bul f(x)_{[1]}$;
\end{enumerate}
where we used the notation introduced in \reref{ParSweedler}, and where the second condition make sense thanks to the first one. 
\end{remark}

\begin{lemma}
If $f:(X,\pi_X,\rho_X,\theta_X)\to (Y,\pi_Y,\rho_Y,\theta_Y)$ is a morphism of quasi partial $H$-comodules, then there exist unique morphisms $(f\bul H)\bul H$, $f\bul(H\ot H)$, $f\bul (H\bul H)$ and $\theta_f$ such that the following diagrams commute
\[
\xymatrix{
X\bul H \ar[r]^-{\rho_X\bul H} \ar[d]_{f\bul H} & (X\bul H)\bul H \ar[d]^{(f\bul H)\bul H} & (X\bul H)\ot H \ar[l]_{\pi_{X\bul H}} \ar[d]^{(f\bul H)\ot H} \\
Y\bul H \ar[r]^-{\rho_Y\bul H}  & (Y\bul H)\bul H & (Y\bul H)\ot H \ar[l]_{\pi_{Y\bul H}} \\
}\quad 
\xymatrix{
X\bul H \ar[r]^-{X\bul\Delta} \ar[d]_{f\bul H} & X\bul (H\ot H) \ar[d]^{f\bul (H\ot H)} & X\ot H\ot H \ar[l]_{\pi_{X,\Delta}} \ar[d]^{f\ot H\ot H} \\
Y\bul H \ar[r]^-{Y\bul\Delta}  & Y\bul (H\ot H) & Y\ot H\ot H \ar[l]_{\pi_{Y,\Delta}} \\
}\]
\[
\xymatrix{
X\bul H \ar[r]^-{\pi'_X\circ X\bul\Delta} \ar[d]_{f\bul H} & X\bul (H\bul H) \ar[d]^{f\bul (H\bul H)} & (X\bul H)\ot H \ar[l]_{\pi'_{X,\Delta}} \ar[d]^{(f\bul H)\ot H} \\
Y\bul H \ar[r]^-{\pi'_Y\circ Y\bul\Delta}  & Y\bul (H\bul H) & (Y\bul H)\ot H \ar[l]_{\pi'_{Y,\Delta}} \\
}\quad
\xymatrix{
(X\bul H)\bul H \ar[r]^-{\theta_1^X} \ar[d]_{(f\bul H)\bul H} & \Theta_X \ar[d]^{\theta_f} & X\bul (H\bul H) \ar[l]_-{\theta_2^X} \ar[d]^{f\bul (H\bul H)} \\
(Y\bul H)\bul H \ar[r]^-{\theta_1^Y}  & \Theta_Y & Y\bul (H\bul H) \ar[l]_-{\theta_2^Y} \\
}
\]
If moreover $X$ and $Y$ are lax, then the following diagram commutes as well.
\[
\xymatrix{
(X\bul H)\bul H \ar[rr]^-{(f\bul H)\bul H} \ar[d]_-{\theta_X} && (Y\bul H)\bul H\ar[d]^-{\theta_Y}\\
X\bul (H\bul H) \ar[rr]^-{f\bul (H\bul H)} && Y\bul (H\bul H)
}
\]
\end{lemma}

\begin{proof}
This follows by the universal property of the considered pushouts. 
For example, $(f\bul H)\bul H: (X\bul H)\bul H\to (Y\bul H)\bul H$ is defined as the unique morphism that makes the following diagrams commute, where the inner and outer diamond are pushouts
\[
\xymatrix@!C{
&& X\ot H \ar[lldd]_-{\pi_X} \ar[ddrr]^-{\rho_X\ot H} \ar[d]^{f\ot H} \\
&& Y\ot H \ar[dl]_-{\pi_Y} \ar[dr]^-{\rho_Y\ot H} \\
X\bul H \ar[ddrr]_-{\rho_X\bul H} \ar[r]^{f\bul H} & Y\bul H \ar[dr]_{\rho_Y\bul H} && (Y\bul H)\ot H 
\ar[dl]^{\pi_{Y\bul H}} & (X\bul H)\ot H \ar[ddll]^{\pi_{X\bul H}} \ar[l]_{(f\bul H)\ot H} \\
&& (Y\bul H)\bul H \\
&& (X\bul H)\bul H \ar[u]^{(f\bul H)\bul H}
}
\]
\end{proof}

\subsection{Coassociativity}\selabel{gencoass}

For a usual $H$-comodule $(M,\rho)$, it is well-known that the coassociativity condition implies a generalized coassociativity condition saying that all morphisms from $M$ to $M\ot H^{\ot n}$ that is constructed out of a combination of $\rho$, $\Delta$ and identity maps are identical. Our next aim is to prove a similar theorem for geometric partial comodules. To this end, let consider the following compositions of partial mappings from $X$ to $X\ot H\ot H\ot H$.
Let us first construct \[\rho^1=((\rho\bul H)\bul H)\circ (\rho\bul H)\circ \rho:X\to ((X\bul H)\bul H)\bul H\] which is done in the following diagram, where all quadrangles are pushouts.
\[
\scalebox{.8}{
\hspace{-1cm}
\xymatrix@!C{
X \ar@{.>}[r]^{\rho} \ar[d]_{\rho_X} & X\ot H \ar@{->>}[dl]^{\pi_X} \ar[dr]_{\rho_X\ot H} \ar@{.>}[rr]^{\rho\ot H} && X\ot H\ot H \ar@{.>}[r]^{\rho\ot H\ot H} \ar@{->>}[dl]^{\pi_X\ot H} \ar[dr]_{\rho_X\ot H\ot H} & X\ot H\ot H \ot H \ar@{->>}[d]^{\pi_X\ot H\ot H}\\
 X\bul H \ar[dr]_{\rho_X\bul H} & (a) & (X\bul H)\ot H \ar[dr]_{(\rho_X\bul H)\ot H} \ar@{->>}[dl]^{\pi_{X\bul H}}  && (X\bul H)\ot H \ot H \ar@{->>}[dl]^{\pi_{X\bul H}\ot H} \\
& (X\bul H)\bul H \ar[dr]_{(\rho_X\bul H)\bul H} && ((X\bul H)\bul H)\ot H \ar@{->>}[dl]^{\pi_{(X\bul H)\bul H}}\\
&& ((X\bul H)\bul H)\bul H
}
}
\]
In the same way, we can construct \[\rho^2:((\rho\bul H)\bul H)\circ \ol{(X\bul \Delta)}\circ \rho:X\to (X\bul H)\bul (H\bul H),\]
where we denote as before $\ol{X\bul \Delta}=\pi'_X\circ (X\bul \Delta)$ and which is defined by the following diagram. 
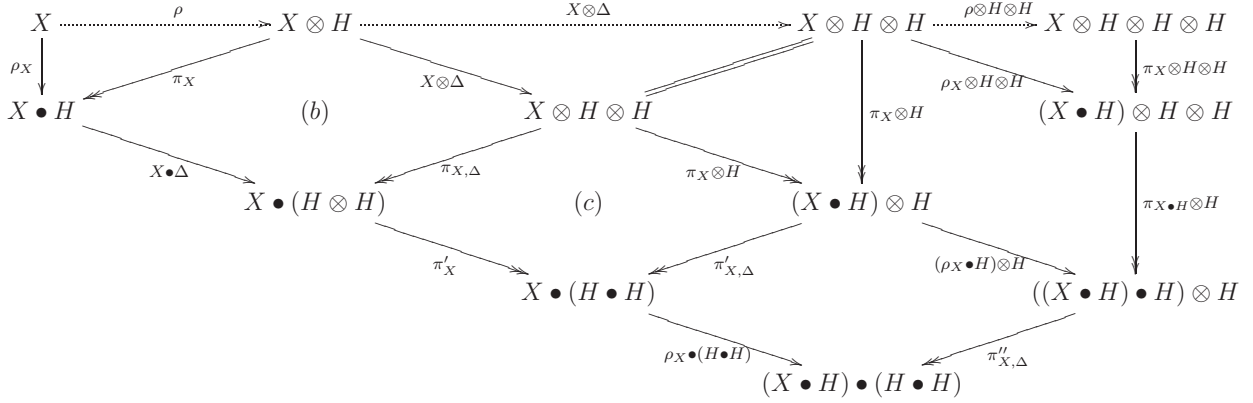
\[
\scalebox{.8}{
\hspace{-1cm}
\xymatrix@!C{
X \ar@{.>}[r]^{\rho} \ar[d]_{\rho_X} & X\ot H \ar@{->>}[dl]^{\pi_X} \ar[dr]_{X\ot \Delta} \ar@{.>}[rr]^{X\ot \Delta} && X\ot H\ot H  \ar@{->>}[dd]^{\pi_{X}\ot H} \ar@{.>}[r]^{\rho\ot H\ot H} \ar@{=}[dl] \ar[dr]_{\rho_X\ot H\ot H} & X\ot H\ot H \ot H  \ar@{->>}[d]^{\pi_X\ot H\ot H}\\
 X\bul H \ar[dr]_{X\bul \Delta} & (b) & X\ot H\ot H \ar@{->>}[dr]_{\pi_X\ot H} \ar@{->>}[dl]^{\pi_{X,\Delta}}   && (X\bul H)\ot H \ot H 
\ar@{->>}[dd]^{\pi_{X\bul H}\ot H}
\\
& X\bul (H\ot H) \ar@{->>}[dr]_{\pi'_X} & (c) & (X\bul H)\ot H \ar@{->>}[dl]^{\pi'_{X,\Delta}} \ar[dr]_{(\rho_X\bul H)\ot H}\\
&& X\bul (H\bul H) \ar[dr]_{\rho_X\bul(H\bul H)} && ((X\bul H)\bul H)\ot H \ar@{->>}[dl]^{\pi''_{X,\Delta}} \\
&&& (X\bul H)\bul (H\bul H)
}
}
\]
Since we know by the coassociativity on $X$ that the pushouts $(X\bul H)\bul H$ given by the diagram (a) is isomorphic to the pushout  $X\bul (H\bul H)$ which is the combinination of diagrams (b) and (c). 
Therefore it follows that the pushouts $((X\bul H)\bul H)\bul H$ and $(X\bul H)\bul (H\bul H)$ constructed above are isomorphic as well, in such a way that the constructed maps $\rho^1$ and $\rho^2$ from $X$ into these pushouts are identical up to this isomorphism.

Next, we construct a morphism
\[\rho^3:\ol{((X\bul H)\bul \Delta)}\circ (\rho\bul H)\circ \rho: X\to (X\bul H)\bul (H\bul H)\]
denoting $\ol{((X\bul H)\bul \Delta)}=\pi'_{X\bul H}\circ ((X\bul H)\bul \Delta)$,
as in the following diagram.
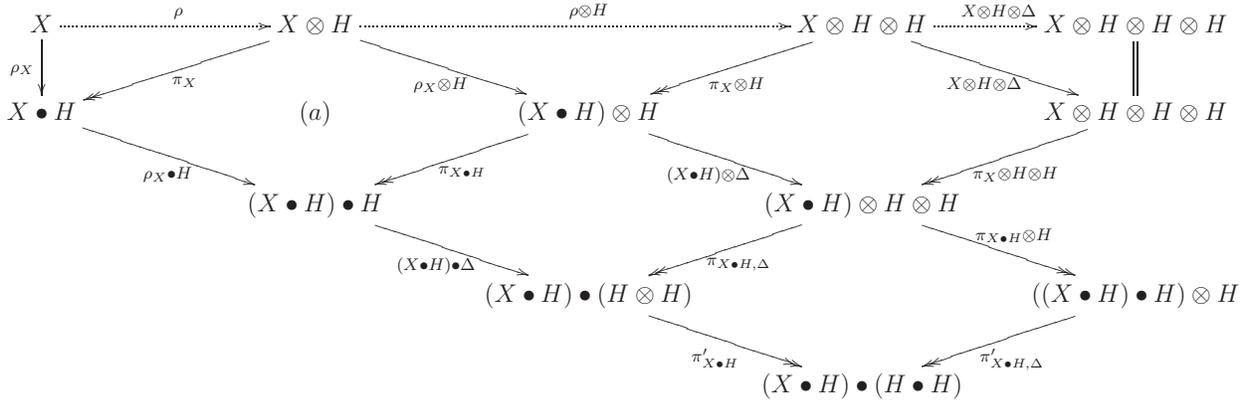
\[
\scalebox{.8}{
\hspace{-1cm}
\xymatrix@!C{
X \ar@{.>}[r]^\rho \ar[d]_{\rho_X} & X\ot H \ar@{->>}[dl]^{\pi_X} \ar[dr]_{\rho_X\ot H} \ar@{.>}[rr]^{\rho\ot H} && X\ot H\ot H \ar@{.>}[r]^{X\ot H\ot \Delta} \ar@{->>}[dl]^{\pi_X\ot H} \ar[dr]_{X\ot H\ot \Delta} 
& X\ot H\ot H \ot H \ar@{=}[d]\\
 X\bul H \ar[dr]_{\rho_X\bul H} & (a) & (X\bul H)\ot H \ar[dr]_{(X\bul H)\ot \Delta} \ar@{->>}[dl]^{\pi_{X\bul H}}  && X\ot H\ot H \ot H \ar@{->>}[dl]^{\pi_{X}\ot H\ot H}  
\\
& (X\bul H)\bul H \ar[dr]_{(X\bul H)\bul \Delta} && (X\bul H)\ot H\ot H \ar@{->>}[dl]^{\pi_{X\bul H,\Delta}} \ar@{->>}[dr]^{\pi_{X\bul H}\ot H}
\\
&& (X\bul H)\bul (H\ot H) \ar@{->>}[dr]_{\pi'_{X\bul H}} && ((X\bul H)\bul H)\ot H \ar@{->>}[dl]^{\pi'_{X\bul H,\Delta}} \\
&&& (X\bul H)\bul (H \bul H)
}
}
\]
Let us first remark that the constructed pushout is the same as the one from the previous diagram. Indeed, we had constructed $(X\bul H)\bul (H \bul H)$ as the pushout of $\pi_X$ with
\begin{eqnarray*}
((\rho_X\bul H)\ot H)\circ (\pi_X\ot H)\circ (X\ot \Delta)&=& (\pi_{X\bul H}\ot H)\circ (\rho_X\ot H\ot H)\circ (X\ot \Delta)\\
&=&  (\pi_{X\bul H}\ot H)\circ ((X\bul H)\ot \Delta)\circ (\rho_X\ot H)
\end{eqnarray*}
It follows that the morphism $\rho^3$ is identical to $\rho^2$ (and to $\rho^1$).

Furthermore, one sees that the pushout (a) appears again in the last diagram, by a same argument as before, this can be replaced by the combination of the pushouts (b) and (c), since $\theta_X:X\bul(H\bul H)\to (X\bul H)\bul H$ is an isomorhpism. This leads us to the map 
\[\rho^4:\ol{((X\bul H)\bul \Delta)}\circ \ol{X\bul \Delta}\circ \rho: X\to X\bul (H\bul (H\bul H))\]
\[
\scalebox{.8}{
\hspace{-1cm}
\xymatrix@!C{
X \ar@{.>}[r]^\rho \ar[d]_{\rho_X} & X\ot H \ar@{->>}[dl]^{\pi_X} \ar[dr]^{(\pi_{X}\ot H)\circ (X\ot\Delta)} \ar@{.>}[rr]^{\rho\ot H} && X\ot H\ot H \ar@{.>}[r]^{X\ot H\ot \Delta} \ar@{->>}[dl]^{\pi_X\ot H} \ar[dr]_{X\ot H\ot \Delta} 
& X\ot H\ot H \ot H \ar@{=}[d]\\
 X\bul H \ar[dr]_{\ol{X\bul\Delta}} & (b)+(c) & (X\bul H)\ot H \ar[dr]_{(X\bul H)\ot \Delta} \ar@{->>}[dl]^{\pi'_{X,\Delta}}  && X\ot H\ot H \ot H \ar@{->>}[dl]^{\pi_{X}\ot H\ot H}  
\\
& X\bul (H\bul H) \ar[dr]_{X\bul (H\bul \Delta)} && (X\bul H)\ot H\ot H \ar@{->>}[dl]^{\pi_{X,H,\Delta}} \ar@{->>}[dr]^{\pi'_{X,\Delta}\ot H}
\\
&& X\bul (H\bul (H\ot H)) \ar@{->>}[dr]_{\pi''_{X,1}} && (X\bul (H\bul H))\ot H \ar@{->>}[dl]^{\pi'_{X,H,\Delta}} \\
&&& X\bul (H\bul (H \bul H))
}
}
\]
Remark that $(X\bul H)\bul (H \bul H)$ is the pushout of the pair $(\pi_{X\bul H},(\pi_{X\bul H}\ot H)\circ ((X\bul H)\ot \Delta)$ and $X\bul (H\bul (H \bul H))$ is the pushout of the pair $(\pi'_{X,\Delta},(\pi'_{X,\Delta}\ot H)\circ ((X\bul H)\ot \Delta)$. Since $\pi_{X\bul H}=\theta_X\circ \pi'_{X,\Delta}$ and $\theta_X$ is an isomorphism, it follows that both pushouts are isomorphic and $\phi^3$ and $\phi^4$ are identical up to this isomorphism.

Let us now consider the morphism 
\[\rho^5:\ol{((X\bul \Delta)\bul H)}\circ (\rho\bul H)\circ \rho: X\to (X\bul (H\bul H))\bul H\]
where $\ol{((X\bul \Delta)\bul H)}=(\pi'_X\bul H)\circ ((X\bul \Delta)\bul H)$ and that is given by the following diagram
\[
\scalebox{.8}{
\hspace{-1cm}
\xymatrix@!C{
X \ar@{.>}[r]^{\rho} \ar[d]_{\rho_X} & X\ot H \ar@{->>}[dl]^{\pi_X} \ar[dr]_{\rho_X\ot H} \ar@{.>}[rr]^{\rho \ot H} && X\ot H\ot H \ar@{.>}[r]^{X\ot \Delta\ot H} \ar@{->>}[dl]^{\pi_X\ot H} \ar[dr]_{X\ot \Delta\ot H} 
& X\ot H\ot H \ot H \ar@{=}[d]\\
X\bul H \ar[dr]_{\rho_X\bul H} & (a) & (X\bul H)\ot H \ar[dr]_{(X\bul\Delta)\ot H} \ar@{->>}[dl]^{\pi_{X\bul H}}  && X\ot H\ot H \ot H \ar@{->>}[dl]^{\pi_{X,\Delta}\ot H} 
\\
& (X\bul H)\bul H \ar[dr]_{(X\bul \Delta)\bul H} && (X\bul (H\ot H))\ot H \ar@{->>}[dl]^{\pi_{X\bul(H\ot H)}} \ar@{->>}[dr]^{\pi'_X\ot H}
\\
&& (X\bul (H\ot H))\bul H \ar@{->>}[dr]_{\pi'_{X}\bul H} && (X\bul (H\bul H))\ot H \ar@{->>}[dl]^{\pi_{X\bul(H\bul H)}} \\
&&& (X\bul (H\bul H))\bul H
}
}
\]
And again, by replacing the pullback (a) by the pullback (b)+(c), we obtain a map that is the same up-to-isomorphism the same as $\rho^5$:
\[\rho^6:\ol{(X\bul (\Delta\bul H))}\circ \ol{X\bul \Delta} \circ \rho: X\to X\bul ((H\bul H)\bul H)\]
where $\ol{(X\bul (\Delta\bul H))}=\pi''_{X,2}\circ (X\bul (\Delta\bul H))$. This map $\rho^6$ is defined by the following diagram.
\[
\scalebox{.8}{
\hspace{-1cm}
\xymatrix@!C{
X \ar@{.>}[r]^{\rho} \ar[d]_{\rho_X} & X\ot H \ar@{->>}[dl]^{\pi_X} \ar[dr]^{(\pi_X\ot H)\circ(X\ot \Delta)} \ar@{.>}[rr]^{\rho \ot H} && X\ot H\ot H \ar@{.>}[r]^{X\ot \Delta\ot H} \ar@{->>}[dl]^{\pi_X\ot H} \ar[dr]_{X\ot \Delta\ot H} 
& X\ot H\ot H \ot H \ar@{=}[d]\\
X\bul H \ar[dr]_{\ol{X\bul \Delta}} & (b)+ (c)& (X\bul H)\ot H \ar[dr]_{(X\bul\Delta)\ot H} \ar@{->>}[dl]^{\pi'_{X,\Delta}}  && X\ot H\ot H \ot H \ar@{->>}[dl]^{\pi_{X,\Delta}\ot H} 
\\
& X\bul (H\bul H) \ar[dr]_{X\bul (\Delta\bul H)} && (X\bul (H\ot H))\ot H \ar@{->>}[dl]^{\pi_{X,\Delta,H}} \ar@{->>}[dr]^{\pi'_X\ot H}
\\
&& X\bul ((H\ot H))\bul H) \ar@{->>}[dr]_{\pi''_{X,2}} && (X\bul (H\bul H))\ot H \ar@{->>}[dl]^{\pi'_{X,\Delta,H}} \\
&&& X\bul ((H\bul H))\bul H)
}
}
\]
By \exref{XbulHcomodule}, we know that $X\bul ((H\bul H))\bul H)\cong X\bul (H\bul(H\bul H))$ and the maps $\rho^4$ and $\rho^6$ are identical up to this isomorphism. 

Hence we have hereby proven that the all above constructed pushouts are isomorphic and the maps $\rho^i$ ($i=1,\ldots,6$) are identical up to these isomorphisms. All this is summarized in the following result.

\begin{theorem}[generalized coassociativity]\thlabel{gencoass}
Let $(X,\pi_X,\rho_X,\theta_X)$ be a geometric partial comodule. 
Then the pushouts introduced above are all isomorphic
\begin{eqnarray*}
X\bul (H\bul(H\bul H))&\cong& (X\bul H)\bul (H\bul H)\\
\cong ((X\bul H)\bul H)\bul H &\cong& (X\bul (H\bul H))\bul H\\  &\cong & X\bul ((H\bul H)\bul H)
\end{eqnarray*}
Moreover up to these isomorphisms, the following morphisms $X\to X\bul H\bul H\bul H$ are identical 
\begin{eqnarray*}
 \ol{(X\bul H\bul \Delta)}\circ \ol{(X\bul \Delta)}\circ \rho & \simeq 
(\rho\bul H\bul H)\circ \ol{(X\bul \Delta)}\circ \rho \simeq &
 \ol{(X\bul H\bul \Delta)}\circ (\rho\bul H)\circ \rho \simeq\\
(\rho\bul H\bul H)\circ (\rho\bul H)\circ \rho &
\simeq \ol{(X\bul \Delta\bul H)}\circ (\rho\bul H)\circ \rho \simeq & 
\ol{(X\bul \Delta\bul H)}\circ \ol{(X\bul \Delta)}\circ \rho 
\end{eqnarray*}
\end{theorem}

\begin{corollary}\colabel{XbulHcomodule}
If $(X,\pi_X,\rho_X)$ is a geometrically partial $H$-comodule, then $(X\bul H, (X\bul H)\bul H, \pi_{X\bul H},\rho_X\bul H)$ is a geometrically partial $H$-comodule.
\end{corollary}

\begin{corollary}
All higher coassociativity conditions follow now by an induction argument from the previous two results.
\end{corollary}

\begin{remarks}
\begin{enumerate}
\item
A lax version of the above results on generalized coassociativity can be proven in the same way. Indeed, analysing the reasoning at the start of this section, each of the isomorphisms between the pullbacks obtained in \thref{gencoass} follows from the isomorphism $\theta:X\bul (H\bul H)\to (X\bul H)\bul H$ at appropriate places. When $\theta_X$ is only assumed to be a morphism (not an isomorphism), then we also obtain only morphisms (and not isomorphisms) between the constructed pullbacks. The coassociativity will then hold up to composition with the induced morphisms onto $((X\bul H)\bul H)\bul H$. 
\item 
As the isomorphisms between the respective pullbacks are constructed by applying the universal property of the pullback, one can moreover easily see, that these isomorphisms are compatible in a way that the following diagram commutes
\[
\xymatrix{
X\bul(H\bul (H\bul H)) \ar[r] \ar[d] & (X\bul H)\bul (H\bul H) \ar[r] & ((X\bul H)\bul H)\bul H\\
X\bul ((H\bul H)\bul H) \ar[rr] && (X\bul (H\bul H))\bul H\ar[u]
}
\]
where all arrows are isomorphisms in the geometric case, and just morphisms in the lax case. The commutativity of this diagrams seems to suggest that there is an underlying (skew) monoidal structure with tensor product $-\bul-$. 
In the next section, we will show that in case $H$ is a bialgebra, there is at least a lax monoidal structure on the category of geometric partial modules, which coincides with the $\bul$-product that we encountered so far.
\end{enumerate}
\end{remarks}

\subsection{Completeness and cocompleteness of the category of partial comodules}

For global comodules, the forgetful functor $U:\Mod^H\to \Cc$ allows a right adjoint given by the free functor $-\ot H:\Cc\to \Mod^H$. Since every global comodule is also a partial module, the free functor $-\ot H:\Cc\to \PMod^H$ still makes sense, however it no longer serves as a right adjoint for the forgetful functor $U:\PMod^H\to\Cc$, which is defined as $U(X,\rho_X,\pi_X,\theta_X)=X$ on objects and $U(f,f\bul H)=f$ on morphisms. We now show that the forgetful functor still has a right adjoint. Troughout this section, we suppose that the counit $\epsilon$ of the coalgebra $H$ is an epimorphism in $\Cc$ (this is for example the case if $\Cc=\Vect$ or if $H$ is a bialgebra).

\begin{proposition}
Let $V$ be any object in $\Cc$, then $V$ can be endowed with a partial $H$-comodule structure putting
$V\bul H=V$, $\pi=V\ot\epsilon_H$ and $\rho=id_V$. We call this the ``trivial partial comodule structure'' on $V$.

Moreover a trivial partial comodule is always geometric and the functor $T:\Cc\to \PMod$ that assigns to each $\Cc$-object the trivial partial comodule structure, is fully faithful and a right adjoint for the forgetful functor $U:\PMod^H\to\Cc$.
\end{proposition}

\begin{proof}
It can be easily verified that $(V,V,V\ot\epsilon_H,id_V)$ is a geometric partial $H$-comodule with $(V\bul H)\bul H=V\bul(H\bul H)=V$. 

Given a partial comodule $(X,X\bul H,\pi_X,\rho_X)$, we find that $TU(X)=(X,X,X\ot\epsilon_H,id_X)$ and we define the unit of the adjunction as $\eta_X=(id_X,X\bul \epsilon_H):X\to TU(X)$. For any object $V$ in $\Cc$, we see that $UT(V)=V$. Then the unit-counit conditions become trivial. Since the counit is the identity, we obtain that $T$ is fully faithful. 
\end{proof}

Since the forgetful functor has a right, it preserves all colimits that exist in $\PMod^H$. The main aim of this section is to show that colimits and limits indeed exist in $\PMod^H$. Let us first show that thanks to the observation of the previous proposition, the category $\PMod^H$ is well-copowered.

Recall that a category is called {\em well-copowered} if and only if for any object $X$, there exist up-to-isomorphism only a set of epimorphisms $f:X\to Y$. 

\begin{corollary}\colabel{surjepi}
A morphism $f\in \PMod^H$ is an epimorphism if and only if $U(f)=f$ is an epimorphism in $\Cc$.
Furthermore, the category $\PMod^H$ is well-copowered if $\Cc$ is so.
\end{corollary}

\begin{proof}
Since the forgetful functor $U:\PMod^H\to \Cc$ has a right adjoint, $U$ preserves epimorphisms. 
Conversely, if $f:X\to Y$ in $\PMod^H$ is such that $U(f)$ is an epimorphism, then $f$ is an epimorphism as well. Indeed, suppose that we have $g,h:Y\to Z$ in $\PMod^H$ such that $g\circ f=h\circ f$. Then also $U(g)\circ U(f)=U(g)\circ U(f)$ in $\Cc$ and hence $U(g)=U(f)$. But in \reref{fbulunique}, we remarked that for a morphism $f\in\PMod^H$, $f\bul H$ is completely determined by $f$ (or by $U(f)$ to be precise). Hence we find that $g=h$ in $\PMod^H$.

Let $(X,X\bul H,\pi_X,\rho_X)$ be a partial comodule datum. Since $\Cc$ is well-copowered, there exists up-to-isomorphism only a set of epimorphisms $f:X\to Y$ in $\Cc$. Moreover, for each $Y$, there exist again since $\Cc$ is well-copowered, only a set of epimorphisms $Y\ot H\to Y\bul H$, hence also only a set of partial comodule data over $Y$. We conclude that there will be only a set of epimorphisms $f:X\to Y$ in $\PMod^H$ and hence $\PMod^H$ is also well-copowered.
\end{proof}

\begin{theorem}\thlabel{cocomplete}
Suppose that the endofunctor $-\ot H:\Cc\to\Cc$ preserves colimits.
Then the following statements hold.
\begin{enumerate}[(i)]
\item
If the category $\Cc$ is $k$-linear then $\PMod^H$ is also $k$-linear and the forgetful functor is $k$-linear.
\item
If the category $\Cc$ has all colimits of a shape $\Zz$, then $\PMod^H$ also has colimits of shape $\Zz$. 
Hence, if $\Cc$ is cocomplete then $\PMod^H$ is cocomplete and the forgetful functor $U:\PMod^H\to \Cc$ preserves colimits.
\end{enumerate}
Consequently, in case the above conditions hold, $\PMod^H$ is addive.
%If the category $\Cc$ is additive, then $\PMod^H$ is also additive.
\end{theorem}

\begin{proof}
\ul{(i)}
Let $X=(X,X\bul H,\pi_X,\rho_X)$ and $(Y,Y\bul H,\pi_Y,\rho_Y)$ be a two partial comodule data and $(f,f\bul H),(g,g\bul H):X\to Y$ two morphisms. Let us verify that $(f+g,f\bul H+g\bul H)$ is again a morphism.
Then we have
\[\rho_Y\circ (f+g)=\rho_Y\circ f+ \rho_Y\circ g= (f\bul H)\circ \rho_X +(g\bul H)\circ \rho_X=((f\bul H) +(g\bul H))\circ \rho_X\]
And similarly, $((f\bul H)+(g\bul H))\circ \pi_X=\pi_Y\circ ((f\ot H)+(g\ot H)).$
Hence $(f+g,f\bul H+g\bul H)$ is indeed a morphism in $\PMod^H$.

Similarly, for any $a\in k$, we define $a(f,f\bul H)=(af,af\bul H)$. One easily verifies that this is again a morphism, and using this addition and scalar multiplication, the Hom-sets in $\PMod^H$ are $k$-modules and composition is $k$-bilinear.\\
\ul{(ii)}
Let $\Zz$ be any small category and $F:\Zz\to \PMod^H$ a functor, where we denote for each $Z\in \Zz$, $FZ=(FZ,FZ\bul H,\rho_{FZ},\pi_{FZ})$, i.e. we denote $UFZ=FZ$ for short. Consider the functor $UF:\Zz\to \Cc$ and denote $(C,\gamma_Z)=\colim UF$, where $\gamma_Z:FZ\to C$ are such that $\gamma_Z=\gamma_{Z'}\circ Ff$ for any $f:Z\to Z'$ in $\Zz$. Consider now the functor $(UF)^H:\Zz\to\Cc$ given by $(UF)^HZ=FZ\ot H$. Then by assumption we have that $\colim (UF)^H=(C\ot H,\gamma_Z\ot H)$. Finally consider the functor $(UF)^\bul:\Zz\to\Cc$ given by $(UF)^\bul Z=FZ\bul H$ for all $Z\in\Zz$, and denote $\colim (UF)^\bul=(C\bul H,\delta_Z)$ where $\delta_Z:FZ\bul H\to C\bul H$ are such that $\delta_Z=\delta_{Z'}\circ Ff\bul H$. Let us verify that $(C\bul H,\delta_Z\circ \rho_{FZ})$ is a cocone for $UF$. Indeed, for any morphism $f:Z\to Z'$ in $\Zz$, $Ff$ is a morphism in $\PMod^H$ and hence the following diagram commutes
\[
\xymatrix{
FZ \ar[rr]^-{Ff} \ar[d]_{\rho_{FZ}} && FZ' \ar[d]_{\rho_{FZ'}}\\
FZ\bul H \ar[rr]^-{Ff\bul H} \ar[dr]_{\delta_Z} && FZ'\bul H \ar[dl]^{\delta_{Z'}}\\
& C\bul H
}
\]
By the universal property of $\colim UF$, we then obtain a unique morphism $\rho_C:C\to C\bul H$ such that $\delta_Z\circ \rho_{FZ}=\rho_C\circ \gamma_{Z}$ for all $Z\in \Zz$. 
In the same way, one shows that $(C\bul H,\delta_Z\circ \pi_{FZ})$ is a cocone for $UF^H$, and hence there exists a morphism $\pi_C:C\ot H\to C\bul H$ such that $\delta_Z\circ \pi_{FZ}=\pi_C\circ \gamma_{Z}\ot H$ for all $Z\in\Zz$. The situation is summarized in the next diagram.
\[
\xymatrix{
FZ \ar[rr]^-{\gamma_Z} \ar[d]_{\rho_{FZ}} && C \ar[d]^-{\rho_C}\\
FZ\bul H \ar[rr]^-{\delta_Z} && C\bul H\\
FZ\ot H \ar[u]^{\pi_{FZ}} \ar[rr]^{\gamma_Z\ot H} && C\ot H \ar[u]_{\pi_C}
}
\]
Let us show that $(C,C\bul H,\rho_C,\pi_C)$ is a partial comodule datum, i.e. that $\pi_C:C\ot H\to C\bul H$ is an epimorphism in $\Cc$. To this end, consider $f,g:C\bul H\to X$ in $\Cc$ such that $f\circ \pi_C=g\circ \pi_C$. Then for all $Z\in \Zz$ we have that 
\begin{eqnarray*}
f\circ \pi_C\circ (\gamma_Z\ot H) &=& f\circ \delta_Z\circ \pi_{FZ}\\
= g\circ \pi_C\circ (\gamma_Z\ot H) &=& g\circ \delta_Z\circ \pi_{FZ}
\end{eqnarray*}
Since each $\pi_{FZ}$ is epi, we find $f\circ \delta_Z=g\circ \delta_Z$ for all $Z$ and since the $\delta_Z$ are jointly epi, we obtain that $f=g$ and therefore $\pi_C$ is indeed an epimorphism. 

Furthermore, by the interchange law for colimits, it follows that the pushouts $C\bul H(\bul H)$, $(C\bul H)\bul H$ and $\Theta_C$ can be computed as the colimits of the respective functors $\Zz\to\Cc$ that construct the pushouts $Z\bul H(\bul H)$, $(Z\bul H)\bul H$ and $\Theta_Z$. Hence, it follows that if all $FZ$ are quasi, lax or geometric comodules, then $Z$ will be such as well.\\
The last statement follows from the above, since it is well-know that a preadditive category with binary coproducts is additive.
\end{proof}

As we will show further in this section, there exist monomorphisms $f$ in $\PMod^H$ such that $U(f)$ is not a monomorphism in $\Cc$. In particular, $U$ does not have a left adjoint. Nevertheless, we have the following result.

\begin{lemma}
Consider a morphism $f:X\to Y$ in $\PMod^H$. 
If $Uf:UX\to UY$ is a monomorphism in $\Cc$, then $f$ is also a monomorphism in $\PMod^H$.
\end{lemma}

\begin{proof}
Consider two morphisms $g,h:Z\to X$ in $\PMod^H$ such that $f\circ g=f\circ h$. Since $Uf$ is a monomorphism, we obtain $Ug=Uh$. Then by \reref{fbulunique}, we find that also $g\bul H=h\bul H$, i.e. $g=h$ in $\PMod$.
\end{proof}

\begin{definition}
A {\em subcomodule} of a partial comodule $(X,X\bul H,\rho_X,\pi_X)$ is a partial comodule datum $(Y,Y\bul H,\rho_Y,\pi_Y)$, together with a morphism $f: Y \to X$ for which both $f$ and $f\bul H$ are monomorphisms in $\Cc$.
\end{definition}

From now on, we restrict to our case of interest $\Cc=\Vect_k$ where $k$ is a field.

\begin{proposition}\prlabel{YinX}
Let $(X,X\bul H,\rho_X,\pi_X)$ be a partial comodule datum and $j:Y\to X$ a subobject of $X$ in $\Vect_k$.
Consider the epi-mono factorization of $\pi_X\circ (j\ot H):Y\ot H\to X\bul H$, which we denote as follows:
\[
\xymatrix{
Y\ot H \ar@{->>}[rr]^-{\pi_Y} && Y\bul H \ar@{^(->}[rr]^-{j\bul H} && X\bul H
}
\]
Then 
\begin{enumerate}[(i)]
\item $\ker\pi_Y\cong \im (j\ot H)\cap \ker\pi_X$;
\item 
%Denote as usual by $Y\bul (H\bul H)$ the pushout of $(\pi_Y,(\pi_Y\ot H)\circ Y\ot \Delta)$. Then $Y\bul (H\bul H)$ is isomorphic to the image of the map $\pi'_{X,\Delta}\circ (j\bul H)\ot H$;
%{\color{blue} Alternative formulation. 
Denote as usual by $Y\bul (H\bul H)$ the pushout of the pair $(\pi_Y,(\pi_Y\ot H)\circ Y\ot \Delta)$ and by $\pi'_{Y,\Delta}:(Y\bul H)\ot H\to Y\bul (H\bul H)$ associated pushout of the morphism $\pi_Y$. Then $\ker\pi'_{Y,\Delta}\cong \im((j\bul H)\ot H) \cap \ker\pi'_{X,\Delta}$ and therefore $Y\bul (H\bul H)$ is isomorphic to the image of the map $\pi'_{X,\Delta}\circ (j\bul H)\ot H$;
%}
\end{enumerate}
If moreover $Y$ allows a partial comodule datum of the form $(Y,Y\bul H,\rho_Y,\pi_Y)$ such that $j$ is a morphism of partial comodule data, then
\begin{enumerate}[(i)]
\setcounter{enumi}{2}
\item $Y$ is a partial subcomodule of $X$.
\item $(Y\bul H)\bul H$ is isomorphic to the image of the map $\pi_{X\bul H}\circ (j\bul H)\ot H$;
\item if $X$ is a lax (resp. geometric) partial comodule, then $Y$ is as well a lax (resp. geometric) partial comodule.
\end{enumerate}
\end{proposition}

\begin{proof}
\ul{(i)}. By construction we have the following commutative diagram
\[
\xymatrix{
Y\ot H \ar[d]_{\pi_Y} \ar@{^(->}[rr]^{j\ot H} && X\ot H\ar@{->>}[d]^{\pi_X}\\
Y\bul H \ar@{^(->}[rr]^{j\bul H} && X\bul H
}
\]
Since $j\bul H$ is injective and by the commutativity of the diagram, we then find that 
$$\ker\pi_Y=\ker((j\bul H)\circ \pi_Y)=\ker(\pi_X\circ (j\ot H))\cong\im(j\ot H)\cap\ker\pi_X.$$
\ul{(ii)}.
By definition of a partial comodule, we know that $(X\bul \epsilon)\circ \pi_X(\sum x_i\ot h_i)=\sum x_i\epsilon(h_i)$ for all $\sum x_i\ot h_i\in X\ot H$, so in particular also for all $\sum x_i\ot h_i\in Y\ot H$. Hence, it follows that the map $(X\bul\epsilon)\ot H$ is a left inverse for the map $(\pi_X\ot H)\circ (X\ot \Delta)$. Also, remember that the kernel of the map $\pi'_{X,\Delta}$ is given by $(\pi_X\ot H)\circ (X\ot \Delta)(\ker\pi_X)$. Combining these observations, we find that $(\pi_X\ot H)\circ (X\ot \Delta)$ is an isomorphism between $\ker \pi_X$ and $\ker\pi'_{X,\Delta}$. As all maps restrict to $Y$, we find in the same way that $(\pi_Y\ot H)\circ (Y\ot \Delta):Y\ot H\to (Y\bul H)\ot H$ has a left inverse and this map defines an isomorphism between $\ker\pi_Y$ and $\ker\pi'_{Y,\Delta}$. This is summarized in the next diagram.
\[
\xymatrix{
\ker\pi'_{Y,\Delta} \ar@{^(->}[rrr] \ar@{^(->}[ddd] \ar[dr]^\cong &&& \ker\pi'_{X,\Delta} \ar@{^(->}[ddd] \ar[dl]_\cong \\
& \ker\pi_Y \ar@{^(->}[r] \ar@{^(->}[d] & \ker\pi_X \ar@{^(->}[d]\\
& Y\ot H \ar[r] \ar@{^(->}@<.5ex>[dl] & X\ot H \ar@{_(->}@<-.5ex>[dr]\\
(Y\bul H)\ot H \ar@{^(->}[rrr]^{(j\bul H)\ot H} \ar@<.5ex>@{->>}[ur] &&& (X\bul H)\ot H \ar@<-.5ex>@{->>}[ul]
}
\]
As part (i) states exactly that the inner square of this diagram is a pullback, a simple diagram chasing argument shows that the outer square is also a pullback, i.e.\ $\ker\pi'_{Y,\Delta}\cong \im((j\bul H)\ot H) \cap \ker\pi'_{X,\Delta}$. 

From this, it follows immediately that the map $(j\bul H)\bul H$ is injective, which proofs the statement.
\[
\xymatrix{
\ker\pi'_{Y,\Delta} \ar@{^(->}[d]  \ar@{^(->}[rr]^{(j\bul H)\ot H} && (\ker\pi'_{X,\Delta})\ar@{^(->}[d] \\
(Y\bul H)\ot H \ar@{^(->}[rr]^-{(j\bul H)\ot H} \ar@{->>}[d]_{\pi'_{Y,\Delta}} && (X\bul H)\ot H \ar@{->>}[d]^{\pi'_{X,\Delta}}\\
Y\bul (H\bul H) \ar[rr]^{j\bul (H\bul H)} && X\bul (H\bul H) \\
}\]
%}
%\\
\ul{(iii)}.
It is clear by construction that $(j,j\bul H)$ is a morphism of partial comodule data and $j\bul H$ is injective.\\
\ul{(iv)}.
This is proven in the same way as in part (ii). We have to show that $(j\bul H)\bul H:(Y\bul H)\bul H\to (X\bul H)\bul H$ is injective. So suppose that $(y\bul h)\bul h'\in (Y\bul H)\bul H$ is such that $(j(y)\bul h)\bul h'=0$ in $(X\bul H)\bul H$. Since $\pi_{Y\bul H}$ is surjective, we find that $(y\bul h)\bul h'=\pi_{Y\bul H}((y\bul h)\ot h')$ and $(j(y)\bul h)\ot h'\in \ker\pi_{X\bul H}=(\rho_X\ot H)(\ker\pi_X)$. Hence, $(j(y)\bul h)\ot h'=(x_{i[0]}\bul x_{i[1]})\ot h_{i}$ for some $x_i\ot h_i\in\ker\pi_X$. Applying $(X\bul \epsilon)\ot H$ to the last identity, we obtain by part (i) that
\[x_i\ot h_i=j(y)\ot \epsilon(h)h'\in j(Y)\ot H\cap \ker\pi_X\cong \ker\pi_Y.\]
Hence $(j(y)\bul h)\ot h'=(x_{i[0]}\bul x_{i[1]})\ot h_{i}\in (j\bul H)\ot H\circ (\rho_Y\ot H) (\ker\pi_Y)\cong \ker\pi_{Y\bul H}$, so $(y\bul h)\bul h'=0$.\\
\ul{(v)}.  Suppose that $X$ is a lax partial module. Then by part (iii) and (iv) above, we can restrict and corestrict $\theta_X$ to obtain a morphism $\theta_Y:Y\bul(H\bul H)\to (Y\bul H)\bul H$. If moreover $X$ is geometric, than we can also restrict and corestrict $\theta^{-1}_X$ to obtain an inverse $\theta_Y^{-1}$ of $\theta_Y$ and $Y$ is again geometric.
\end{proof}

The following corollary describes a phenomenon that was also observed in \cite{ABV2} for the case of partial representations.

\begin{corollary}\colabel{subglobal}
Any partial subcomodule of a global comodule is again global.
\end{corollary}

\begin{proof}
By \prref{YinX}, we know that for partial subcomodule $Y$ of partial comodule $X$ that $\ker\pi_Y\subset \ker\pi_X$. Moreover, if $X$ is global then $\ker\pi_X=0$ and therefore also $\ker\pi_Y=0$ so $Y$ is global.
\end{proof}

We are now ready to prove the `fundamental theorem for partial comodules'.

\begin{theorem}[Fundamental theorem for partial comodules]
Let $X=(X,X\bul H,\rho_X,\pi_X)$ be a quasi partial comodule over the $k$-coalgebra $H$, and consider any $x\in X$. Then there exists a finite dimensional (quasi) partial subcomodule $Y\subset X$ such that $x\in Y$.

Consequently, if $X$ is moreover lax (resp. geometric), then $Y$ is as well lax (resp. geometric).
\end{theorem}

\begin{proof}
Take $x\in X$ and write $\rho_X(x)=\sum_i y_i\bul h_i=\pi_X(\sum y_i\ot h_i)$, where $h_i$ is a base of $H$. We know that in the tensor $\sum_i y_i\bul h_i$ only a finite number of the elements $y_i$ are non-zero.

By coassociativity in the partial comodule $X$, the identity
\begin{equation}\eqlabel{fund_coass}
\theta_1\circ (\rho_X\bul H)(\rho_X(x))=\theta_2\circ \pi'_{X}\circ (X\bul \Delta)(\rho_X(x))
\end{equation}
holds in the coassociativity pushout $\Theta$. 
Write as before $\rho_X(x)=\pi_X(\sum y_i\ot h_i)$,
and let us denote $\Delta(h_i)=\sum a^i_{jk}h_j\ot h_k$ for certain $a_{jk}^i\in k$.
Then we can rewrite the right hand side of \equref{fund_coass} as
\begin{eqnarray*}
\theta_2\circ \pi'_{X}\circ (X\bul \Delta)(\rho_X(x)) &=& 
\theta_2\circ \pi'_{X}\circ (X\bul \Delta)(\pi_X(\sum y_i\ot h_i))\\ &=&
\theta_2\circ \pi'_{X,\Delta}(\sum \pi_X(y_i\ot h_{i(1)})\ot h_{i(2)})\\
&=& \theta_2\circ \pi'_{X,\Delta}(\sum\pi_X(y_k\ot a^k_{ji}h_j) \ot h_i).
\end{eqnarray*}
On the other hand, the left hand side of \equref{fund_coass} can be rewritten as
\begin{eqnarray*}
\theta_1\circ  (\rho_X\bul H)(\rho_X(x))&=& \theta_1\circ  (\rho_X\bul H)(\pi_X(\sum y_i\ot h_i))\\
&=& \theta_1\circ  \pi_{X\bul H} (\sum \rho(y_i)\ot h_i).
\end{eqnarray*}

Furthermore, by definition of the coassociativity pushout, we have the identity $\theta_1\circ \pi_{X\bul H}=\theta_2\circ \pi'_{X,\Delta}: (X\bul H)\ot H \to \Theta$. Moreover, these maps are surjective, and their kernel is given by $\ker\pi_{X\bul H}+\ker\pi'_{X,\Delta}$. From \leref{pushoutvect} it follows that $\ker\pi_{X\bul H}=(\rho\ot H)(\ker\pi_X)$ and $\ker\pi'_{X,\Delta}=(\pi_X\ot H)\circ (X\ot \Delta)(\ker\pi_X)$. Hence we know that there exist elements $\sum z_i\ot h_i, \sum z'_i\ot h_i\in \ker\pi_X$ such that
$$ \sum \rho(y_i)\ot h_i + \sum \rho(z_i)\ot h_i = 
\sum\pi_X(y_k\ot a^k_{ji}h_j) \ot h_i+ \sum\pi_X(z'_k\ot a^k_{ji}h_j) \ot h_i$$
in $(X\bul H)\ot H$.
When we apply now $(X\bul \epsilon)\ot H$ to this identity, we obtain that
$$\sum y_i\ot h_i + \sum z_i\ot h_i = \sum y_i\ot h_i + \sum z'_i\ot h_i$$
from which it follows that $\sum z_i\ot h_i=\sum z'_i\ot h_i$. 
Since $\sum z_i\ot h_i\in \ker \pi_X$, we still have that $\rho(x)=\pi_X((y_i+z_i)\ot h_i)$.
Combining the above, we find that
$$ \sum \rho(y_i+z_i)\ot h_i = 
\sum\pi_X((y_k+z_k)\ot a^k_{ji}h_j) \ot h_i$$
and using the linear independence of the $h_i$'s, we obtain that 
$$\rho(y_i+z_i)=\pi_X((y_k+z_k)\ot a^k_{ji}h_j).$$
Hence, when we define $Y$ as the subspace of $X$ generated by the (finite number of non-zero) elements $y_i+z_i$, we see that $x\in Y$ and $\rho(Y)\subset \pi_X(Y\ot H)$. 
As in \prref{YinX}, define $Y\bul H$ and $\pi_Y:Y\ot H\to Y\bul H$ by means of the epi-mono factorisation of the map $\pi_X$ restricted to $Y$, then by construction we have the following commutative diagram
\[
\xymatrix{
Y\ot H \ar@{^(->}[rr] \ar@{->>}[d]^{\pi_Y} && X\ot H \ar@{->>}[d]^{\pi_X}\\
Y\bul H \ar@{^(->}[rr] && X\bul H
}
\]
and moreover, the above shows that $(Y,Y\bul H,\pi_Y,\rho_Y)$, where $\rho_Y$ is the restriction of $\rho_X$ to $Y$, is a partial comodule datum and the inclusion map $j:Y\to X$ is a morphism of partial comodules. Then, by \prref{YinX}, we know that $Y$ is a partial subcomodule of $X$. The  last statement follows from the fact that any quasi partial subcomodule of a lax (resp. geometric) partial comodule is itself lax (resp. geometric) by \prref{YinX}.
\end{proof}

\begin{corollary}
The category of geometric partial comodules has a generator.
\end{corollary}

\begin{proof}
Let $I$ be the set of isomorphism classes of finite dimensional geometric partial comodules over $H$. Since there exists clearly only a set of partial comodule structures over a given finite dimensional vector space, it follows that $I$ is indeed a set.
For any $i\in I$ choose one comodule $M_i$ and denote by $G$ the coproduct $\coprod_{i\in I}M_i$. Then the fundamental theorem implies there is a surjective morphism $G\to X$ for any geometric partial comodule. Hence, $G$ is a generator for $\gPMod^H$.
\end{proof}

\begin{corollary}
The category of geometric partial comodules is complete and cocomplete.
\end{corollary}

\begin{proof}
This follows from the known fact that a cocomplete well-copowered category with a generator is complete.
\end{proof}

\begin{remark}
Although the category $\PMod^H$ is complete, its limits are not preserved by the forgetful functor $U$ to $\Vect$. 
More precisely, if $L$ is a limit of a diagram $D$ in $\PMod^H$, then it is clear that $U(L)$ is a cone for the diagram $U(D)$ in $\Vect$. Hence there is a morphism $u:U(L)\to L'$ in $\Vect$ where $L'$ is the limit in $\Vect$ of $U(D)$. In general however, this morphism $u$ is not a bijection. Rather, $L$ can be understood as the biggest partial comodule inside $L'$ that allows a cone on $D$. Remark however, that in order to be able to speak about the `biggest' partial comodule inside $L'$, we already use implicitly the existence of limits in $\PMod^H$. 
This can be seen more explicitly by considering the kernel of a morphism $f:X\to Y$ in $\PMod^H$ which can be understood as the biggest partial subcomodule $K$ of $X$ such that $U(K)$ is contained in the vector space kernel of $f$. Thanks to the completeness and cocompleteness of $\PMod^H$, we can construct from two partial subcomodules $v:V\to X$ and $w:W\to X$ the pushout of the pullback of $v$ and $w$, which is then a partial subcomodule of $X$ containing both $V$ and $W$.
\end{remark}

The following result will be important in the next section.

\begin{corollary}
The forgetful functor $\gPMod^H\to \PCD$ ($\PCD$ denotes the category of partial comodule data), is fully faithful and has a left adjoint $B$.
\end{corollary}

\begin{proof}
Let $X$ be a partial comodule datum. Then $BX$ is the biggest partial subcomodule of $X$ which is geometric. As we explained in the previous remark this construction makes sense. It is easily verified that this provides a left adjoint to the forgetful functor.
\end{proof}

In the remaining part of this section, we will show that the category of Partial modules is not abelian. To this end, we will construct an example of a morphism $f$ such that $\ker\coker f$ and $\coker \ker f$ are not isomorphic.

Consider a global comodule $X$ and $Y$ a linear subspace of $X$ which is not a (global) subcomodule (recall that by \coref{subglobal} any subcomodule of global module is global). We can then construct the induced partial comodule $X/Y$ as in \exref{quotient} and consider the canonical projection $p:X\to X/Y$ which is a morphism of partial comodules. Then the vector space kernel of $p$ is just $Y$. However as we assumed that $Y$ was not a subcomodule of $X$, $Y$ is also not a partial subcomodule of $X$ and hence it can not be the kernel of $p$ in $\PMod^H$. Rather, this kernel is the biggest (global) subcomodule of $X$ contained in $Y$. Suppose that $Y$ was a one-dimensional subspace of $X$, then it follows that the kernel of $p$ has to be $0$. Then $p$ is both a monomorphism (as morphisms with a zero kernel in additive categories are monomorphisms) and an epimorphism (as $p$ is surjective and \coref{surjepi}) but not an isomorphism. Hence $\PMod^H$ is not abelian. We also see as mentioned earlier that there exist monomorphisms in $PMod^H$ whose underlying map in $\Vect$ is not injective.

\section{Partial comodules over a bialgebra}\selabel{bialgebra}

\subsection{Lax monoidal categories}

A category $\Cc$ is called {\em lax monoidal} \cite{Leinster} if
\begin{itemize}
\item  for each $n\in\NN$ there exists an $n$-fold tensor functor 
\[\ot_n:\underbrace{\Cc\times \cdots\times \Cc}_n\to \Cc;\]
\item for each for each $(k_1,\ldots,k_n)\in\NN^n$, there exists a natural transformation 
\[\gamma^{k_1,\ldots,k_n}:\ot_n\circ (\ot_{k_1}\times \ldots\times \ot_{k_n})\to \ot_{k_1+\ldots +k_n}\]
\item there exists natural transformation \[\iota:id_\Cc\to \ot_1,\]
\end{itemize}
that satisfy the following associativity and unitality conditions.
\[
\xy
(40,0)*{\ot_n\circ (\ot_{k_1}\times \ldots\times \ot_{k_n})\circ ((\ot_{\ell_{11}}\times \ldots \times\ot_{\ell_{1k_1}})\times \ldots \times (\ot_{\ell_{n1}}\times \ldots \times \ot_{\ell_{nk_n}})}="1";
(0,-15)*{\ot_{k_1+ \ldots +k_n}\circ ((\ot_{\ell_{11}}\times \ldots \times \ot_{\ell_{1k_1}})\times \ldots \times (\ot_{\ell_{n1}}\times \ldots \times \ot_{\ell_{nk_n}})}="2"; 
(80,-25)*{\ot_n \circ (\ot_{\ell_{11}+\ldots +\ell_{1k_1}}\times\ldots \times \ot_{\ell_{n1}+\ldots +\ell{nk_n}})} ="3";
(40,-40)*{\ot_{\ell_{11}+\ldots +\ell_{1k_1}+\ldots +\ell_{n1}+\ldots +\ell{nk_n}}}="4";
{\ar_{\gamma^{k_1, \ldots ,k_n} * id} "1";"2"}
{\ar^{\qquad id * (\gamma^{\ell_{11},\ldots ,\ell_{1k_1}}\times\ldots \times \gamma^{\ell_{n1},\ldots ,\ell_{nk_n}})} "1";"3"}
{\ar_{\gamma^{\ell_{11},\ldots ,\ell_{1k_1},\ldots,\ell_{n1},\ldots ,\ell_{nk_n}} } "2";"4"}
{\ar^{\qquad\gamma^{\ell_{11}+\ldots +\ell_{1k_1},\ldots,\ell_{n1}+\ldots +\ell_{nk_n}}} "3";"4"}
\endxy
\]
\[
\xymatrix{
\ot_n \ar[rr]^-{id_{\ot_n}*(\iota,\ldots,\iota)}\ar@{=}[drr] && \ot_n\circ (\ot_1,\ldots,\ot_1) \ar[d]^{\gamma^{1,\ldots,1}}\\
&& \ot_n
}\qquad\xymatrix{
\ot_n \ar[rr]^-{\iota*id_{\iota_n}}\ar@{=}[drr] && \ot_1\circ \ot_n \ar[d]^{\gamma^n}\\
&& \ot_n
}
\]
Remark that the last two conditions imply in particular that the functor $\ot_1$ is idempotent. 

There is an obvious notion of {\em oplax monoidal categories}, where the direction of the natural transformations $\gamma$ and $\iota$ is reversed.
If the natural transformations $\gamma$ and $\iota$ are invertible, then a lax monoidal category is just a monoidal category. 

A {\em (lax) monoidal functor} between lax monoidal categories is a functor $F:\Cc\to\Dd$ that comes equipped with natural transformations
\begin{eqnarray*}
\zeta_n:& \ot^\Dd_n F^n\to F\ot^\Cc_n:&\Cc^n\to \Dd
\end{eqnarray*}
for each $n\in\NN$, satisfying the following compatibility conditions with $\gamma$ and $\iota$
\[
\xymatrix{
\ot^\Dd_n(\ot^\Dd_{k_1}\times \ldots\times \ot^\Dd_{k_n})F^{k_1+\ldots+k_n} 
\ar[rr]^-{\gamma_\Dd^{k_1,\ldots,k_n}F^{k_1+\ldots+k_n}} 
\ar[d]_{\ot_n(\zeta_{k_1}\times \ldots\times \zeta_{k_n}) }
&& \ot^\Dd_{k_1+\ldots +k_n}F^{k_1+\ldots+k_n} \ar[dd]^{\zeta_{k_1+\ldots+k_n}}\\
\ot^\Dd_nF^n(\ot^\Cc_{k_1}\times \ldots\times \ot^\Cc_{k_n}) \ar[d]_-{\zeta_n(\ot_{k_1}\times \ldots\times \ot_{k_n})}\\
F\ot^\Cc_n(\ot^\Cc_{k_1}\times \ldots\times \ot^\Cc_{k_n}) \ar[rr]^-{F\gamma_\Cc^{k_1,\ldots,k_n}} && 
F \ot^\Cc_{k_1+\ldots +k_n}
}
\xymatrix{
F \ar[rr]^-{F\iota_\Cc} \ar[dr]_-{\iota_\Dd F} && F\ot^\Cc_1\\
& \ot^\Dd_1 F \ar[ur]_-{\zeta_1}
}
\]
Similarly, a functor $G:\Cc\to\Dd$ is called {\em opmonoidal} if there are natural transformations
\begin{eqnarray*}
\delta_n:& F\ot^\Cc_n\to \ot^\Dd_n F^n &\Cc^n\to \Dd.
\end{eqnarray*}
satisfying appropriate compatibility conditions with $\gamma$ and $\iota$.

The following result might be well-known, but as we didn't found a reference we state it and give a sketch of the proof, which is quite elementary, but because of notational problems becomes quite technical. This result allows to construct many lax monoidal categories. 

\begin{theorem}\thlabel{laxadjoint}
\begin{enumerate}[(i)]
\item
Let $\Dd$ be a lax monoidal category and consider a pair $(L,R)$ of adjoint functors 
\[
\xymatrix{
\Cc \ar@<.5ex>[rr]^-{L} &&\Dd \ar@<.5ex>[ll]^-{R}
}
\]
then $\Cc$ is also a lax monoidal category such that $R$ is a monoidal functor and $L$ is an opmonoidal functor. 
\item 
Let $\Dd$ be an oplax monoidal category and consider a pair $(L,R)$ of adjoint functors 
\[
\xymatrix{
\Dd \ar@<.5ex>[rr]^-{L} &&\Ee \ar@<.5ex>[ll]^-{R}
}
\]
then $\Ee$ is also an oplax monoidal category such that $R$ is a monoidal functor and $L$ is an opmonoidal functor.
\end{enumerate}
\end{theorem}

\begin{proof}
We only give a sketch of the proof of part (i), the second follows by duality.

Denote the $n$-fold tensor products in $\Dd$ by $\ot^\Dd_n$ and its associativity and unity constraints by $\gamma_\Dd$ and $\iota_\Dd$. 

For any $n$-tuple $(c_1,\ldots,c_n)$ of objects in $\Cc$, define the $n$-fold tensor product in $\Cc$ as
\[\ot^\Cc_n(c_1,\ldots,c_n)=R(\ot^\Dd_n(Lc_1,\ldots,Lc_n)).\]
More precisely, the $n$-fold tensor product in $\Cc$ is defined as the following composition of functors
\[\ot^\Cc_n=R\circ \ot^\Dd_n \circ L^n:\Cc^n\to \Cc.\]

Let us denote by $\eta:id_\Cc\to RL$ and $\epsilon:LR\to id_\Dd$ the unit and counit of the adjunction $(L,R)$. For any $n$-tuple $(c_1,\ldots,c_n)$ in $\Cc$, we can consider the morphism
\begin{eqnarray}\eqlabel{Lopmonoidal}
&&\delta_n^{c_1,\ldots,c_n}=\epsilon_{\ot_n(Lc_1,\ldots,Lc_n)}:\\
&&\hspace{2cm}L\ot^\Cc_n(c_1,\ldots,c_n)=LR\ot^\Dd_n(Lc_1,\ldots,Lc_n)\to \ot^\Dd_n(Lc_1,\ldots,Lc_n)\nonumber
\end{eqnarray}
which is natural in each of the entries $c_i$, defining in this way  for each $n\in\NN$ a natural transformation 
\[\delta_n=\epsilon\ot^\Dd_nL^n:L\ot^\Cc_n=LR\ot^\Dd_nL\to \ot^\Dd_nL^n:\Cc^n\to \Dd.\] 
Similarly, for each $n$-tuple $(d_1,\ldots,d_n)$ in $\Dd$ we put
\begin{eqnarray}\eqlabel{Rmonoidal}
&&\zeta_n^{d_1,\ldots,d_n}=R\ot^\Dd_n(\epsilon_{d_1},\ldots,\epsilon_{d_n}):\\
&&\hspace{2cm}R\ot^\Dd_n(LRd_1,\ldots LRd_n)=\ot^\Cc_n(Rd_1,\ldots Rd_n) \to R\ot^\Dd(d_1,\ldots,d_n)\nonumber
\end{eqnarray}
which defines a natural transformation 
\[\zeta_n=R\ot^\Dd_n\epsilon^n:\ot^\Cc_nR^n =R\ot^\Dd_n(LR)^n\to R\ot^\Dd_n:\Dd^n\to \Cc.\]

To define the associativity constraint of $\Cc$, first remark that 
\begin{eqnarray*}
\ot_n\circ (\ot_{k_1}\times \ldots\times \ot_{k_n})
&=&(R\ot^\Dd_n L^n)\circ ((R\ot^\Dd_{k_1} L^{k_1})\times \ldots\times (R\ot^\Dd_{k_n} L^{k_n}))\\
&=&R\ot^\Dd_n (LR)^n(\ot^\Dd_{k_1}\times \ldots\times \ot^\Dd_{k_n})L^{k_1+\ldots+k_n}
\end{eqnarray*}
We now define the associativity constraint $\gamma_\Cc$ of $\Cc$ as the following composition
\[
\xy
(0,0)*{R\ot^\Dd_n (LR)^n(\ot^\Dd_{k_1}\times \ldots\times \ot^\Dd_{k_n})L^{k_1+\ldots+k_n}}="1";
(90,0)*{R(\ot^\Dd_{k_1+\ldots+k_n}))L^{k_1+\ldots+k_n} = \ot^\Cc_{k_1+\ldots+k_n}}="3";
(45,-20)*{R\ot^\Dd_n(\ot^\Dd_{k_1}\times \ldots\times \ot^\Dd_{k_n})L^{k_1+\ldots+k_n}}="2";
{\ar_-{(\zeta_n(\ot^\Dd_{k_1}\times \ldots\times \ot^\Dd_{k_n})L^{k_1+\ldots+k_n})\qquad } "1";"2"}
{\ar_-{\qquad(R\gamma_\Dd^{k_1,\ldots,k_n}L^{k_1+\ldots+k_n})} "2";"3"}
{\ar^{\gamma_\Cc^{k_1,\ldots,k_n}} "1";"3"}
\endxy
\]

The unitality constraint of $\Cc$ is defined as the composition
\[\iota_\Cc=(R\iota_\Dd L)\circ \eta:id_\Cc\to R\ot_1^\Dd L=\ot^\Cc_1.\]

The associativity conditions for the lax monoidal structure on $\Cc$ then follow directly from the naturality of $\epsilon$ and the associativity in $\Dd$. The unitality conditions in $\Cc$ follow from the unit-counit condition of the adjunction $(L,R)$ and the unitality conditions in $\Dd$.

The monoidal structure on the functor $R$ is given by \equref{Rmonoidal}, the op-monoidal structure on $L$ is given by \equref{Lopmonoidal}.
\end{proof}

The previous proposition can be applied in particular to a monoidal category category $\Dd$ and allows to produce in this way many natural examples of (op)lax monoidal categories.

As an intermediate notion between lax monoidal categories and monoidal categories, one can consider a {\em monoidal category with lax unit}. This is a category $\Cc$ endowed with a monoidal tensor product $\ot:\Cc\times\Cc\to \Cc$, endowed with an associativity constraint
\[\alpha_{C,C',C''}:(C\ot C')\ot C''\to C\ot (C'\ot C'')\]
which is a natural isomorphism that satisfies the usual pentagon condition. 
A {\em lax unit} $I$ for such an associative tensor product is an object $I$ in $\Cc$ such that for any $C\in\Cc$ there are natural transformations
\[\ell_C:I\ot C\to C, \quad r:C\ot I\to C\]
satisfying the usual compatibility constraints with $\alpha$: 
\[
\xymatrix{
(C\ot I)\ot C' \ar[dr] \ar[rr]^{\alpha_{C,I,C'}} && C\ot (I\ot C') \ar[dl]\\
& C\ot C'
}
\]
The following is now an easy observation.

\begin{lemma}
If $(\Cc,\ot,I)$ is a monoidal category with lax unit, then $\Cc$ is a lax monoidal category by defining
\begin{itemize}
\item $\ot_0=I$, $\ot_1=id_\Cc$, $\ot_2=\ot$;
\item for all $n>2$, $\ot_n= \ot\circ(id \times \ot_{n-1})$;
\item $\iota=id:id_\Cc\to \ot_1$;
\item for all  $(k_1,\ldots,k_n)\in \NN_0^n$, 
$\gamma^{k_1,\ldots,k_n}$ is canonically obtained from combinations of $\alpha$ and identities and are therefore invertible;
\item for any $(k_1,\ldots,k_n)\in \NN^n$, where $k_{i_1}=\ldots=k_{i_m}=0$ ($m<n$), $\gamma^{k_1,\ldots,k_n}$ is canonically obtained from combinations of $\ell$, $r$, $\alpha$ and identities and are not invertible;
\end{itemize}
\end{lemma}

Similarly, one introduces the notion of a monoidal category with an oplax unit, which gives rise to an oplax monoidal category.

\subsection{The lax monoidal category of geometric partial comodules over a bialgebra}\selabel{bialg}

The following result is essentially due to Johnstone \cite{Js}, who formulated the proof in case of cartesian closed categories, but the argument easily generalizes to closed monoidal categories.

Let us recall first that a monoidal category is called {\em left closed monoidal} if for each object $X$ in $\Cc$, the endofunctor $X\ot-:\Cc\to\Cc$ has a right adjoint, that we denote by $[X,-]$ and that is called the internal hom. In other words, if $\Cc$ is right closed, then for any triple of objects $X,Y,Z$ in $\Cc$ we have isomorphisms
\[\Hom_\Cc(X\ot Y,Z)\cong \Hom_\Cc(Y,[X,Z])\]
for any $f\in \Hom_\Cc(X\ot Y,Z)$ we denote the corresponding element in $\Hom_\Cc(Y,[X,Z])$ by $\hat f$, and conversely for any $g\in\Hom_\Cc(Y,[X,Z])$, we have $\hat g\in \Hom_\Cc(X\ot Y,Z)$ with $\hat{\hat f}=f$ and $\hat{\hat g}=g$. If one considers the evaluation and coevaluation maps
\[\ev_Y^X:X\ot [X,Y]\to Y;\quad \coev_Y^X:Y\to [X,X\ot Y],\]
then we can write 
\begin{eqnarray*}
\hat f&=&[X,f]\circ \coev_Y^X;\\
\hat g&=&\ev_Z^X\circ (X\ot f).
\end{eqnarray*}
Given any morphism $f:X\to X'$, we can consider the composition 
\[\xymatrix{
X\ot [X',Y] \ar[rr]^{f\ot id} && X'\ot [X',Y] \ar[rr]^{\ev_Y^{X'}} && Y}.\]
We denote the corresponding morphism $[X',Y]\to [X,Y]$ by $[f,Y]$. It is well-known that in this way, $[-,Y]:\Cc\to \Cc$ becomes a contravariant functor for any object $Y$ in $\Cc$.

Suppose that $\Cc$ is moreover right closed, and where the right internal hom denoted by $\{-,-\}$. Then we find for any three objects $X,Y,Z$ in $\Cc$ that
\[\Hom_\Cc(Y,[X,Z])\cong \Hom_\Cc(X\ot Y,Z)\cong \Hom_\Cc(X,\{Y,Z\})\]
Hence 
\[\Hom_{\Cc^{op}}([X,Z],Y)\cong \Hom_\Cc(X,\{Y,Z\})\]
and the (contravariant) functor $[-,Z]:\Cc\to \Cc^{op}$ has a right adjoint $\{-,Z\}$, and therefore $[-,Z]:\Cc\to\Cc$ sends epimorphisms to monomorphisms.

\begin{lemma}\lelabel{Js}
Let $\Cc$ be a bi-closed monoidal category, $f:A\to B$ and epimorphism and $g:C\to D$ a regular epimorphism.  Then the pushout of the pair $(f\ot C, A\ot g)$ is given by $(B\ot D,B\ot g, f\ot D)$.
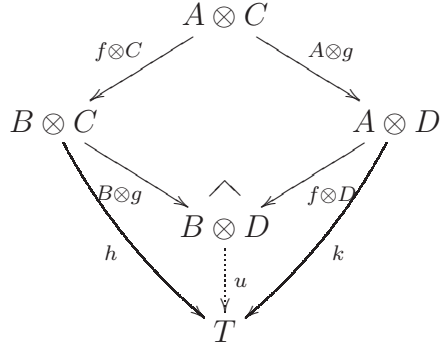
\[
\xymatrix{
& A\ot C \ar[dl]_{f\ot C} \ar[dr]^{A\ot g} \\
B\ot C \ar[dr]_{B\ot g} \ar@/_/[ddr]_h && A\ot D \ar[dl]^{f\ot D} \ar@/^/[ddl]^k\\
& B\ot D \ar@{}[u]|<<<{\pushout} \ar@{.>}[d]^u \\
& T
}
\]
\end{lemma}

\begin{proof}
Suppose that $g$ is the coequalizer of the pair $r,s:R\to C$. Consider any object $T$ and maps $h:B\ot C\to T$, $k:A\ot D\to T$ such that $\ell=h\circ (f\ot C)=k\circ (A\ot g):A\ot C\to T$. 
Using the left closure on $\Cc$, we find that $\ell$ corresponds uniquely to a morphism $\hat \ell\in \Hom(C,[A,T])$ and one checks that
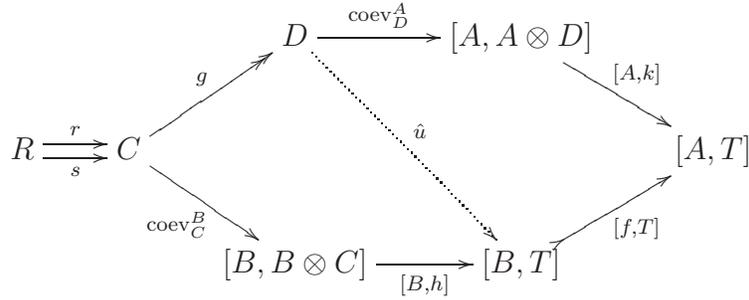
\[\hat \ell=[A,k]\circ \coev_D^A\circ g=[f,T]\circ [B,h]\circ \coev^B_C.\]
\[
\xymatrix{
&& D\ar[r]^-{\coev^A_D} \ar@{.>}[ddr]^-{\hat u} & [A,A\ot D] \ar[dr]^{[A,k]}\\
R\ar@<.5ex>[r]^r \ar@<-.5ex>[r]_s &C\ar@{->>}[ur]^g \ar[dr]_-{\coev^B_C} &&& [A,T]\\
&& [B,B\ot C] \ar[r]_-{[B,h]} & [B,T] \ar@{>->}[ur]_-{[f,T]}
}
\]
Since $g$ coequalizes the pair $(r,s)$, it follows from the first equality that $\hat\ell$ also coequalizes the pair $(r,s)$. Furthermore, since $f$ is an epimorphism, $[f,T]$ is a monomorphism and we find that
\[[B,h]\circ \coev^B_C\circ r=[B,h]\circ \coev^B_C \circ s.\]
Therefore, the universal property of the coequalizer $g$ leads to a unique morphism
$\hat u:D\to [B,T]$ such that 
\[\hat u\circ g= [B,h]\circ \coev^B_C:C\to [B,T].\]
Moreover, since $g$ is an epimorphism, $\hat u$ also satisfies
\[[f,T]\circ \hat u=[A,k]\circ \coev^A_D\]
Consequently the induced morphism $\hat{\hat u}=u:B\ot D\to T$ satisfies
\[u\circ (B\ot g)=h,\quad u\circ (f\ot D)=k\]
and is unique in this sense, which proves the universal property of the pushout $(B\ot D,B\ot g, f\ot D)$.
\end{proof}

Let $\Cc$ be a braided monoidal category with pushouts and consider be a bialgebra $H$ in $\Cc$. 
Let $(X,X\bul H,\pi_X,\rho_X)$ and $(Y,Y\bul H,\pi_Y,\rho_Y)$ be two partial comodule data over $H$. 
Then we can construct a new partial comodule datum $(X\ot Y,(X\ot Y)\bul H,\pi_{X\ot Y},\rho_{X\ot Y})$ in the following way. Consider the map  $\mu_{X, Y}=(X\ot Y\ot \mu_H)\circ (X\ot \sigma_{H,Y}\ot H)$, where $\sigma$ denotes the braiding of the category. Then define $(X\ot Y)\bul H$ and $\pi_{X\ot Y}$ by the following pushout.
\[
\xymatrix{
& X\ot H\ot Y\ot H \ar[dr]^{\mu_{X, Y}} \ar@{->>}[dl]_{\pi_X\ot \pi_Y}\\
(X\bul H)\ot (Y\bul H) \ar[dr]_{\ol\mu} && X\ot Y\ot H \ar@{->>}[dl]^{\pi_{X\ot Y}}\\
& (X\ot  Y)\bul H \ar@{}[u]|<<<{\pushout}
}
\]
By taking $\rho_{X\ot Y}=\ol\mu\circ(\rho_X\ot\rho_Y)$, we obtain the desired partial comodule datum.

This construction lead to the following result.

\begin{proposition}\prlabel{PCDmonoidal}
Let $\Cc$ be a braided closed monoidal category where all epimorphisms are regular and let $H$ be a bialgebra in $\Cc$.

Then by the above defined tensor product, the category of partial comodule data over $H$ is a monoidal category with an op-lax unit, such that the following is a diagram of strict monoidal functors.
\[
\xymatrix{
\Mod^H \ar[rr] \ar[dr] && \PCD^H \ar[dl]\\
& \Cc
}
\]
\end{proposition}

\begin{proof}
Let us first verify the associativity of the defined tensor product for $\PCD$. Consider 3 partial comodule data $\ul X,\ul Y,\ul Z$ and
consider the following diagram. 
\[
\xymatrix{
&&&& \hbox to 3em{\hss$ X\ot H\ot Y\ot H \ot Z\ot H$\hss} \ar[drr]^{\mu_{X, Y}\ot Z\ot H} \ar@{->>}[dll]_{\pi_X\ot \pi_Y\ot Z\ot H}\\
&&\hbox to 3em{\hss$(X\bul H)\ot (Y\bul H)\ot Z\ot H$\hss} \ar[drr]_{\ol\mu\ot Z\ot H} \ar@{->>}[dll]_{X\bul H\ot Y\bul H\ot \pi_Z} &&&& \hbox to 3em{\hss$X\ot Y\ot H\ot Z\ot H$\hss} \ar@{->>}[dll]^{\pi_{X\ot Y}\ot Z\ot H} \ar[drr]^{\mu_{(X\ot Y), Z}}\\
\hbox to 3em{\hss$(X\bul H)\ot (Y\bul H)\ot (Z\bul H)$\hss} \ar[drr]_{\ol\mu\ot (Z\bul H)}
&&&& \hbox to 3em{\hss$(X\ot  Y)\bul H \ot Z\ot H$\hss} \ar@{}[u]|<<<{\pushout} \ar[dll]^{((X\ot Y)\bul H)\ot\pi_Z} 
&&&& \hbox to 3em{\hss$X\ot Y\ot Z\ot H$\hss} \ar[ddllll]\\
&& \hbox to 3em{\hss$((X\ot Y)\bul H)\ot (Z\bul H)$\hss} \ar@{}[u]|<<<{\pushout} \ar[drr] && \\
&&&& \hbox to 3em{\hss$((X\ot Y)\ot Z)\bul H$\hss} \ar@{}[u]|<<<{\pushout}
}
\]
The upper square is a pushout by definition of the tensor product and the fact that the functor $-\ot Z\ot H$ preserves pushouts since it has a right adjoint. The down square is a pushout by definition of the tensor product. The left square is a pushout by the \leref{Js}. hence, by combining these pushouts we find that 
$((X\ot Y)\ot Z)\bul H$ is the pushout of $(X\bul H\ot Y\bul H\ot\pi_Z)\circ(\pi_X\ot\pi_Y\ot Z)\simeq\pi_X\ot\pi_Y\ot\pi_Z$ along $\mu_{(X\ot Y), Z}\circ (\mu_{X,Y}\ot Z\ot H)$.
In the same way, $(X\ot (Y\ot Z))\bul H$ is shown to be the pushout of $\pi_X\ot\pi_Y\ot\pi_Z$ and $\mu_{X, Y\ot Z}\circ (X\ot H\ot \mu_{Y, Z})$. One can easily verify that by the properties of the braiding in $\Cc$ and associativity of the multiplication of $H$, the maps $\mu_{(X\ot Y), Z}\circ (\mu_{X,Y}\ot Z\ot H)$ and $\mu_{X, Y\ot Z}\circ (X\ot H\ot \mu_{Y, Z})$ are the same up-to-isomorphism. Hence we find that $((X\ot Y)\ot Z)\bul H\cong (X\ot (Y\ot Z))\bul H$, which induces the associativity constraint for the monoidal product in $\PCD$.

Next, let us verify that the partial comodule datum $\ul k=(k,H,id_H,\eta)$ is an oplax unit for this monoidal product. Consider any partial comodule datum $\ul X=(X,X\bul H,\pi_X,\rho_X)$ and construct the tensor product $\ul X\ot \ul k$. We know that the underlying $\Cc$ object is just $X\ot k\cong X$ via the isomorphism $r_X:X\to X\ot k$. Futhermore, the object $(X\ot k)\bul H$ is constructed by the following pushout.
\[
\xymatrix{
& X\ot H \ot H \ar[dr]^{X\ot \mu } \ar@{->>}[dl]_{\pi_X\ot H}\\
(X\bul H)\ot H \ar[dr]_{\ol\mu} && X\ot H \ar@{->>}[dl]^{\pi_{X\ot k}}\\
& (X\ot  k)\bul H \ar@{}[u]|<<<{\pushout}
}
\]
Then consider the map $r_X\bul H:=\ol\mu\circ (X\bul H)\ot\eta:X\bul H\to (X\ot  k)\bul H$. One easily verifies that $r_X\bul H\circ \pi_X=\pi_{X\ot k}$ and therefore $(r_X,r_X\bul H):\ul X\to \ul X\ot \ul k$ is a morphism of partial comodule data. 
\end{proof}

Let us remark that the category of partial comodule data can not have a (strong) monoidal unit. 
Indeed, since the forgetful functor is strict monoidal, the underlying $\Cc$-object of the monoidal unit needs to be the monoidal unit $k$ of $\Cc$. Hence the monoidal unit should be of the form $(k,K,\pi,\rho)$, where $K$ is a quotient of $H$. However, the pushout
\[
\xymatrix{
& X\ot H \ot H \ar[dr]^{X\ot \mu } \ar@{->>}[dl]_{\pi_X\ot \pi}\\
(X\bul H)\ot K \ar[dr]_{\ol\mu} && X\ot H \ar@{->>}[dl]^{\pi_{X\ot k}}\\
& (X\ot  k)\bul H \ar@{}[u]|<<<{\pushout}
}
\]
can be computed explicitly in the case $\Cc=\Vect$ via \leref{pushoutvect}
and in this case we see that $(X\ot  k)\bul H$ is the quotient of $X\ot H$ with respect to $(X\ot \mu)(\ker \pi_X \ot H + X\ot H\ot\ker \pi)$. However, this last set is strictly larger then $\ker\pi_X$, so we can never get that $(X\ot  k)\bul H\cong X\bul H$. Nevertheless, the oplax monoidal unit from \prref{PCDmonoidal} becomes a strong unit for a suitable subcategory that we will define now.

\begin{definition}
We call a partial comodule datum $X$ over a bialgebra $H$ {\em left equivariant} if the morphism $\pi_X:X\ot H\to X\bul H$ is left $H$-linear by the free $H$-action on $X\ot H$. More explicitly, in case $\Cc=\Vect_k$, this means that if $x\ot h\in\ker\pi_X$, then also $x\ot h'h\in\ker\pi_X$ for all $h'\in H$. In the same way, one defines {\em right equivariant} partial comodule data. When $X$ is both left and right equivariant, we simply say it is equivariant.
\end{definition}

\begin{example}
Consider the \exref{CJ} of a Caenepeel-Janssen partial action. Since $A\bul H=(A\ot H)e$ it is a left $A\ot H$-module, hence in particular left equivariant.
\end{example}

\begin{corollary}
The oplax monoidal unit of $\PCD$ is a left monoidal unit for all left equivariant partial comodule data, and a right monoidal unit for all right equivariant partial comodule data.

Consequently, the category of equivariant partial comodule data over a $k$-bialgebra $H$ is monoidal.
\end{corollary}

\begin{proof}
Consider the pushout
\[
\xymatrix{
& X\ot H \ot H \ar[dr]^{X\ot \mu } \ar@{->>}[dl]_{\pi_X\ot H}\\
(X\bul H)\ot H \ar[dr]_{p_1} && X\ot H \ar@{->>}[dl]^{p_2}\\
& P \ar@{}[u]|<<<{\pushout}
}
\]
By \leref{pushoutvect}, we know that $P$ can be computed as the quotient of $X\ot H$ by the subspace $(X\ot \mu)(\ker\pi_X)$. However, under the stated assumptions, $\ker\pi_X$ is a right $H$-submodule of $X\ot H$ and hence $(X\ot \mu)(\ker\pi_X)=\ker\pi_X$. Therefore, $P$ is exactly $X\bul H$ and $p_2$ is nothing else than $\pi_X$.

From this observation and the definition of the tensor product, it now follows that $X\ot k\cong X$ in $\PCD$. Similarly, using that $\ker\pi_X$ is a left $H$-submodule of $X\ot H$, we find that $k$ is a left unit for the monoidal structure on $\PCD$.
\end{proof}

The result from \prref{PCDmonoidal} leads to the natural question whether the full subcategories of the category of partial comodule data, consisting of quasi, lax and geometric partial comodules inherit a monoidal structure. To answer this question, let us compute the pushout $(X\ot Y)\bul(H\bul H)$, which is given by the following composition of pushouts
\[
\xymatrix{
X\ot H\ot Y\ot H \ar[d]_{\pi_X\ot \pi_Y} \ar[r]^-{\mu_{X,Y}} & X\ot Y\ot H \ar[r]^-{X\ot Y\ot\Delta} \ar[d]^{\pi_{X\ot Y}} & X\ot Y\ot H\ot H \ar[r]^-{\pi_{X\ot Y}\ot H} \ar[d]^{\pi_{X\ot Y,\Delta}} & ((X\ot Y)\bul H)\ot H \ar[d]^{\pi'_{X\ot Y,\Delta}}\\
(X\bul H)\ot (Y\bul H) \ar[r]_-{\ol \mu_{X,Y}} & (X\ot Y)\bul H \ar[r]_-{(X\ot Y)\bul\Delta} \ar@{}[ul]|<<<{\pushoutv}
& (X\ot Y)\bul(H\ot H) \ar[r]_-{\pi'_{X\ot Y}} \ar@{}[ul]|<<<{\pushoutv}& (X\ot Y)\bul (H\bul H)\ar@{}[ul]|<<<{\pushoutv}
}
\]
Furthermore, checks that the composition of the upper morphisms can be rewritten as
\begin{eqnarray*}
&&(\pi_{X\ot Y}\ot H)\circ (X\ot Y\ot\Delta)\circ \mu_{X,Y}= \\
&&\hspace{2cm}(\ol\mu_{X,Y}\ot H) \circ \mu_{X\bul H,Y\bul H} \circ (\pi_X\ot H\ot\pi_Y\ot H)\circ (X\ot\Delta\ot Y\ot\Delta)
\end{eqnarray*}
In the same way, the pushout $((X\ot Y)\bul H)\bul H$ is given by the following composition of pushouts
\[
\xymatrix{
X\ot H\ot Y\ot H \ar[d]_{\pi_X\ot \pi_Y} \ar[r]^-{\mu_{X,Y}} 
& X\ot Y\ot H \ar[r]^-{\rho_X\ot \rho_Y\ot H} \ar[d]^{\pi_{X\ot Y}} & 
(X\bul H)\ot (Y\bul H)\ot H \ar[r]^-{\ol\mu_{X,Y}\ot H}  & ((X\ot Y)\bul H)\ot H \ar[d]^{\pi_{(X\ot Y)\bul H}}\\
(X\bul H)\ot (Y\bul H) \ar[r]_-{\ol \mu_{X,Y}} & (X\ot Y)\bul H \ar[rr]^{\rho_{X\ot Y}\bul H} 
&& ((X\ot Y)\bul H)\bul H\ar@{}[ul]|<<<{\pushoutv}
}
\]
and again we can rewrite the composition of the upper morphisms
\begin{eqnarray*}
&&(\ol\mu_{X,Y}\ot H)\circ (\rho_X\ot \rho_Y\ot H)\circ \mu_{X,Y} =\\
&&\hspace{2cm} (\ol\mu_{X,Y}\ot H)\circ \mu_{X\bul H,Y\bul H}\circ (\rho_X\ot H\ot \rho_Y\ot H)
\end{eqnarray*}
Hence, to study the relation between $(X\ot Y)\bul(H\bul H)$ and $((X\ot Y)\bul H)\bul H$, we have to compare the following pushouts
\[
\hspace{-1cm}
\xymatrix{
X\ot H\ot Y\ot H \ar[d]_{\pi_X\ot \pi_Y} \ar[rr]^-{X\ot\Delta\ot Y\ot\Delta} 
&& X\ot H\ot H\ot Y\ot H\ot H \ar[rr]^-{\pi_X\ot H\ot\pi_Y\ot H} 
& &
(X\bul H)\ot H\ot (Y\bul H)\ot H \ar[d]^{p_2}\\
(X\bul H)\ot (Y\bul H) \ar[rrrr]^{p_1}
&&&& P
 \ar@{}[ul]|<<<{\pushoutv}
}
\]
and 
\[
\xymatrix{
X\ot H\ot Y\ot H \ar[d]_{\pi_X\ot \pi_Y} \ar[rr]^-{\rho_X\ot H\ot \rho_Y\ot H} 
&&
(X\bul H)\ot H\ot (Y\bul H)\ot H \ar[d]^{q_2}\\
(X\bul H)\ot (Y\bul H) \ar[rr]^{q_1}
&& Q
 \ar@{}[ul]|<<<{\pushoutv}
}
\]
Using \leref{pushoutvect}, we find that $P$ and $Q$ are isomorphic to quotients of $(X\bul H)\ot H\ot (Y\bul H)\ot H$ by the respective subspaces
\begin{eqnarray*}
&&(\pi_X\ot H\ot\pi_Y\ot H)\circ (X\ot\Delta\ot Y\ot\Delta)(\ker(\pi_X\ot\pi_Y))=\\
&&\hspace{2cm}(\pi_X\ot H\ot\pi_Y\ot H)\circ (X\ot\Delta\ot Y\ot\Delta)(\ker \pi_X\ot Y\ot H+X\ot H\ot \ker\pi_Y)
\end{eqnarray*}
and
\begin{eqnarray*}
&&(\rho_X\ot \rho_Y)(\ker(\pi_X\ot\pi_Y))=(\rho_X\ot \rho_Y)(\ker \pi_X\ot Y\ot H+X\ot H\ot \ker\pi_Y)
\end{eqnarray*}
Since in general $(\pi_X\ot H)\circ (X\ot \Delta)(X\ot H)\not\cong (\rho_X\ot H)(X\ot H)$, we find that $P$ and $Q$ are non-isomorphic, and therefore also $(X\ot Y)\bul(H\bul H)$ and $((X\ot Y)\bul H)\bul H$ are non-isomorphic (even if $X$ and $Y$ are geometric partial comodules). 

Hence, we can conclude that when $X$ and $Y$ are geometric (respectively lax) partial comodules, then $X\ot Y$ is in general no longer a geometric (respectively lax) partial comodule. 
However, when $X$ and $Y$ are both quasi comodules, then $\theta_1\circ (\rho_X\bul H)$ and $\theta_2\circ \ol{(X\bul \Delta)}$ {\em do} have identical images when restricted to the image of $\rho_X$, we find that $X\ot Y$ is still a quasi partial comodule. We can then conclude on the following.

\begin{theorem}
Let $\Cc$ be a category satisfying the conditions of \prref{PCDmonoidal}.
The category of quasi partial comodules over a bialgebra in $\Cc$ is monoidal with an oplax monoidal unit, and the forgetful functor to $\Cc$ is strict monoidal.  

The category of equivariant quasi partial comodules over a bialgebra is a monoidal category.
\end{theorem}

Although the above introduced tensor product is not well-defined on the category geometric partial comodules, in case work with a bialgebra over a field, we can combine \prref{PCDmonoidal} with \thref{laxadjoint} and obtain immediately the following result.

\begin{theorem}
The category of geometric partial comodules over a bialgebra $H$ over a field $k$ is an op-lax monoidal category and the forgetful functor $U:\gPMod^H\to \Vect_k$ is monoidal. 
\end{theorem}

\begin{remark}
Let us describe the oplax tensor product of the category $\gPMod^H$ a bit more explicitly.
Consider two geometric partial comodules $M$ and $N$, and let $M\ot N$ be the tensor product partial comodule datum (which we know is a quasi partial comodule). Then $M\bul N$ is the geometric partial comodule that is uniquely defined by the following universal property. There exists a morphism $p:M\ot N\to M\bul N$ and for every other geometric partial comodule $T$ with a morphism $M\ot N\to T$, there exists a unique morphism $u: M\bul N\to T$ such that $t=u\circ p$.

Since the zero module is a geometric partial comodule that is a minimal solution for the above problem, $M\bul N$ will always exist. 
Given two $p:M\ot N\to P$ and $q:M\ot N\to Q$, Let $R$ be the pullback of the pushout of $p$ and $q$ (which exist since we proved that geometric partial comodules are bi-complete). Then there is a unique morphism $M\ot N\to R$ compatible with both $p$ and $q$. In this way, we can construct the ``biggest'' quotient $M\bul N$ of $M\ot N$ that is still a geometric partial comodule.

Remark that if one of the geometric partial comodules $X$ and $Y$ is global, then $X\bul Y=X\ot Y$. 
\end{remark}

\section{Partial comodule algebras and Hopf-Galois theory}\selabel{Galois}

In the partial case, there turn out to be two kinds of `comodule algebras': those which arise as partial comodules in the category of algebras, and those that arise as algebras in the (oplax monoidal) category of partial comodules. While these notions coincide in the global case, for partial coactions they differ, as a consequence of the fact that pushouts in the category of algebras are different from pushouts in the category of vector spaces. Throughout this section, we assume that $k$ is a field.

\subsection{Algebras in the category of partial comodules}\selabel{comodalg}

Let $\Cc$ be an oplax monoidal category and denote $\ot_0(\emptyset)=I$ and $\ot_n(X_1,\dots,X_n)=(X_1\ot \ldots \ot X_n)$. An algebra $\Cc$, is an object $A$ endowed with morphisms $m:(A\ot A)\to A$ and $u:I\to A$ satisfying the following conditions
\[
\xymatrix@!C{
& ((A)\ot (A\ot A)) \ar[r]^-{\iota\ot m} & (A\ot A)\ar[dr]^m\\
(A\ot A\ot A) \ar[ur]^{\gamma} \ar[dr]_\gamma &&& A\\
& ((A\ot A)\ot (A)) \ar[r]^-{m\ot \iota} & (A\ot A)\ar[ur]_m\\
}\]\[
\xymatrix{
& (A\ot I) \ar[r]^-{(A\ot u)} & (A\ot A)\ar[dr]^{m}\\
(A) \ar[dr]_{\gamma} \ar[ur]^{\gamma} \ar[rrr]^{\iota_A} &&& A\\
& (I\ot A) \ar[r]_-{(u\ot A)} & (A\ot A)\ar[ur]_{m}
}
\]
One can verify that just as in the case of algebras in usual strong monoidal categories, all higher associativity conditions follow form the one given here.

Then we obtain the following natural definitions.
\begin{definition}
Let $H$ be a $k$-bialgebra. 
A {\em quasi (resp.\ geometric) partial $H$-comodule algebra} is a an algebra in the oplax monoidal category of quasi (resp.\ geometric) partial $H$-comodules. 
\end{definition}

Since the forgetful functors $\gPMod^H\to\qPMod^H\to \Vect_k$ are monoidal, each geometric partial comodule  algebra is also a quasi partial comodule algebra and a quasi or geometric partial comodule algebra is also a $k$-algebra. More precisely, we have the following result.

\begin{proposition}\prlabel{comodalgisalgcomod}
Let $\Cc$ be a category satisfying the conditions of \prref{PCDmonoidal} and $H$ a Hopf algebra in $\Cc$.
If $(A,A\bul H,\pi_A,\rho_A)$ is an algebra in the monoidal category with oplax unit $\qPMod^H$, then $A$ and $A\bul H$ are algebras in $\Cc$ and the morphisms $\pi_A$ and $\rho_A$ are algebra morphisms.
\end{proposition}

\begin{proof}
We already remarked that $A$ is an algebra in $\Cc$, since the forgetful functor $\qPMod^H\to \Cc$ is monoidal. To see that $A\bul H$ is an algebra consider the following, which expresses that the multiplication $\mu:A\ot A\to A$ is a morphism of partial comodules, and the construction of $(A\ot A)\bul H$ as pushout.
\[
\xymatrix{
&& A\ot A \ar[rr]^-{\mu_A} \ar[d]^{\rho_{A\ot A}} \ar[dll]_{\rho_A\ot\rho_A} && A \ar[d]^{\rho_A}\\
(A\bul H)\ot (A\bul H) \ar[rr]^{\ol\mu_{A,A}} && (A\ot A)\bul H \ar[rr]^{\mu_A\bul H} \ar@{}[dl]|<<<{\pushoutw} && A\bul H\\
A\ot H\ot A\ot H  \ar[rr]^{\mu_{A,A}} \ar[u]_{\pi_A\ot \pi_A} && A\ot A\ot H \ar[rr]^{\mu_A\ot H} \ar[u]_{\pi_{A\ot A}} && A\ot H \ar[u]_{\pi_A}
}
\]
One can then easily verify that the morphism $(\mu_{A}\bul H)\circ \ol\mu_{A,A}$ defines an associative multiplication on $A\bul H$, and by construction $\pi_A$ and $\rho_A$ are then multiplicative.
In a similar way, the unit morphism $u:k\to A$ is a morphism of partial comodules if the following diagram commutes
\[
\xymatrix{
k \ar[d]^{\eta_H} \ar[rr]^-u && A \ar[d]^{\rho_A}\\
H \ar[rr]^{u\bul H} && A\bul H\\
H \ar@{=}[u] \ar[rr]^{u\ot H} &&  A\ot H \ar[u]_{\pi_A}
}
\]
which means in Sweedler notation that
\begin{equation}\eqlabel{rho1}
\rho_A(1_A)=1_A\bul 1_H
\end{equation}
Then $\rho_A\circ u:k\to A\bul H$ is a unit for the algebra $A\bul H$ and the morphisms $\pi_A$ and $\rho_A$ are unital.
\end{proof}

Again, since the functor $\gPMod^H\to\qPMod^H$ is monoidal, it follows that for a geometric partial comodule algebra $A$, the vector space $A\bul H$ is naturally a $k$-algebra and $pi_A$ is a $k$-algebra morhpism. This implies that $\ker\pi_A$ is a two-sided ideal in $A\ot H$.

\begin{remark}
In contrast to what might think naively, the $\Cc$-objects $(A\bul H)\bul H$, $A\bul (H\bul H)$ or $\Theta_A$ do not posses a natural algebra structure in general. The main reason for this, is that these objects are defined as pushouts in $\Cc$ without any interaction with the multiplication $\mu_A$. 
This is the main motivation to introduce a second type of comodule algebras in the next section.
\end{remark}

\begin{examples}\exlabel{comodulealg}
Let $A$ be a (global) $H$-comodule algebra, and $p:A\to B$ a surjective algebra morphism. Then just as in \exref{quotient}, we can endow $B$ with the structure of a (geometric) partial comodule. Taking however the pushout \equref{quotient} in the category of algebras, one endowes $B$ with the structure of a partial comodule algebra. Let us make this construction a bit more explicit in the following example.

Consider the example of geometric partial $k[x,y]$-comodule from \exref{affine}, whose underlying object is $B=k[x,y]/(xy)$. Then $B$ has a natural algebra structure, however $B\bul H=k[x,y,x',y']/\rho((xy))$ is not an algebra since $\rho((xy))$ is not an ideal. Hence, $(B,B\bul H,\pi_B,\rho_B)$ is not a partial comodule algebra. However, consider $k[x,y,x',y']/(\rho(xy))$,  where $(\rho(xy))$ is the ideal generated by $\rho((xy))$, then there is a surjective algebra morphism $\pi:B\bul H\to k[x,y,x',y']/(\rho(xy))$. The partial comodule datum $(B',B'\bul H,\pi_{B'},\rho_{B'})$, where $B'=B$, 
$B'\bul H=k[x,y,x',y']/(\rho(xy))$, $\pi_{B'}=\pi\circ \pi_B$ and $\rho_{B'}=\pi\circ \rho_B$ is a partial $k[x,y]$-comodule algebra. Since the partial comodule $B'$ is geometric (being a quotient of a global one), $B'$ is also a geometric comodule algebra.

In the same way, one can turn the partial comodule from \exref{affine} and \exref{quantum} into a geometric partial comodule algebra and the partial comodule from \exref{CJ} into quasi partial comodule algebra. For this last example, and using the notation of \exref{CJ}, it is easy to see that if $e\in A\ot H$ is central, then already has that $A\bul H=(A\ot H)e$ is an algebra, so the partial comodule datum of \exref{CJ} is already a partial comodule algebra.
\end{examples}

\subsection{Partial comodules in the category of algebras}

Recall that the category of algebras $\Alg(\Cc)$ in a braided monoidal category $\Cc$, is again a monoidal category. Furthermore, a coalgebra $H$ in $\Alg(\Cc)$ is exactly a bialgebra in $\Cc$ and a comodule over $H$ in $\Alg(\Cc)$ is exactly an $H$-comodule algebra in $\Cc$. 
Following this point of view, we introduce the following definitions. 

\begin{definition}
Let $H$ be a bialgebra in the braided monoidal category $\Cc$, and consider $H$ as a coalgebra in the category $\Alg(\Cc)$.
A {\em quasi (resp. lax, geometric) partial algebra-comodule} over $H$ is a quasi (resp.\ lax, geometric) partial $H$-comodule $(A,A\bul H,\pi_A,\rho_A)$ in $\Alg(\Cc)$. 
\end{definition}

\begin{remark}
In the global case, algebra-comodules and comodule-algebras are identical structures. In the partial setting this however is no longer the case.
Firstly, given a partial comodule datum $(A,A\bul H,\pi_A,\rho_A)$ in $\Alg(\Cc)$, 
then applying the forgetful functor $U:\Alg(\Cc)\to \Cc$ 
yields a partial $H$-comodule datum $UA=(A,A\bul H,\pi_A,\rho_A)$ in $\Cc$, provided that $U(\pi_A)$ is an epimorphism in $\Cc$. This last condition is not necessarily the case if we take $\Cc=\Mod_k$ where $k$ is an arbitrary commutative ring, but it holds if $k$ is a field. However, even in case of $\Cc=\Vect_k$, the forgetful functor $U:\Alg_k\to \Vect_k$ does not preserve pushouts. In other words, the canonical morphism $\Theta_{UA}\to U(\Theta_A)$ is not an isomorphism in general. If $A$ is a quasi partial $H$-comodule algebra, then coassociativity holds in $U\Theta_A$, but not necessarily in $\Theta_{UA}$, and $UA$ is not necessarily a 
quasi partial $H$-comodule.

The next result tells however that conversely, partial comodule-algebras are still algebra-comodules.
\end{remark}

\begin{proposition}\prlabel{comodalgisalgcomod2}
If $(A,A\bul H,\pi_A,\rho_A)$ is a quasi partial $H$-comodule algebra, then $(A,A\bul H,\pi_A,\rho_A)$ is also a quasi partial algebra-comodule over $H$.
\end{proposition}

\begin{proof}
It follows from \prref{comodalgisalgcomod} that $(A,A\bul H,\pi_A,\rho_A)$ is a partial comodule datum in the category $\Alg(\Cc)$. 
If $A$ is a quasi partial comodule, then the coassociativity holds in the sense that 
\begin{equation}\eqlabel{coassA}
\theta_1\circ (\rho_A\bul H)\bul \rho_A=\theta_2\circ \ol{X\bul \Delta}\circ \rho_A
\end{equation}
where $(\Theta_A,\theta_1,\theta_2)$ is the coassociativity pushouts in $\Cc$. On the other hand, we can also consider the coassociativity pushout $(\Theta'_A,\theta'_1,\theta'_2)$ in $\Alg(\Cc)$. Since the forgetful functor $\Alg(\Cc)\to \Cc$ does not preserve pushouts, $\Theta_A$ and $\Theta'_A$ can be non-isomorphic objects in $\Cc$, but by the universal property of the pushouts $(A\bul H)\bul H$, $A\bul (H\bul H)$ and $\Theta_A$, we will obtain a morphism $\pi:\Theta_A\to\Theta'_A$. By composing both sides of \equref{coassA} with $\pi$, we find that the coassociativity also holds in $\Alg(\Cc)$, and hence $A$ is a quasi partial algebra-comodule.
\end{proof}

The difference between the pushouts in $\Alg(\Cc)$ and $\Cc$ can be understood very well in the situation where $\Cc=\Vect_k$. Indeed, consider $k$-algebra morphisms 
\[
\xymatrix{
& R \ar[dl]_a \ar[dr]^b\\
A && B
}
\]
where $b$ is surjective. Then we know from \leref{pushoutvect} that the pushout of $a$ and $b$ in $\Vect$ is given by the quotient $A/a(\ker b)$. In general (or more precisely, when $a$ is not surjective) $a(\ker b)$ is not an ideal in $A$, and hence $A/a(\ker b)$ is not an algebra. However, if we denote by $I$ the ideal generated by $a(\ker b)$, then one can easily see that $A/I$ is the pushout of $(a,b)$ in $\Alg_k$. 

As a consequence, we find the following.

\begin{corollary}
If $(A,A\bul H,\pi_A,\rho_A)$ is a geometric partial $H$-comodule $k$-algebra, then $(A,A\bul H,\pi_A,\rho_A)$ is also a geometric partial algebra-comodule over $H$.
\end{corollary}

\begin{proof}
By \prref{comodalgisalgcomod2} we know already that $A$ is a quasi partial algebra-comodule. On the other hand, since $A$ is geometric as comodule algebra, we find that the pushouts $(A\bul H)\bul H$ and $A\bul (H\bul H)$ are isomorphic in $\Vect_k$. Because of the explicit description of these pushouts recalled above, this means that the following subspaces of $(A\bul H)\ot H$ are isomorphic (even identical):
\[(\rho_A\ot H)(\ker \pi_A)=(\pi_A\ot H)\circ (A\ot \Delta)(\ker\pi_A)\]
Hence the ideals generated by these subspaces will also be the same, and therefore the corresponding coassociativity pushouts in $\Alg_k$ will be isomorphic, which means exactly that $A$ is geometric as a partial algebra-comodule.
\end{proof}

As the following examples illustrate, the converse of the previous corollary does not hold.

\begin{examples}
All examples from \exref{comodulealg} will give rise to examples of algebra-comodules. Since the examples obtained from \exref{affine} and \exref{quantum} are geometric as comodule-algebra, they are by the previous proposition also geometric as algebra-comodules. We remarked before that the example from \exref{CJ} is not geometric as comodule-algebra, however we will show now that is does become geometric as algebra-comodule. 

Let $A$ be a partial coaction of a Hopf algebra $H$ in the sense of \cite{CJ}. Consider as in \exref{CJ} the partial comodule datum $(A,A\bul H,\pi,\rho)$, where $A\bul H=\{a1_{[0]}\ot h1_{[1]}\}=(A\ot H)e$, which is a direct summand of $A\ot H$ and the left $A\ot H$-module generated by the idempotent $e=\rho(1)=1_{[0]}\ot 1_{[1]}$ and which can be seen as the quotient of $A\ot H$ by the left ideal $(A\ot H)e'$ where $e'=1\ot 1_H-e$ (we denote $1=1_A$ the unit of $A$). As explained in \exref{comodulealg}, in order to obtain a partial comodule algebra one has to consider an alternative partial comodule datum, where $A\bul' H$ is the quotient of $A\ot H$ by the two-sided ideal $(A\ot H)e'(A\ot H)$.

The ideal in $A\ot H\ot H$ generated by $(\rho\ot H)((A\ot H)e'(A\ot H))$ is then nothing else than the ideal generated by $(\rho\ot H)(e')$. Similarly, the ideal in $A\ot H\ot H$ generated by the image of $(A\ot H)e'(A\ot H)$ under $(\pi\ot H)\circ A\ot \Delta$ is the ideal generated by $(\pi\ot H)\circ (A\ot \Delta)(e')$. Using axiom (CJ2), we find 
\begin{eqnarray*}
(\rho\ot H)(e')&=& (\rho\ot H)(1\ot 1_H-e) = 1_{[0]}\ot 1_{[1]}\ot 1_H - 1_{[0][0]}\ot 1_{[0][1]}\ot 1_{[1]}\\
&=&1_{[0]}\ot 1_{[1]}\ot 1_H-1_{[0]}1_{[0']}\ot 1_{[1](1)}1_{[1']}\ot 1_{[1](2)}\\
&=&(\pi\ot H)\circ A\ot \Delta(e')
\end{eqnarray*}
Hence it follows that both elements generate the same ideals, which implies that $A$ is geometric as a partial algebra-comodule.
\end{examples}

\subsection{Partial Hopf modules}

Let $\Cc$ be an oplax monoidal category and denote as in \seref{comodalg} $\ot_0(\emptyset)=I$ and $\ot_n(X_1,\dots,X_n)=(X_1\ot \ldots \ot X_n)$. Let $(A,m,u)$ be an algebra in $\Cc$. Then a (right) $A$-module in $\Cc$ is an object $M$ endowed with a morphism $\mu_M:(M\ot A)\to M$
satisfying the following conditions
\[
\xymatrix@!C{
& ((M)\ot (A\ot A)) \ar[r]^-{\iota\ot m} & (M\ot A)\ar[dr]^{\mu_M}\\
(M\ot A\ot A) \ar[ur]^{\gamma} \ar[dr]_\gamma &&& M\\
& ((M\ot A)\ot (A)) \ar[r]^-{\mu_M\ot \iota} & (A\ot A)\ar[ur]_{\mu_M}\\
}\]\[
\xymatrix{
& (M\ot I) \ar[r]^-{(A\ot u)} & (M\ot A)\ar[dr]^{\mu_M}\\
(M)  \ar[ur]^{\gamma} \ar[rrr]^{\iota_A} &&& M\\
}
\]
Then we obtain the following natural definitions.
\begin{definition}
Let $H$ be a Hopf $k$-algebra and $(A,A\bul H,\pi_A,\rho_A)$ a quasi (resp. geometric) partial $H$-comodule algebra, i.e.\ an algebra in the lax monoidal category of quasi (resp. geometric) partial $H$-comodules. A 
{\em quasi (resp. geometric) partial $(A,H)$-relative Hopf module} is a right $A$-module in the lax monoidal category of quasi (resp. geometric) partial $H$-comodules. When $A$ and $H$ are clear from the context, we will just call this a partial Hopf module.

We will denote by $\PMod_A^H$ the category whose objects are quasi partial $(A,H)$-relative Hopf modules and whose morphisms are morphism of partial $H$-modules that are at the same time $A$-linear.
\end{definition}

As it is the case for partial comodule-algebras, since the forgetful functors $\gPMod^H\to\qPMod^H\to \Vect_k$ are monoidal, each geometric partial relative Hopf module is also a quasi partial relative Hopf module and a quasi or geometric partial relative Hopf module is also a module for the $k$-algebra $A$.

Let us make the previous definition a bit more explicit.
\begin{lemma}\lelabel{MbulH}
Let $(A,A\bul H,\pi_A,\rho_A)$ be  a quasi partial $H$-comodule algebra. A quasi partial $(A,H)$-relative Hopf module is a quasi partial $H$-comodule $(M,M\bul H,\pi_M,\rho_M)$ endowed with an $A$-module structure $\mu_M:M\ot A\to M$ such that the following compatibility conditions hold:
\begin{enumerate}[{[PRHM1]}]
\item $\ker(\pi_M\ot \pi_A)\subset \ker(\pi_M\circ\mu_{M\ot H})$;
\item $(ma)_{[0]}\bul (ma)_{[1]}=m_{[0]}a_{[0]}\bul m_{[1]}a_{[1]}$
for all $m\in M$ and $a\in A$.
\end{enumerate}
where \[\mu_{M\ot H}:M\ot H\ot A\ot H\to M\ot H,\ \mu_{M\ot H}((m\ot h)(a\ot k))=ma\ot hk.\]
is the induced $A\ot H$-module on $M\ot H$.

Under these conditions, $M\bul H$ is a right $A\bul H$-module.
\end{lemma}

\begin{proof}
Similarly as in the proof of \prref{comodalgisalgcomod}, consider the following diagram 
which expresses that the $A$-action $\mu_M:M\ot A\to M$ is a morphism of partial comodules, and the construction of $(M\ot A)\bul H$ as pushout.
\[
\xymatrix{
&& M\ot A \ar[rr]^-{\mu_M} \ar[d]^{\rho_{M\ot A}} \ar[dll]_{\rho_M\ot\rho_A} && M \ar[d]^{\rho_M}\\
(M\bul H)\ot (A\bul H) \ar[rr]^{\ol\mu_{M,A}} && (M\ot A)\bul H \ar[rr]^{\mu_M\bul H} \ar@{}[dl]|<<<{\pushoutw} && M\bul H\\
M\ot H\ot A\ot H  \ar[rr]^{\mu_{M,A}} \ar[u]_{\pi_M\ot \pi_A} && M\ot A\ot H \ar[rr]^{\mu_M\ot H} \ar[u]_{\pi_{M\ot A}} && M\ot H \ar[u]_{\pi_M}
}
\]
By construction, we know that 
\begin{eqnarray*}
\ker \pi_{M\ot A}&=&\mu_{M,A}(\ker (\pi_M\ot \pi_A))
\end{eqnarray*}
Then by \leref{comodulemorphismabelian}, in order for the linear map $\mu_M$ to be a morphism of partial $H$-comodules it is needed that $(\mu_M\ot H)(\ker \pi_{M\ot A})\subset \ker \pi_M$. Since $\mu_{M\ot H}=(\mu_M\ot H)\circ\mu_{M,A}$, this means furthermore that
\begin{eqnarray*}
\pi_M\circ (\mu_{M}\ot H) (\ker \pi_{M\ot A})
=\pi_M\circ \mu_{M\ot H} (\ker (\pi_M\ot \pi_A))=0
\end{eqnarray*}
or equivalently,
$\ker(\pi_M\ot \pi_A)\subset \ker(\pi_M\circ\mu_{M\ot H})$, i.e. [PRHM1] holds.

This condition implies that the map
\begin{eqnarray*}
&\mu_{M\bul H}=(\mu_{M}\bul H)\circ \ol\mu_{M,A}:(M\bul H)\ot (A\bul H)\to M\bul H,\\ 
&\mu_{M\bul H}((m\bul h)(a\bul k))=ma\bul hk
\end{eqnarray*}
is well-defined
and
defines an action of $A\bul H$ on $M\bul H$. 
Moreover, 
$\mu_M$ will be a morphism of partial $H$-modules if and only if
\[\rho_M\circ\mu_M= \mu_{M\bul H}\circ (\rho_M\ot\rho_A)\]
which gives exactly condition [PRHM2].
\end{proof}

\begin{example}
Clearly any algebra in a lax monoidal category is a module over itself, hence any 
quasi partial $H$-comodule algebra $(A,A\bul H,\pi_A,\rho_A)$
is also a quasi partial $(A,H)$-relative Hopf module.
\end{example}

The following observation will be useful later.

\begin{lemma}\lelabel{kerpi}
If $M$ is a quasi partial $(A,H)$-relative Hopf module, then \[\pi_M(ma\ot h)=0\] for all
\begin{enumerate}[(i)]
\item
$m\ot h\in \ker\pi_M$ and $a\in A$;
\item
$m\in M$ and $a\ot h\in \ker \pi_A$.
\end{enumerate}
Consequently, the isomorphism $M\ot_A(A\ot H)\cong M\ot H$ induces an surjective map $M\ot_A(A\bul H)\to M\bul H$.
\end{lemma}

\begin{proof}
Since $\ker(\pi_M\ot \pi_A)=\ker\pi_M\ot A\ot H+M\ot H\ot\ker \pi_A$, we know that for all $m\ot h\in \ker\pi_M$ and $a\in A$, $m\ot h\ot a\ot 1_H\in \ker(\pi_M\ot \pi_A)$. Hence by axiom (PRHM1), we find that $\pi_M(ma\ot h)=0$. Similarly, for $m\in M$ and $a\ot h\in \ker \pi_A$, we have $m\ot 1_H\ot a\ot h\in \ker(\pi_M\ot \pi_A)$ and we can follow the same reasoning.

For the last statement, it follows from part (ii) that $M\ot_A(\ker\pi_A)$ is included in $\ker\pi_M$ via the canonical isomorphism $M\ot_A(A\ot H)\cong M\ot H$. Consequently, we obtain the stated surjection $M\ot_A(A\bul H)\to M\bul H$.
\end{proof}

\subsection{Hopf-Galois theory}

\begin{definition}
Let $H$ be a Hopf $k$-algebra, $(A,A\bul H,\pi_A,\rho_A)$ a quasi partial $H$-comodule algebra and $(M,M\bul H,\pi_M,\rho_M,\mu_M)$ a quasi partial $(A,H)$-relative Hopf module. The {\em $H$-coinvariants} of $M$ are defined as the following equalizer in $\Vect_k$
\begin{equation}\eqlabel{coinv}
\xymatrix{M^{coH} \ar[rr] && M \ar@<.5ex>[rr]^-{\rho_M} \ar@<-.5ex>[rr]_-{\pi_M\circ (M\ot\eta_H)} && M\bul H}
\end{equation}
i.e.\ $M^{coH}=\{m~|~\rho_M(m)=m\bul 1_H\}$.
\end{definition}

\begin{proposition}
Let $H$ be a Hopf $k$-algebra, $(A,A\bul H,\pi_A,\rho_A)$ a quasi partial $H$-comodule algebra.
\begin{enumerate}[(i)]
\item The coinvariants $A^{coH}$ of $A$ form a subalgebra of $A$;
\item The coinvariants $M^{coH}$ of a quasi partial $(A,H)$-relative Hopf module $M$ form a module over $A^{coH}$; 
\item This induces a functor $(-)^{coH}:\PMod^H_A\to \Mod_{A^{coH}}$.
\end{enumerate}
\end{proposition}

\begin{proof}
\ul{(i)}.
Since $\rho_A$ and $\pi_A\circ (A\ot\eta_H)$ are both algebra morphisms and the forgetful functor $U:\Alg_k\to \Vect_k$ creates limits, $A^{coH}$ is a subalgebra of $A$. Alternatively, this follows by a direct computation similar to the next part.\\
\ul{(ii)}. Consider the restriction of the multiplication map $\mu_M:M^{coH}\ot A^{coH}\to M,\ \mu_M(m\ot a)= ma$. Then 
\[\rho(ma)=m_{[0]}a_{[0]}\bul m_{[1]}a_{[1]}=(m_{[0]}\bul m_{[1]})(a_{[0]}\bul a_{[1]})=(m\bul 1_H)(a\bul 1_H)=(ma\bul 1_H)\]
hence $\rho_M(ma)\in M^{coH}$.\\
\ul{(iii)}. This part is easily verified.
\end{proof}

Let us remark that the coinvariant functor is representable.

\begin{lemma}\lelabel{coinvHom}
For any  quasi partial $(A,H)$-relative Hopf module, we have that 
\[M^{coH}\cong \Hom_A^H(A,M),\]
the set of partial $H$-comodule morphisms from $A$ into $M$ that are right $A$-linear.
\end{lemma}

\begin{proof}
Let $f:A\to M$ be a right $A$-linear morphism of partial $H$-comodules. Then the following diagram commutes
\[
\xymatrix{
A \ar[d]_{\rho_A} \ar[rr]^-f && M \ar[d]^{\rho_M}\\
A\bul H \ar[rr]^-{f\bul H} && M\bul H\\
A\ot H \ar[u]^{\pi_A} \ar[rr]^-{f\ot H} && M\ot H \ar[u]_{\pi_M}
}
\]
Since we know that $\rho_A(1_A)=\pi_A(1_A\ot 1_H)$, it follows from the commutativity of this diagram that $f(1_A)\in M^{coH}$. Conversely, given any $m\in M^{coH}$ we clearly have that the map $f_m:A\to M,\ f_m(a)=ma$ is right $A$-linear. Moreover, thanks \leref{kerpi}(ii), we also know that $(f_m\ot H)(\ker\pi_A)\subset \ker\pi_M$. Hence the morphism $f_m\bul H:A\bul H\to M\bul H,\ f_m(a\bul h)=ma\bul h$ is well-defined. Finally, we find for any $a\in A$ that 
\begin{eqnarray*}
\rho_M(f_m(a))&=&\rho_M(ma)=\rho_M(m)\rho_A(a)\\
&=&(m\bul 1_H)(a_{[0]}\bul a_{[1]})= ma_{[0]}\bul a_{[1]}\\
& =& f_m\bul H(a_{[0]}\bul a_{[1]})
\end{eqnarray*}
I.e. $f_m\in\Hom_A^H(A,M)$.
The above constructions provide well-defined mutual inverses between $M^{coH}$ and $\Hom_A^H(A,M)$.
\end{proof}

\begin{proposition}\prlabel{Galoisadjunction}
For any $A^{coH}$-module $N$, the right $A$-module $N\ot_{A^{coH}}A$ can be endowed with the structure of a quasi partial $(A,H)$-relative Hopf module by means of the following partial $H$-comodule structure 
\begin{eqnarray*}
\pi_{N\ot_{A^{coH}}A}= N\ot_{A^{coH}}\pi_A:&&N\ot_{A^{coH}}A\ot H \to N\ot_{A^{coH}}(A\bul H)=: (N\ot_{A^{coH}}A)\bul H\\
\rho_{N\ot_{A^{coH}}A}=N\ot_{A^{coH}}\rho_A:&&N\ot_{A^{coH}}A \to N\ot_{A^{coH}}(A\bul H)
\end{eqnarray*}
Moreover, this construction yields a functor $-\ot_{A^{coH}}A:\Mod_{A^{coH}}\to \PMod^H_A$ that is a left adjoint to the coinvariant functor $(-)^{coH}:\PMod^H\to \Mod_{A^{coH}}$. 
\end{proposition}

\begin{proof}
The construction of the functor $-\ot_{A^{coH}}A:\Mod_{A^{coH}}\to \PMod^H_A$ is clear from the statement. To verify the adjunction property, we will define a counit $\zeta$ and a unit $\nu$. For any quasi partial $(A,H)$-relative Hopf module $M$ we define
\begin{eqnarray*}
\zeta_M:M^{coH}\ot_{A^{coH}}A \to M&& \zeta_M(m\ot_{A^{coH}}a)=ma
\end{eqnarray*}
Clearly, $\zeta_M$ is a right $A$-linear map. Let us check that it is also a morphism of partial $H$-comodules. Firstly, we need to verify that $\pi_M \circ (\zeta_M\ot H)(\ker \pi_{M^{coH}\ot_{A^{coH}}A}) = 0$ (see \leref{comodulemorphismabelian}). 
Since by construction $\ker \pi_{M^{coH}\ot_{A^{coH}}A}=M^{coH}\ot_{A^{coH}}\ker\pi_A$, 
this follows directly by \leref{kerpi}(ii).
Secondly, we should check that $\rho_M\circ \zeta_M=(\zeta_M\bul H)\circ (M^{coH}\ot_{A^{coH}}\rho_A)$, where we know that $\zeta_M\bul H$ is well-defined by the first part. Indeed, take any $m\ot a\in M^{coH}\ot_{A^{coH}A}$ then 
\begin{eqnarray*}
\rho_M(\zeta_M(m\ot a))&=&(ma)_{[0]}\bul (ma)_{[1]} = (m_{[0]}\bul m_{[1]})(a_{[0]}\bul a_{[1]})\\
&=& (m\bul 1_H)(a_{[0]}\bul a_{[1]})=ma_{[0]}\bul a_{[1]}\\
&=& \zeta_M(m\ot_{A^{coH}}a_{[0]})\bul a_{[1]}=(\zeta_M\bul H)(m\ot_{A^{coH}}\rho_A(a)).
\end{eqnarray*}
On the other hand, for any $A^{co H}$-module $N$, we define 
\[\nu_N:N\to (N\ot_{A^{coH}}A)^{coH},\ \nu_N(n)=n\ot 1_A.\]
Since $A$ is an algebra in the category of partial $H$-modules we have that $\rho_A(1)\in A^{coH}$ (see \equref{rho1}).
It is now easily verified that $\zeta$ and $\nu$ are indeed the counit and unit for this adjunction.
\end{proof}

\begin{remarks}
If $A$ is lax (respectively geometric) as partial $H$-comodule, then we find for any $N\in\Mod_{A^{coH}}$ that the partial Hopf module $N\ot_{A^{coH}}A$ constructed in the previous proposition is also lax (resp.\ geometric).

Let $\iota:B\to A^{coH}$ be any ring morphism, then of course the adjunction from \prref{Galoisadjunction} can be combined with the extension-restriction of scalars functors, to obtain a pair of adjoint functors
\[-\ot_BA: \Mod_{B}\rightleftarrows \PMod^H_A:(-)^{coH}.\]
\end{remarks}

\begin{definition}
Let $A$ be a quasi partial $H$-comodule algebra and $B\to A^{coH}$ a ring morphism.
We call the morphism $\iota:B\to A$ a partial Hopf-Galois extension if and only if the following canonical map is bijective
\[\can:A\ot_B A\to A\bul H, \can(a\ot_B a')=aa'_{[0]}\bul a'_{[1]}\]
Remark that here $aa'_{[0]}\bul a'_{[1]}$ denotes the product $(a\bul 1_H)(a'_{[0]}\bul a'_{[1]})$ which is well-defined since $m:A\ot A\to A$ is a morphism of partial $H$-comodules.
\end{definition}

The following examples show how partial Hopf-Galois extensions can be interpreted as ``partial principle bundles''.

\begin{example}
Let $A$ be a global $H$-comodule algebra, and suppose that $A/A^{coH}$ is Galois, i.e.\ the canonical map $A\ot_{A^{coH}}A\to A\ot H$ is bijective. Consider a surjective algebra morphism $p:A\to B$ and endow $B$ with the induced structure of a partial comodule algebra. 
Then we obtain a canonical algebra morphism $A^{coH}\to B^{coH}$ and in fact $B^{coH}\cong A^{coH}/(A^{coH}\cap \ker p)$. We find that the following diagram commutes
\[
\xymatrix{
A\ot_{A^{coH}}A \ar@{->>}[d]^{p\ot p} \ar[rr]^-{\can_A} && A\ot H \ar@{->>}[d]^{\pi_B\circ (p\ot H)} \\
B\ot_{B^{coH}}B \ar[rr]^-{\can_B} && B\bul H
}
\]
If $\can_A$ is an isomorphism, it is clear that $\can_B$ is surjective. Moreover, consider any $b\ot b'\in \ker\can_B$. Since $p\ot p$ is surjective, we can write $b\ot b'=p(a)\ot p(a')$, such that $\can(a\ot a')\in \ker \pi_B\circ (p\ot H)$. From the construction of $B\bul H$ as a pushout, we know that $\ker\pi_B = (p\ot H)\circ \rho_A(\ker p)$. Hence, we find that
$\can(a\ot a')=u_{[0]}\ot u_{[1]} + v\ot h$ for some $u\in \ker p$ and $v\ot h\in \ker p\ot H$.
Applying $\can^{-1}$ to this equation, we obtain that $a\ot a'=1\ot u + \can^{-1}(v\ot h)$. Moreover, since $\can$ is left $A$-linear, $\can^{-1}$ is left $A$-linear as well and therefore $\can^{-1}(v\ot h)=v\can^{-1}(1\ot h)\in \ker p\ot_{A^{coH}} A$, since $v\in\ker p$ which is an ideal in $A$ as $p$ is an algebra map. We can conlude that $a\ot a'\in A\ot_{A^{coH}}\ker p+ \ker p\ot_{A^{coH}}A$ 
and therefore $b\ot b'=p(a)\ot p(a')=0$. So $\can_B$ is bijective as well, i.e.\ $B$ is partially Hopf Galois. 
\end{example}

\begin{example}
It is known that if an algebraic group $G$ acts strictly transitive on an algebraic space $X$ (i.e. $X$ is a principal homogeneous $G$-space), then the coordinate algebra $A=\Oo(G)$ is $\Oo(G)$-Galois with trivial coinvariants. 
If we take any subvariety $Y\subset X$ then we know that $\Oo(Y)$ will be a partial $\Oo(G)$-comodule algebra. Applying the previous example, we find that that $\Oo(Y)$ will be partially Hopf-Galois.
For example, the partial comodule algebras from \exref{affine}, \exref{quantum} (see \exref{comodulealg}) provide examples of partial principle homogeneous $G$-spaces (where $G$ is respectively $k[x,y]$ and $k\bk{x,y}$). 

More generally, $X$ is a principal $\Oo(G)$-bundle if and only if $\Oo(X)$ is an $\Oo(G)$-Galois extension (with possible non-trivial coinvariants), see e.g.\ \cite{Schneider}. Again, any subvariety $Y$ of the principle bundle $X$ will give rise to a partial principle bundle. 
\end{example}

In the global case, we know that if $A/A^{coH}$ is Hopf-Galois and $A$ is faithfully flat as left $A^{coH}$-module, then the category of relative $(A,H)$-Hopf modules is equivalent to the category of $A^{coH}$-modules. 
This theorem has a nice interpretation in terms of corings, as this equivalence factors through the category of comodules over the canonical Sweedler $A$-coring $A\ot_BA$ and the canonical map is an $A$-coring morphism. For more details we refer to e.g.\ \cite{Cae:descent}, where it is also pointed out that the faithful flatness is in fact too strong.

Since in the partial setting it follows from earlier observations in this paper that the category of partial comodules is not abelian, the category of partial relative $(A,H)$-Hopf modules cannot be expected to be equivalent with a module category. Nevertheless, let us show that under the same mild conditions as in the global case, we can characterize when the functor $-\ot_{A^{coH}}A:\Mod_{A^{coH}}\to \PMod_A^H$ is fully faithful. The following is an adaptation of the approach from \cite{Cae:descent} (see also \cite{CDV}).

Recall that a morphism of left $B$-modules $f:N\to M$ is called {\em pure} if and only if for any right $B$-module $P$, the map $P\ot_B f:P\ot_BN\to P\ot_B M$ is injective. In particular, if $\iota:B\to A$ be a ring morphism, then $\iota$ is said to be pure (as left $B$-module morphism) if for any right $B$-module $P$ the map $\iota_P:P\to P\ot_BA,\ \iota_P(p)=p\ot_B 1_A$ is injective.

\begin{lemma}
Let $\iota:B\to A$ be a ring morphism. Then the following statements are equivalent:
\begin{enumerate}[(i)]
\item $\iota$ is pure as left $B$-module morphism 
\item For any right $B$-module $N$, the fork
\begin{equation}\eqlabel{Nfork}
\xymatrix{N\ar[rr]^-{\iota_N} && N\ot_BA \ar@<.5ex>[rr]^-{\iota_N\ot_BA} \ar@<-.5ex>[rr]_-{N\ot_B\iota_A} && N\ot_BA\ot_BA}
\end{equation}
is an equalizer in $\Mod_B$.
\end{enumerate}
In particular, if $A$ is faithfully flat as left $B$-module, then $A$ is left pure.
\end{lemma}

\begin{proof}
$\ul{(i)\Rightarrow(ii)}$
Denote by $E$ the equalizer of \equref{Nfork}, and define $P=E/\iota_N(N)$. Take any $e\in E$, the we can write $e=n_i\ot_B a_i$ and $n_i\ot_B 1_A\ot_B a_i = n_i\ot_B a_i \ot_B 1_A$. Apply $\pi\ot_BA$ to this identity, then we have that 
$\iota_P(\pi(e))=\pi(e)\ot_B1_A = \pi(n_i\ot_B a_i) \ot_B 1_A = \pi(n_i\ot_B 1_A)\ot_B a_i=0$, since $n_i\ot_B1_A\in \iota_N(N)$. Since $\iota_P$ is injective, it follows that
$\pi(e)=\pi(n_i\ot_B a_i)=0$ in $P=E/\iota(N)$, hence $n_i\ot_B a_i\in \iota(N)$. \\
$\ul{(ii)\Rightarrow(i)}$. Since \equref{Nfork} is an equalizer, we have in particular that $\iota_N$ is injective.
\end{proof}

\begin{proposition}\prlabel{unitiso}
Let $\iota:B\to A$ be a partial $H$-Galois extension, then the functor $-\ot_BA:\Mod_B\to \PMod_A^H$ is fully faithful if and only if $\iota$ is pure.

Under these conditions, $B\cong A^{coH}$.
\end{proposition}

\begin{proof}
Consider the following commutative diagram
\[
\xymatrix{
N \ar[rr]^-{N\ot_{B}\iota} \ar[d]_-{\nu_N} && N\ot_{B} A  \ar@<.5ex>[rrr]^-{N\ot_{B}\iota\ot_{B} A } \ar@<-.5ex>[rrr]_-{N\ot_{B}A\ot_{B}\iota} \ar@{=}[d] &&& N\ot_{B}A\ot_{B}A \ar[d]^-{N\ot_{B}\can}\\
(N\ot_{B}A)^{co H} \ar[rr] && N\ot_{B}A \ar@<.5ex>[rrr]^-{\rho_{N\ot_{B}A}=N\ot_B\rho_A} \ar@<-.5ex>[rrr]_-{N\ot_BA\bul\eta_H} &&& 
N\ot_{B}A\bul H
}
\]
The lower row is an equalizer by the definition of the coinvariants $(N\ot_BA)^{coH}$. Since $\can$ is an isomorphism, it then follows that the upper row is an equalizer if and only if $\nu_N$ is an isomorphism. 
The upper row in the above diagram is exactly \equref{Nfork}. 
By the previous lemma, this means that $\iota:B\to A$ is pure if and only if the unit $\nu$ of the adjunction from \prref{Galoisadjunction} is a natural isomorphism, i.e.\ $-\ot_BA$ is fully faithful.

For the last statement, just remark that $\nu_B:B\to (B\ot_BA)^{coH}=A^{coH}$ is an isomorphism.
\end{proof}

As we have remarked before, since partial comodules do not provide an abelian category, one cannot expect that the functor $-\ot_BA:\Mod_{B}\to \PMod_A^H$ is an equivalence in general. The following observation shows that as soon as $H$ is non-trivial, this functor will never be an equivalence.
Indeed, by construction any induced partial Hopf module $M=N\ot_{B}A$ satisfies $M\bul H=M\ot_A(A\bul H)$. It follows however from \leref{MbulH} that in general we only have an surjection $M\ot_A(A\bul H)\tto M\bul H$. This motivates the following definition.

\begin{definition}
A partial relative Hopf module $M$ is called {\em minimal} iff $M\bul H=M\ot_A(A\bul H)$. 
\end{definition}

\begin{example}
Let $A$ be a partial $H$-coaction as in \exref{CJ}, considered as a partial comodule algebra, see \exref{comodulealg}. Then we have that 
\[A\bul H=\{a1_{[0]}\ot h1_{[1]}~|~a\ot h\in A\ot H\}\]
where $e=1_{[0]}\ot 1_{[1]}$ is supposed to be central in $A\ot H$.
Let $M$ be a relative Hopf module in the sense of \cite{CJ}, that is $M$ is a right $A$-module, endowed with a coaction $\rho:M\to M\ot H,\ \rho(m)=m_{[0]}\ot m_{[1]}$ satisfying 
\begin{itemize}
\item $m=m_{[0]}\epsilon(m_{[1]})$;
\item $\rho(m_{[0]})\ot m_{[1]}=m_{[0]}1_{[0]}\ot m_{[1](1)}1_{[1]}\ot m_{[1](2)}$;
\item $\rho(ma)=m_{[0]}a_{[0]}\ot m_{[1]}a_{[1]}$.
\end{itemize}
Then by defining $M\bul H=\{m1_{[0]}\ot h1_{[1]}~|~ m\ot h\in M\ot H\}$ we find that $M$ can be endowed with the structure of a partial Hopf module in the sense defined here. Moreover, 
one then easily checks that this partial Hopf module is minimal:
\[M\bul H\cong M\ot_A(A\bul H),\ m1_{[0]}\ot h1_{[1]}\mapsto m\ot_A(1_{[0]}\ot h1_{[1]}).\]
\end{example}

A key tool in \cite{CJ} is the observation that for a partial action as in the previous example, the $A$-bimodule $A\bul H$ is an $A$-coring, whose comodules correspond exactly with the partial relative Hopf modules. We now extend this result to the present setting.

\begin{lemma}\lelabel{coring}
Let $A$ be a quasi partial $H$-comodule algebra. 
\begin{enumerate}[(i)]
\item There is a canonical epimorphism $p:(A\bul H)\bul H\to (A\bul H)\ot_A (A\bul H)$;
\item For any minimal partial Hopf module $M$, we have a canonical epimorphism $p_M:(M\bul H)\bul H\to (M\bul H)\ot_A(A\bul H)$.
\end{enumerate}
Consequently, if $A$ is geometric as $H$-comodule, then $\Cc=A\bul H$ is an $A$-coring and $\can:A\bul H\to A\ot_BA$ is a morphism of $A$-corings. 
Moreover, there is a functor from the category of geometric minimal partial Hopf modules to the category of $\Cc$-comodules.
\end{lemma}

\begin{proof}
\ul{(i)}.
Consider the following diagram.
\[
\xymatrix{
(A\bul H)\ot H  \ar[rr]^-{\pi_{A\bul H}} \ar[d]^{\phi}_\cong && (A\bul H)\bul H \ar[r] & 0 \\
(A\bul H)\ot_A (A\ot H) \ar[rr]^-{(A\bul H)\ot_A\pi_A} && (A\bul H)\ot_A (A\bul H) \ar[r] &0
}
\]
where $\phi$ is the isomorphism given by $\phi((a\bul h)\ot h)=(a\bul h)\ot_A (1_A\ot h)$.
We know that $\ker((A\bul H)\ot_A\pi_A)= (A\bul H)\ot_A\ker\pi_A$ and
$\ker\pi_{A\bul H}=(\rho_A\ot H)(\ker\pi_A)$. Consider any $a\ot h\in\ker\pi_A$. Then we find that
\begin{eqnarray*}
(a_{[0]}\bul a_{[1]})\ot_A (1_A\ot h) = (1_A\bul 1_H)\cdot a\ot_A (1_A\ot h)= (1_A\bul 1_H)\ot_A (a\ot h)
\end{eqnarray*}
Hence we find that $\phi(\ker\pi_{A\bul H})\subset \ker((A\bul H)\ot_A\pi_A)$. Consequently, $\phi$ induces an epimorphism  $p:(A\bul H)\bul H\to (A\bul H)\ot_A (A\bul H)$.\\
\ul{(ii)}. By part (i), we know that the following diagram commutes
\[
\xymatrix{
& A\ot H \ar@{->>}[dl]_{\pi_A} \ar[dr]^{\rho_A\ot H} \\
A\bul H \ar[dr]_{p\circ (\rho_A\bul H)} && (A\bul H)\ot H \ar@{->>}[dl]^{(id\ot_A\pi_A)\circ \phi}\\
& (A\bul H)\ot_A (A\bul H)
}
\]
Applying the functor $M\ot_A-$ to this diagram, and using $M\ot_A(A\bul H)=M\bul H$, we find that the following diagram commutes as well (we avoid naming the arrows to not overload the diagram)
\[
\xymatrix{
& M\ot H \ar@{->>}[dl] \ar[dr] \\
M\bul H \ar[dr] && (M\bul H)\ot H \ar@{->>}[dl]\\
& (M\bul H)\ot_A (A\bul H)
}
\]
Hence, by the universal property of the pushout, we obtain an epimorphism $(M\bul H)\bul H\to (M\bul H)\ot_A (A\bul H)$.
\\
Using the above, we can define a coring structure on $A\bul H$ by means of the counit $A\bul \epsilon$ and comultiplication $p\circ (A\bul \Delta)$, and one can easily observe that $\can$ is a coring morphism for this structure. Similarly, for a minimal partial Hopf module $M$, the partial $H$-coaction is also an $A\bul H$-coaction since $\rho_M:M\to M\bul H=M\ot_A(A\bul H)$.
Furthermore, it is enough to remark that for a geometric partial Hopf module, the coassociativity holds in $M\bul H\bul H$. If $M$ is minimal then by the above $(M\bul H)\ot_A(A\bul H)$ is a quotient of $M\bul H\bul H$, so coassociativity also holds there.
\end{proof}

\begin{remark}
Given a comodule $M$ over the coring $A\bul H$, one can construct a partial comodule datum $(M,M\ot_A(A\bul H),M\ot_A \pi_A,\rho_M)$. However, one cannot expect that all such comodule data provide partial $H$-comodules. Indeed, if $A/B$ is Galois, then the canonical map is a coring isomorphism $A\ot_BA\cong A\bul H$. and the categories $A\ot_BA$-comodules and $A\bul H$-comodules are isomorphic as well. When $A$ is flat as left $B$-module, then the category of $A\ot_BA$-comodules is abelian. 
If $\Mod^{A\bul H}$ coincides with the category of partial Hopf modules, than this would imply that the latter category is Abelian, which we know is not the case.
\end{remark}

If $A$ is flat as left $B$-module, then the functor $-\ot_BA$ preserves all equalizers.
Recall from \cite{Cae:descent} (see \cite{CDV} for a corrected version of the statement) that the functor $-\ot_BA$ preserves the equalizers of the form \equref{coinv} provided $A$ is pure as left $A$-module and $B$ lies in the center of $A$.

\begin{proposition}\prlabel{counitiso}
Let $M$ be a partial Hopf module $M$. 
\begin{enumerate}[(i)]
\item If $\zeta_M$ (counit of the adjunction \prref{Galoisadjunction}) is an isomorphism, then $M$ is minimal.
\item
If $A/B$ is partially Hopf-Galois and $\zeta_M$ is a monomorphism, then $M$ is minimal.
\item
If $M$ is minimal and geometric, $A$ is partially Hopf Galois and the functor $-\ot_BA$ preserves the equalizers of the form \equref{coinv} (e.g. $A$ is flat as left $B$-module, or $A$ is pure as left $A$-module and $B\subset Z(A)$), then $\zeta_M$ is an isomorphism.
\end{enumerate}
\end{proposition}

\begin{proof}
\ul{(i)}
If $\zeta_M$ is an isomorphism of partial Hopf modules, then we find a composition of isomorphisms
\[
\hspace{-2cm}
\xymatrix{
M\ot_A(A\bul H) \ar[rr]^-{\zeta^{-1}_M\ot_A(A\bul H)} && (M^{coH}\ot_BA)\ot_A(A\bul H) \cong M^{coH}\ot_B(A\bul H) = (M^{coH}\ot_BA)\bul H \ar[r]^-{\zeta_M\bul H} & M\bul H
}
\]
and hence $M$ is minimal.\\
\ul{(ii)}. Consider the following commutative diagram
\[
\xymatrix{
M^{coH}\ot_BA \ar[rr] \ar[d]^{\zeta_M} && M\ot_BA \ar[d] \\
M \ar[rr]^-{\rho_M} && M\bul H 
}
\]
Since $M\ot_BA\cong M\ot_A(A\ot_BA)\cong M\ot_A(A\bul H)$, we know that the right vertical arrow is an epimorphism (see \leref{MbulH}).
If $A$ is flat as left $B$-module then the upper horizontal arrow is injective, and the lower horizontal arrow is injective as it has a left inverse $M\bul\epsilon$.  Hence if $\zeta_M$ is injective, we find that  $M\ot_BA\to M\bul H$ is an isomorphism, i.e. $M$ is minimal.\\
\ul{(iii)}.
Since $A/B$ is Hopf-Galois, we find obtain an isomorphism
\[
\xymatrix{
M\ot_BA \ar[r]^-\cong & M\ot_A(A\ot_BA) \ar[rr]^-{M\ot_A\can} && M\ot_A(A\bul H)
}
\]
And therefore, if $M$ is minimal we find that $M\ot_BA\cong M\ot_A(A\bul H)=M\bul H$. 
Consider now the following diagram
\[
\xymatrix{
M^{coH}\ot_BA \ar[rr] \ar@<.5ex>[d]^{\zeta_M} && M\ot_BA \ar[d]^\cong \ar@<.5ex>[rr]^-{\rho_M\ot_BA} \ar@<-.5ex>[rr]_-{\pi_M\circ (M\ot\eta_H)\ot_BA} && M\bul H\ot_BA\\
M \ar[rr]^-{\rho_M} \ar@<.5ex>[u]^{\zeta'_M} && M\bul H \ar@<.5ex>[rr]^{\rho_M\bul H} \ar@<-.5ex>[rr]_-{\ol{M\bul\Delta}} && M\bul H\bul H \ar@{->>}[u]_{p_M}
}
\]
By assumption, the upper row in this diagram is an equalizer. Since $M$ is geometric, the fork on the lower row splits and hence is also an equalizer. The surjective morphism $p_M$ is obtained from \leref{coring} and induces the morphism $\zeta'_M$. Then a diagram chasing argument shows that $\zeta_M$ and $\zeta'_M$ are mutual inverses.
\end{proof}

The following result subsumes the Hopf-Galois theory for partial coactions in the sense of Caenepeel-Janssen \cite{CJ}.

\begin{corollary}\colabel{HopfGalois}
Suppose that $A$ a partial $H$-comodule algebra that is geometric as partial comodule.
If $A/B$ is a partial Hopf-Galois extension and either
\begin{itemize}
\item $A$ is pure as left $B$-module and $B\subset Z(A)$;
\item $A$ is faithfully flat as left $B$-module;
\end{itemize}
then $\Mod_{A^{coH}}$ is equivalent to the category full subcategory of $\PMod_A^H$ consisting of
minimal geometric partial Hopf modules.
\end{corollary}

\begin{proof}
Since $A$ is geometric as partial comodule, the functor $-\ot_BA:\Mod_{A^{coH}}\to \PMod_A^H$ lands in the category of minimal geometric partial Hopf modules. By \prref{unitiso} and \prref{counitiso} we then obtain the stated equivalence of categories. 
\end{proof}

As we have remarked before, the functor $-\ot_BA:\Mod_{A^{coH}}\to \PMod_A^H$ cannot be expected to become an equivalence of categories. More precisely, it follows from \prref{counitiso} that we cannot expect that the functor $(-)^{coH}$ is full whenever it is applied to non-minimal partial Hopf modules.
We will finish our work by characterizing under which conditions this functor remains however faithful.

\begin{proposition}\prlabel{final}
Under the same conditions as in \coref{HopfGalois}, $A$ is a generator in $\PMod_A^H$ if and only if the functor $(-)^{coH}:\PMod_A^H\to \Mod_{A^{coH}}$ is faithful. 
\end{proposition}

\begin{proof}
Recall that $A$ is a generator in $\PMod_A^H$ if and only if for any object $M$ in $\PMod_A^H$ the canonical morphism
\[\phi_M:\coprod_{\Hom_A^H(A,M)}A\to M,\ \phi(a_f)=\sum_f f(a_f)\]
is surjective. Here $\coprod_{\Hom_A^H(A,M)}A$ denotes a coproduct of copies of $A$ indexed by the set $\Hom_A^H(A,M)$.

Recall from \leref{coinvHom} that $M^{coH}=\Hom_A^H(A,M)$ for any 
object $M$ in $\PMod_A^H$. Hence we obtain a well-defined morphism
\[\alpha_M:\coprod_{\Hom_A^H(A,M)}A \to M^{coH}\ot_{A^{coH}}A,\ \alpha_M(a_f)=f(1_A)\ot_{A^{coH}} a_f,\]
which is clearly surjective.

One now easily sees that $\phi_M=\zeta_M\circ \alpha_M$. Hence $\phi_M$ is surjective if and only if $\zeta_M$ is surjective. Finally, it is well-know that a right adjoint functor is faithful if and only if the counit of the adjunction is a natural epimorphism. 
\end{proof}

\begin{remark}
Under the conditions of \prref{final}, we know that for an object $M$ in  $\PMod_A^H$, the partial Hopf module morphism $\zeta_M:M^{coH}\ot_BA\to M$ is surjective. Hence, we find that $M\cong M^{coH}\ot_BA/\ker\zeta_M$ as right $A$-module. However, as we remarked earlier, $\ker\zeta_M$ is not necessarily a partial $H$-comodule. We then know from \exref{quotientpartial} that $M\ot_BA/\ker\zeta_M$ is a geometric partial comodule. Then $\zeta_M$ induces a morphism of partial Hopf modules $\zeta'_M: M\ot_BA/\ker\zeta_M\to M$, such that the underlying $A$-module morphism is an isomorphism. However, $\zeta'_M$ is not necessarily an isomorphism of partial Hopf modules, since in general $(M\ot_BA/\ker\zeta_M)\bul H$ and $M\bul H$ can be different. Therefore consider the following definition.
\begin{quote}
Let $(X,X\bul H,\pi_X,\rho_X)$ and $(Y,Y\bul H,\pi_Y,\rho_Y)$ be two partial $H$-comodule data. Let $f:X\to Y$ be a morphism in $\Cc$. Then consider the pushout
\[
\xymatrix{
X\ot H \ar[d]_{\pi_X} \ar[r]^-{f\ot H} & Y\ot H \ar[r]^-{\pi_Y} & Y\bul H \ar[d]^{p_Y}\\
X\bul H \ar[rr]_-{p_X} && P_{X,Y}  \ar@{}[ul]|<<<{\pushoutv}
}
\]
With this notation, $f$ is said to be a {\em weak morphism} of partial comodules, if the following diagram commutes
\[
\xymatrix{
X\ar[rr]^-f \ar[dd]^{\rho_X} && Y \ar[d]^{\rho_Y}\\
&& Y\bul H \ar[d]^{p_Y}\\
X\bul H \ar[rr]_{p_X} && P_{X,Y} 
}
\]
\end{quote}
Then the $A$-linear inverse of $\zeta'_M$ will be a weak morphism of partial comodules that is moreover a $2$-sided inverse of the (strong) morphism $\zeta'_M$. This motivates that weak morphisms of (geometric) partial comodules might be better behaved than the strong morphisms we studied in this work. We will investigate this further in future work.
\end{remark}

Let us finish by proving result which completely characterizes the image of the functor $-\ot_{A^{coH}}A:\Mod_{A^{coH}}\to \PMod_A^H$.

\begin{theorem}
\begin{enumerate}[(i)]
\item If $A$ and $A\bul H$ are flat as left $A^{coH}$-module (e.g.\ $A$ is flat as left $A^{coH}$-module and $A/A^{coH}$ is $H$-Galois), then the functor $-\ot_{A^{coH}}A:\Mod_{A^{coH}}\to \PMod^H_A$ preserves equalizers.
\item If $A$ is faithfully flat as left $A^{coH}$-module then the functor $-\ot_{A^{coH}}A:\Mod_{A^{coH}}\to \PMod^H_A$ reflects isomorphisms.
\item If $A$ is faithfully flat as left $A^{coH}$-module and $A/A^{coH}$ is partially Hopf-Galois, then the category of $A^{coH}$-modules is equivalent to the Eilenberg-Moore category $(\PMod_A^H)^\CC$, where $\CC$ is the comonad associated to the adjoint pair of \prref{Galoisadjunction}.
\end{enumerate}
\end{theorem}

\begin{proof}
\ul{(i)}. Consider the following equalizer diagram in $\Mod_{A^{coH}}$:
\[\xymatrix{
E\ar[rr]^e && N \ar@<.5ex>[rr]^f \ar@<-.5ex>[rr]_g && M
}\]
By the flatness of $A$ as a left $A^{coH}$-module, we then know that $(E\ot_{A^{coH}}A,e\ot_{A^{coH}}A)$ is the equalizer of the pair $(f\ot_{A^{coH}}A,g\ot_{A^{coH}}A)$ in $\Mod_A$. However, we have to show that this is also an equalizer in $\PMod^H_A$. 
To this end, consider any partial relative Hopf module $T$ with a morphism $t:T\to N\ot_{A^{coH}}A$ such that $(f\ot_{A^{coH}}A)\circ t = (g\ot_{A^{coH}}A)\circ t$. Then we can apply the forgetful functor $\PMod^H_A\to \Mod_{A}$ and we find that there exists a unique right $A$-linear map $u:T\to E\ot_{A^{coH}}A$ such that $t= e\ot_{A^{coH}}A\circ u$. We will be done if we can show that $u$ is a morphism of partial $H$-comodules. Firstly we will verify that $\pi_{E\ot_{A^{coH}}A}\circ (u\ot H)(\ker\pi_T)=0$ (cf.\ \leref{comodulemorphismabelian}). Since $A\bul H$ is flat as a left $A^{coH}$-module and $e$ is an injective map (being an equalizer in a module category), it is equivalent to check that 
\[(e\ot_{A^{coH}}(A\bul H))\circ \pi_{E\ot_{A^{coH}}A}\circ (u\ot H)(\ker\pi_T)=0.\]
By the functoriality of $-\ot_{A^{coH}}A:\Mod_{A^{coH}}\to \PMod_A^H$, we obtain that $e\ot_{A^{coH}}A$ is a morphism of partial comodules and $(e\ot_{A^{coH}}(A\bul H))=(e\ot_{A^{coH}}A)\bul H$). This allows us to rewrite the left hand side of the last equality as
\[\pi_{N\ot_{A^{coH}}A}\circ (e\ot_{A^{coH}}A\ot H)\circ (u\ot H)(\ker\pi_T)
= \pi_{N\ot_{A^{coH}}A}\circ (t\ot H)(\ker\pi_T)=0\]
where the second equality is the defining property of $u$ and the last equality follows from the fact that $t$ is a morphism of partial comodules. Hence the map $u\bul H:T\bul H\to E\ot_{A^{coH}}(A\bul H)$ is well defined and the unique map satisfying $(u\bul H)\circ \pi_T=\pi_{E\ot_{A^{coH}}A}\circ (u\ot H)$. Then we find
\begin{eqnarray*}
((e\ot_{A^{coH}}A)\bul H) \circ(u\bul H)\circ \pi_T&=& ((e\ot_{A^{coH}}A)\bul H) \circ \pi_{E\ot_{A^{coH}}A}\circ (u\ot H)\\
&=& \pi_{N\ot_{A^{coH}}A}\circ  (e\ot_{A^{coH}}A)\ot H \circ (u\ot H)\\
&=& \pi_{N\ot_{A^{coH}}A}\circ (t\ot H) = (t\bul H)\circ \pi_T
\end{eqnarray*}
where we used the fact that $e\ot_{A^{coH}}A$ is a morphism of partial comodules in the second equality, $t= e\ot_{A^{coH}}A\circ u$ in the third equality and the fact that $t$ is a morphism of partial comodules in the last equality. By the surjectivity of $\pi_T$, it follows now that
$(t\bul H)=(e\ot_{A^{coH}}A\bul H)\circ (u\bul H)$.
For $u$ to be a partial comodule morphism, it remains to check that $(u\bul H)\circ \rho_T=\rho_{E\ot_{A^{coH}}A}\circ u$. Again, using the flatness of $A\bul H$ a left $A$-module and the injectivity of $e$ it is sufficient to check that the compositions of the se maps with $e\ot_{A^{coH}}(A\bul H)$ are equal. Using the fact that $t$ and $e$ are partial module morphisms, we can indeed prove that
\begin{eqnarray*}
(e\ot_{A^{coH}}(A\bul H))\circ (u\bul H)\circ \rho_T 
&=& (t\bul H)\circ \rho_T =\rho_{N\ot_{A^{coH}}A}\circ t\\
&=& \rho_{N\ot_{A^{coH}}}\circ e\ot_{A^{coH}}A\circ u\\
&=& (e\ot_{A^{coH}}(A\bul H))\circ \rho_{E\ot_{A^{coH}}A}\circ u
\end{eqnarray*}
Hence $u$ lives already in $\PMod^H_A$ and therefore $(E\ot_{A^{coH}}A,e\ot_{A^{coH}}A)$ satisfies the universal property of the equalizer in $\PMod^H_A$.\\
\ul{(ii)}. Let $f:M\to N$ be a morphism in $\Mod_{A^{coH}}$ such that $f\ot_{A^{coH}}A$ is an isomorphism in $\PMod_A^H$. Then $f\ot_{A^{coH}}A$ is also an isomorphism in $\Mod_A$, and since $A$ is faithfully flat as left $A^{coH}$-module, we find that $f$ is an isomorphism in $\Mod_{A^{coH}}$.\\
\ul{(iii)}. This follows immediately from the previous two parts by the dual of Beck's monadicity theorem, see e.g.\ \cite{tt} or \cite[Theorem 2.7]{GT:comonad}.
\end{proof}

\begin{remark}
The proof of part (i) in the previous theorem can be adapted to show that the functor $-\ot_{A^{coH}}A:\Mod_{A^{coH}}\to \PMod_A^H$ preserves arbitrary limits.
\end{remark}

\subsection*{Acknowledgements}  
The first author wants to thank CSC (China Scholarship Council) for a PhD-student fellowship. He also thanks the ``Fonds David et Alice Van Buuren'' for the prize that allowed him to finish his PhD in Brussels. The second author wants to thank the ``F\'ed\'eration Wallonie-Bruxelles'' for the ARC grant ``Hopf algebras and the Symmetries of Non-commutative Spaces'' that supported this work.

The second author wants to thank Ana Agore for a discussion concerning the completeness of comodule categories and Mitchell Buckley for useful discussions about categorical aspects of spans, bicategories and pushouts. Both authors thank Eliezer Batista and Paolo Saracco for interesting discussions on partial actions and coactions. 

We thank the anonymous referee for this careful reading of the paper and the many useful suggestions that improved the presentation of this paper, as well for pointing out some inaccuracies.

\end{document}